\documentclass[reqno]{amsart}
\usepackage[english]{babel}
\usepackage{amsmath,enumerate,graphicx}
\usepackage[misc,geometry]{ifsym}
\usepackage{amsthm}
\usepackage{amssymb}
\usepackage{amscd}
\usepackage[all]{xy}
\usepackage{xcolor}
\usepackage{mathtools}
\usepackage[latin1]{inputenc}
\usepackage{hyperref}

\usepackage{cite}
\DeclareMathOperator{\id}{id}

\newenvironment{smallarray}[1]
 {\null\,\vcenter\bgroup\scriptsize
  \arraycolsep=.13885em
  \hbox\bgroup$\array{@{}#1@{}}}
 {\endarray$\egroup\egroup\,\null}

\renewcommand{\textbf}[1]{\text{\fontseries{b}\selectfont{\upshape #1}}}

\renewcommand\hom{\mathrm{Hom}}

\def\Ch{\mathrm{Ch}}

\def\Der{\mathbf {D}}

\def\coker{\mathrm{CoKer}}

\def\t{\mathbf{t}}

\def\ker{\mathrm{Ker}}
\def\Ker{\mathrm{Ker}}
\def\Coker{\mathrm{Coker}}

\def\Ab{\mathrm{Ab}}

\def\Hom{\mathrm{Hom}}
\def\Ext{\mathrm{Ext}}

\def\Sub{\mathrm{Sub}}
\def\Gen{\mathrm{Gen}}
\def\Im{\mathrm{Im}}
\def\End{\mathrm{End}}

\def\pd{\mathrm{pd}}
\def\Proj#1{\mathrm{Proj}\text{-}#1}

\def\copres{\mathrm{copres}}
\def\cogen{\mathrm{cogen}}
\def\pres{\mathrm{pres}}

\def\N{\mathbb{N}}
\def\X{\mathcal{X}}
\def\Y{\mathcal{Y}}

\def\Tcal{\mathcal{T}}
\def\X{\mathcal X}
\def\Fcal{\mathcal{F}}

\def\T{\mathcal T}
\def\class#1{\mathcal{#1}}

	\def\A{\class{A}}
	\def\B{\class{B}}
	\def\K{\class{K}}
	\def\C{\class{C}}

	\def\Q{\mathcal Q}
	\def\G{\mathcal G}
\renewcommand\t{\mathbf{t}}
\def\fp{\mathrm{fp}}
\def\add{\mathrm{add}}
\def\Add{\mathrm{Add}}

\def\Gen{\mathrm{Gen}}
\def\subGen{\overline{\mathrm{Gen}}}
\def\Cogen{\mathrm{Cogen}}
\def\subCogen{\underline{\mathrm{Cogen}}}
\def\gen{\mathrm{gen}}
\def\Pres{\mathrm{Pres}}
\def\Copres{\mathrm{Copres}}
\def\Cogen{\mathrm{Cogen}}

\def\Ht{\mathcal{H}_{\t}}

\def\Quot{\mathrm{Quot}}
\def\tr{\mathrm{tr}}
\def\op{{\mathrm{op}}}
\def\lmod#1{#1\text{-}\mathrm{Mod}}
\def\lfpmod#1{#1\text{-}\mathrm{mod}}

\def\mod#1{\mathrm{Mod}\text{-}#1}
\def\fpmod#1{\mathrm{mod}\text{-}#1}
\def\proj#1{\mathrm{proj}\text{-}#1}
\newcommand{\modl}{\text{-}\mathrm{Mod}}

\theoremstyle{plain}
	\newtheorem{theorem}{Theorem}[section]
	\newtheorem{lemma}[theorem]{Lemma}
	
	\newtheorem{proposition}[theorem]{Proposition}
	\newtheorem{corollary}[theorem]{Corollary}
	\newtheorem{example}[theorem]{Example}
	
	\newtheorem{remark}[theorem]{Remark}
	\newtheorem{question}[theorem]{Question}
\theoremstyle{definition}
	\newtheorem{definition}[theorem]{Definition}
	\newtheorem{notation}[theorem]{Notation}

\begin{document}

\title[Tilting and cotilting in Abelian categories]{Tilting preenvelopes and cotilting precovers\\ in general Abelian categories}

\author{Carlos E. Parra}
\address{Instituto de Ciencias F\'\i sicas y Matem\'aticas, Edificio Emilio Pugin, Campus Isla Teja, Universidad Austral de Chile, 5090000 Valdivia, CHILE}
\email{carlos.parra@uach.cl}
\thanks{The first named author was supported by ANID+FONDECYT/REGULAR+1200090}

\author{Manuel Saor\'{\i}n} 
\address{Departamento de Matem\'{a}ticas,
Universidad de Murcia,  Aptdo.\,4021,
30100 Espinardo, Murcia,
SPAIN}
\email{msaorinc@um.es}
\thanks{The second named author was supported by the research projects from
the Ministerio de Econom\'{\i}a y Competitividad of Spain (MTM2016-77445-P) and the Fundaci\'on `S\'eneca' of
Murcia (19880/GERM/15), both with a part of FEDER funds.}

\author{Simone Virili$^{\text{\,(\Letter})}$}
\address{Dipartimento di Scienze Matematiche, Informatiche e Fisiche, Universit\`a degli studi di Udine, Via delle Scienze 206, 33100 Udine, ITALY}
\email[Corresponding author]{virili.simone@gmail.com}
\thanks{The third named author was supported by the
Fundaci\'on `S\'eneca' of Murcia (19880/GERM/15) with a part of FEDER funds.}

\begin{abstract}
{We consider an arbitrary Abelian category $\A$ and a subcategory $\T$ closed under extensions and direct summands, and  characterize those $\T$ that are (semi-)special preenveloping in $\A$; as a byproduct, we generalize to this  setting several classical results for categories of modules. For instance, we get that the special preenveloping subcategories $\T$ of $\A$ closed under extensions and direct summands are precisely those for which $(_{}^{\perp_1}\T,\T)$ is a right complete cotorsion pair, where  \mbox{$_{}^{\perp_1}\T:=\ker (\Ext_\A^1(-,\T))$}.   Particular cases  appear when $\T=V^{\perp_1}:=\Ker (\Ext_\A^1(V,-))$, for an $\Ext^1$-universal object $V$  such that $\Ext_\A^1(V,-)$ vanishes on all (existing) coproducts of copies of $V$. For many choices of $\A$, we show that these latter examples exhaust all the possibilities. }

{We then show that, when $\A$ has an epi-generator, the (semi-)special preenveloping torsion classes $\T$ given by (quasi-)tilting objects are exactly those for which any object $T\in\T$ is the epimorphic image of some object in $_{}^{\perp_1}\T$ (and the subcategory $\mathcal{B}:=\text{Sub}(\T)$ of subobjects of objects in $\T$ is reflective) and they are, in turn, the right constituents of complete cotorsion pairs in $\A$ (resp., $\mathcal{B}$).  } 

{In a final section, we apply the results  when $\A=\fpmod R$ is the category of finitely presented modules over a right coherent ring $R$, something that gives new results and raises new questions even at the level of classical tilting theory in categories of modules.}
\end{abstract}


\maketitle

\setcounter{tocdepth}{1}
\tableofcontents

\section*{Introduction}

{Co/tilting modules} arose in the context of finitely generated modules over finite dimensional algebras, as an attempt to extend to notion of Morita bimodules (see \cite{BB80,happel1982tilted,Bo81}). In this context, equivalences of full module categories were replaced by suitable counter-equivalences of torsion pairs. These results were successively extended to categories of modules over arbitrary rings, but still considering tilting modules with strong finiteness conditions (see \cite{CF90,miyashita-tilt}). 
Few years after their introduction, it was noted that co/tilting modules and complexes could be used to construct equivalences  between derived categories (see \cite{CPS86,H87,happel1988triangulated,rickard1991derived,kel94-dg}).
The last step in this generalization process was that of introducing large (i.e., not {finitely generated}) counterparts of co/tilting modules and complexes (see \cite{hugel2001infinitely,CGM,colpi1995tilting,vstovivcek2014derived,hugel2015silting,wei2013semi}). 

\smallskip
In parallel to co/tilting theory, the idea of left and right approximations of modules originated in the realm of finitely generated modules over Artin algebras, with the work of Auslander, Smal\o\ and Reiten \cite{auslander1980preprojective,auslander1991applications}. On the other hand, the corresponding general notions (for arbitrary modules over general rings) were independently discovered by Enochs \cite{enochs1981injective} and named (pre)envelopes and (pre)covers, see also \cite{golan1973torsion,teply1976torsion}. This general theory is very useful to give a uniform treatment of, e.g., injective and pure injective hulls, or projective, injective and flat covers.

\smallskip
Our starting point in this note is the observation that, when we restrict our attention to categories of modules,  known characterizations relate \mbox{(quasi-)}tilting torsion pairs with (\mbox{semi-})special preenveloping torsion classes (see \cite{hugel2001tilting}) {and these with cotorsion pairs}. Dual results do hold for \mbox{(quasi-)}cotilting modules. Our main goal in this paper is {to understand this relation in the context of}  arbitrary Abelian categories.
 In this framework we study, {for an Abelian category $\A$ and a subcategory $\T$ closed under extensions and direct summands,}
  the relationships between the following properties (and their respective duals):
\begin{itemize} 
\item {$\T$ is the torsion class given by a (quasi-)tilting object};
\item $\mathcal T$ is  (semi-)special preenveloping;
\item {$\T$ is the right constituent of a cotorsion pair.} 
\end{itemize}
Our main results show that {the special preenveloping subcategories $\T$ closed under extensions and direct summands are  the right constituents of right complete cotorsion pairs and, when $\A$ has an epi-generator, the special preenveloping torsion classes are  of the form $\T=V^{\perp_1}:=\Ker (\Ext_\A^1(V,-))$ for a suitable object $V\in\T$ (see Theorem \ref{thm.specialpreenveloping-versus-cotorsionpair} and Corollary \ref{cor.special-preenveloping-arbitraryA}). Among them, the}
tilting torsion classes are precisely the special preenveloping torsion classes $\T$ such that $_{}^{\perp_1}\T$ generates $\T$ {(see Corollary \ref{cor.tilting-bijection}),} while quasi-tilting torsion classes $\T$ such that $\text{Sub}(\T)$ is a reflective subcategory of $\A$  are precisely the semi-special preenveloping torsion classes  $\T$ such that $_{}^{\perp_1}\T$ generates $\T$ {(see Proposition \ref{prop.lorthogonal-generates-torsion})}.

Let us note that, unlike in the context of module categories, our proofs do not rely on the existence of a projective generator nor on the presence of an injective cogenerator. Furthermore, we do not assume exactness of (co)products and this forces us to develop  new tools to control the exactness of certain coproducts (see Section~\ref{exact_coprod_sec}) and a general theory of universal extensions (see Section~\ref{unversal_ext_sec}). With these results we are able to give new proofs for the  correspondences {mentioned above,} that hold in our completely general setting.  

\medskip
The organization of the paper goes as follows: 

In Sections \ref{setting_sec} and  \ref{sec. background},  we introduce the needed concepts and tools in general Abelian categories, with special attention to those Abelian categories that have an epi-generator, a class that includes all (Ab.3) and $\Hom$-finite Abelian categories with a generator. Furthermore, we discuss $\hom$- and $\Ext^1$-finite Abelian categories. 

In Section \ref{sub_cogen_sec} we study (co)reflective subcategories in an Abelian category, with special emphasis on the subcategories of the form $\Sub(\X)$ and $\Quot (\X)$, consisting of the objects which are isomorphic, respectively,  to subobjects and quotients of objects in $\X$, for a subcategory  $\X\subseteq\A$.

In Section \ref{sec. semi-special preenveloping} we study universal extensions and introduce the notion of $\Ext^1$-universal object. 
These concepts are fundamental for the main result of the paper, Theorem \ref{thm.specialpreenveloping-versus-cotorsionpair},  and its Corollary~\ref{cor.special-preenveloping-arbitraryA}, from which most of the subsequent results are a consequence. %That theorem says that, if an Abelian category $\A$ has an epi-generator, then the semi-special preenveloping torsion classes $\T$ are those for which $\B:=\Sub(\T)$ is a reflective subcategory and such that there is an object $V\in\T$, which is $\Ext^1$-universal in $\B$, and $\B\cap V^{\perp_1}=\T$. In particular (see Corollary \ref{cor.special-preenveloping-arbitraryA}), the special preenveloping torsion classes $\T$ are the cogenerating torsion classes of the form $\T=V^{\perp_1}$, for some $\Ext^1$-universal object $V$ of $\A$.
 Let us remark that, to construct an effective theory of universal extensions in the case of ``big'' (i.e., cocomplete) Abelian categories, we need to have some control on the exactness of infinite coproducts. For this, we introduce and study a ``formal derived functor'' of the coproduct in Subsection~\ref{exact_coprod_sec}.

In Section \ref{sec_tilt_pree} we define (quasi-)tilting objects in {arbitrary} Abelian categories and show that tilting objects are   $\Ext^1$-universal (see  Proposition \ref{prop.Ext-universal tilting}). {We show that our general concept of (quasi-)tilting object coincides with the classical ones appearing in the literature in module categories and in Hom- and $\Ext^1$-finite Abelian categories}.  When the ambient Abelian category $\A$ has an epi-generator, Proposition \ref{prop.lorthogonal-generates-torsion} and Corollary {\ref{cor.tilting-bijection}} {identify the (semi-)special preenveloping torsion classes which are given by (quasi-)tilting objects,}
 thus extending known results in module categories.  In the final Subsection \ref{sec. duals} we just state the duals of Theorem \ref{thm.specialpreenveloping-versus-cotorsionpair} and Proposition \ref{prop.lorthogonal-generates-torsion}, that give as a consequence Proposition \ref{prop.}, a result that extends  \cite[Thm.\,3.5]{breaz2018torsion} and \cite[Thm.\,3.5]{zhang2017cosilting} from module categories to  Grothendieck categories.

{In the final section,} we focus on the case when  $\A=\fpmod R$ is the category of finitely presented (right) modules over a right coherent ring $R$. It turns out that the tilting objects $V$ of $\fpmod R$ are, up to equivalence, the classical (=finitely presented) ($1$-)tilting $R$-modules whose associated torsion pair in $\mod R$ restricts to $\fpmod R$, equivalently, such that $\End(V_R)$ is a right coherent ring (see Corollary \ref{cor.restriction of qt torsion pair}). We do not know if all classical tilting modules over a right coherent ring satisfy this latter property.

\subsection*{Data Availability Statement}
Data sharing not applicable to this article as no datasets were generated or analysed during the current study.

\subsection*{Acknowledgement}
We are grateful to Lidia Angeleri-H\"ugel for several comments, to Birge Huisgen-Zimmermann and Ken Goodearl for helping us with the first part of  Example \ref{ex.ring isomorphic to its product} and to Jan Saroch for pointing us to the relevant literature for that same example. We would also like to thank Jeremy Rickard for showing us, via a counterexample, that ``right coherence'' is not a derived invariant property for rings (see Remark \ref{rem. Rickard}).

\section{Basic definitions and notation}\label{setting_sec}

Throughout the paper, {all subcategories are implicitly assumed  to be full}.  Furthermore, we fix the following conventions:
\begin{itemize}
\item the symbol $\mathcal A$ denotes an Abelian category; 
\item given two objects $A$ and $B$ in $\A$, we let $\A(A,B):=\Hom_\A(A,B)$;
\item $\Proj \A$ denotes the subcategory of projective objects in $\A$;
\item the symbol $R$ always denotes a (unitary and associative) ring;
\item $\mod R$ and $\lmod R$ (resp., $\fpmod R$ and $\lfpmod R$) denote the categories of (finitely presented) right and left $R$-modules, respectively;
\item $\Proj R$ (resp., $\proj R$) denotes the category of (finitely presented) projective right $R$-modules;
\item given ring $R$ and two modules $M$ and $N\in\mod R$, we let $\hom_R(M,N):=\hom_{\mod R}(M,N)$ and $\End_R(M):=\End_{\mod R}(M)$, and we use similar notations for left modules.
\end{itemize}
Given a class of objects $\mathcal X$ in an Abelian category $\mathcal A$, we let
\begin{itemize}
\item $\mathcal X^\perp:=\{A\in \mathcal A:  {\mathcal A}(X,A)=0,\ \forall X\in \mathcal X\}$ and ${}^{\perp}\mathcal X:=\{A\in \mathcal A: {\mathcal A}(A,X)=0,\ \forall X\in \mathcal X\}$;
\item $\Add_\A(\X)$ (resp., $\add_\A(\X)$) denotes the class of direct summands of coproducts, taken in $\A$, of  (finite) families of objects taken from $\X$;
\item $\Gen_\A(\X)$ (resp., $\gen_\A(\X)$) is the class of quotients of objects in $\Add_\A(\X)$ ($\add_\A(\X)$);
\item $\Pres_\A(\X)$ (resp., $\pres_\A(\X)$) is the class of cokernels of morphisms  in $\Add_\A(\X)$ ($\add_\A(\X)$).
\end{itemize}
When the category $\A$ is clear from the context, we generally omit it as a subscript. Finally,  $\Cogen(\X)$, $\Copres(\X)$, $\cogen(\X)$ and $\copres(\X)$ are defined dually. 

\begin{remark} \label{rem.about (co)products}
Unless otherwise stated, in this paper Abelian categories are not assumed to have (all small) (co)products. So the (co)products that appear in the definitions of, say, $\Pres(\X)$ and $\Copres(\X)$ are just those that exist in $\A$. This same convention applies throughout the paper. 
\end{remark}

\subsection{Abelian exact subcategories}
A subcategory $\mathcal{B}$ of an Abelian category $\A$ is said to be an
 {\bf Abelian exact subcategory} of $\A$ when it is Abelian and the inclusion functor 
$\mathcal{B}\to \A$ is exact. Equivalently,  $\mathcal{B}$ is closed under taking finite co/products, kernels and cokernels in $\A$.
\begin{definition} \label{def.quotX and subX}
Given a subcategory $\mathcal{X}$ of an Abelian category $\A$, define:
\begin{itemize}
\item $\Sub(\mathcal{X})\subseteq \A$, the class of those objects that embed in some object in $\mathcal{X}$;
\item $\Quot(\mathcal{X})\subseteq \A$, the class of those objects that are epimorphic images of some object in $\mathcal{X}$.
\end{itemize}
Then, $\mathcal{X}$ is said to be {\bf generating} (resp., {\bf cogenerating}) if $\A=\Quot(\mathcal{X})$ (resp., $\A=\Sub(\mathcal{X})$).
\end{definition}

The classes of the form $\Sub(\X)$ and $\Quot(\X)$ can be used to construct natural examples of Abelian exact subcategories of $\A$:

\begin{example} \label{ex.Sub(X)-Gen(X) Abelian exact}
Let $\A$ be an Abelian category and $\mathcal X$  a subcategory closed under finite co/products, then $\Quot(\Sub(\mathcal X))=\Sub(\Quot(\mathcal{X}))$ is an Abelian exact subcategory of $\A$. Indeed, given $M\in \Sub(\Quot(\mathcal{X}))$, there is an embedding $M\to Q$ and an epimorphism $X\to Q$, for some $X\in \mathcal X$. Consider then the following pullback diagram:
\[
\xymatrix{
S\ar@{^(.>}[d]\ar@{.>>}[r]\ar@{}[rd]|{\text{\tiny P.B.}}&M\ar@{^(->}[d]\\
X\ar@{->>}[r]&Q
}
\]
that shows that $M\in \Quot(\Sub(\mathcal X))$, hence $\Sub(\Quot(\mathcal{X}))\subseteq \Quot(\Sub(\mathcal X))$. The converse inclusion follows by a dual argument, using a pushout square. Hence, $\Quot(\Sub(\mathcal X))$ is closed under taking quotients, subobjects and it is easily seen to be also closed under taking finite co/products. In particular, we obtain the following statements:
\begin{itemize}
\item if $\mathcal{X}$ is also closed under quotients,  $\Sub(\mathcal{X})$ is an Abelian exact subcategory of $\A$;
\item if $\mathcal{X}$ is also closed under subobjects,  $\Quot(\mathcal{X})$ is an Abelian exact subcategory of $\A$.
\end{itemize}
\end{example}

\subsection{$\Ext$-groups}

For each positive integer $n$, we denote by $\Ext^n_{\mathcal A}(-,-)$ the big group of (equivalence classes of) $n$-fold extensions, with the usual Baer sum, in the  category $\A$ (see \cite{McL}). Given a class of objects $\mathcal X$ in $\mathcal A$, we use the following notations:
\begin{align*}
\mathcal X^{\perp_n}=\{A\in \mathcal A: \Ext^n_{\mathcal A}(X,A)=0,\ \forall X\in \mathcal X\}\quad\text{and}\\
{}^{\perp_{n}}\mathcal X=\{A\in \mathcal A: \Ext^n_{\mathcal A}(A,X)=0,\ \forall X\in \mathcal X\}.
\end{align*}
When $\X=\{X\},$ for simplicity we will write $X^{\perp_n}$ (resp. ${}^{\perp_n}X$) instead of $\X^{\perp_n}$ (resp. ${}^{\perp_n}\X$). Let us start with the following elementary remark that will be often useful hereafter.

\begin{remark} \label{rem.monic-morphism-for-coproducts}
Let $\mathcal{A}$ be an Abelian category, $X\in \A$ and $\{Y_i\}_{I}$ a family of objects whose coproduct exists in $\mathcal{A}$. Then, the following canonical map is injective
\[
\xymatrix{
\phi\colon \coprod_{I}\mathcal{A}(X,Y_i)\longrightarrow\mathcal{A}\left(X,\coprod_{I}Y_i\right).
}
\]
In fact, if $(f_i)_{I}\in \coprod_{I}\mathcal{A}(X,Y_i)$ and we take  $J:=\{i\in I:f_i\neq 0\}$, which is a finite subset of $I$,  then $f:=\phi ((f_i)_{I})$   decomposes as  $\iota_J\circ f_J\colon X\stackrel{}{\rightarrow} \prod_{J}Y_j\cong (\coprod_{J}Y_j) \hookrightarrow \coprod_{I} Y_i $  where $f_J$ is given by the universal property of (finite) products and $\iota_J$  is the canonical section.  If $\phi ((f_i)_{I})=f=0$ then $f_J=0$, which implies that  $J=\emptyset$ and so $(f_i)_{I}=0$.
\end{remark}

A similar ``embedding property'' also holds for $\Ext^1$-groups, as the following lemma shows:

\begin{lemma} \label{lem.Ext-V-versus-Ext-VI}
Let $\A$ be an Abelian category and $\{V_i\}_{I}$ a non-empty family of objects whose coproduct  exists in $\mathcal{A}$. Then, for any $X\in\A$, the following canonical morphism of big Abelian groups is injective:
\[
\xymatrix{
\Phi\colon\Ext^1_\A\left(\coprod_{I}V_i,X\right)\longrightarrow\prod_{I}\Ext^1_\A(V_i,X).
}
\] 
In particular, the following assertions hold true:
\begin{enumerate}[\rm (1)]
\item given $V\in\mathcal{A}$ and a set $J\ne \emptyset$ such that $V^{(J)}$ exists in $\A$, then $V^{\perp_1}=(V^{(J)})^{\perp_1}$;
\item if $\mathcal{X}$ is a subcategory of $\mathcal{A}$, then $_{}^{\perp_1}\mathcal{X}$ is closed under taking coproducts in $\mathcal{A}$.
\end{enumerate}
\end{lemma}
\begin{proof}
Assertion (2) clearly follows from the initial statement of the lemma. As for assertion (1), let us  fix $j\in J$. Then we have a decomposition $V^{(J)}\cong V\oplus V^{(J\setminus\{j\})}$ which, in 
turn, gives a decomposition of big Abelian groups 
\[
\Ext^1_\A(V^{(J)},X)\cong \Ext^1_\A(V,X)\oplus\Ext^1_\A(V^{(J\setminus\{j\})},X).
\] 
This gives
the inclusion ``$\supseteq$'' of the desired equality. The inclusion ``$\subseteq$'' follows directly from the initial statement of the lemma.

To prove the initial statement, let $[\epsilon]$ be an element of the kernel of  $\Phi$, represented by the following short exact sequence  in $\mathcal{A}$:
\begin{equation}\label{ses_lem.Ext-V-versus-Ext-VI}
\xymatrix{
\epsilon:\quad 0\rightarrow X\stackrel{u}{\longrightarrow} Y\stackrel{p}{\longrightarrow} \coprod_{I}V_i\rightarrow 0.
}
\end{equation} 
For each $j\in I$, if $\iota_j\colon V_j\to \coprod_{I}V_i$ is the $j$-th inclusion into the coproduct,  the map $\iota_j^*:=\Ext^1_\mathcal{A}(\iota_j,X)$ decomposes as follows:

\[
\xymatrix{
\Ext^1_\A(\coprod_{I}V_i,X)\stackrel{\Phi}{\longrightarrow}\prod_{I}\Ext^1_\A(V_i,X)\stackrel{\pi_j}{\longrightarrow}\Ext^1_\mathcal{A}(V_j,X),
}
\] 
where $\pi_j$ is the $j$-th projection from the product in $\Ab$. Therefore,  $\iota_j^*([\epsilon])=0$, which implies that $\iota_j$ has the following factorization: 
\[
\xymatrix@C=30pt{
V_j\ar@/_15pt/[rr]_{\iota_j}\ar[r]^{\lambda_j}&Y\ar[r]^{p}&\coprod_{I}V_i,
}
\] 
for some morphism $\lambda_j\colon V_j\to Y$. We obtain a family of morphisms $\{\lambda_j\colon V_j\to Y\}_{I}$ in $\mathcal{A}$ which, by the universal property of coproducts, gives a unique morphism $\lambda\colon \coprod_{I}V_i\to Y$ such that $\lambda\circ\iota_j=\lambda_j$, for all $j\in I$. Thus $p\circ\lambda=\id_{\coprod_{I}V_i}$, showing that the sequence \eqref{ses_lem.Ext-V-versus-Ext-VI} splits, and so $[\epsilon] =0$.
\end{proof}

Recall now that, given $n\in\N$ and an object $X$ in an Abelian category $\A$, we say that the {\bf projective dimension of $X$} is less than or equal to $n$, in symbols $\pd_\A(X)\leq n$, if $X^{\perp_{n+1}}=\A$. 

\begin{corollary}\label{coro_proj_dim}
Let $\A$ be an Abelian category, $V\in\A$, and $I$ a set for which $V^{(I)}$ exists in $\A$. If there is a cogenerating class $\mathcal{X}\subseteq  V^{\perp_1}$ which is closed under taking quotients, then $\pd_\A(V^{(I)})\leq 1$.
\end{corollary}
\begin{proof}
Given $A\in \A$ and using that $\mathcal{X}$ is cogenerating one gets, by a suitable pushout construction, that any element $[\epsilon]\in \Ext_\mathcal{A}^2(V^{(I)},A)$ is represented by an exact sequence of the form:
\[
\epsilon:\quad 0\rightarrow A\longrightarrow X\stackrel{f}{\longrightarrow} Y\longrightarrow V^{(I)}\rightarrow 0,
\] 
where $X$ is in $\mathcal{X}$. But $\Im(f)=:X'$ is also in $\mathcal{X}$ since this class is closed under taking quotients.  It follows that $[\epsilon]\in\Ext_\mathcal{A}^1(X',A)\circ\Ext_\mathcal{A}^1(V^{(I)},X')$, where $\circ$ denotes the Yoneda product. Then, $\Ext_\mathcal{A}^1(V^{(I)},X')=0$ (see Lemma~\ref{lem.Ext-V-versus-Ext-VI}(1)), and so $\Ext_\mathcal{A}^2(V^{(I)},A)=0$.
\end{proof}

\subsection{Complexes, homotopy category and derived category}\label{subs_complexes}

For an Abelian category $\A$ and { $\B$ an additive subcategory of $\A$}, we denote by $\Ch(\A)$, {$\Ch(\B)$}, $\K(\A)$ and {$\K(\B)$}, the category of, cohomologically graded, {\bf cochain complexes}, {and} the {\bf homotopy category} { over $\A$ and $\B$, respectively}. {Moreover, we denote by $\Der(\A)$, the {\bf derived category} of $\A$. We use the suffixes $b$, $+$, \mbox{and $-$} to denote the corresponding subcategories consisting of bounded, bounded below and bounded above complexes. For instance, $\K^+(\A)$ denotes the homotopy category of bounded below complexes. {We also use the symbol} { $\K^{-,b}(\B)$ to denote the subcategory of $\K^{-}(\B)$ whose objects have almost all cohomology groups equal to zero.}

%For an Abelian category $\A$, we denote by $\Ch(\A)$, $\K(\A)$ and $\Der(\A)$, the category of, cohomologically graded, {\bf cochain complexes}, the {\bf homotopy category} and the {\bf derived category}, respectively. Furthermore, we use the suffixes $b$, $+$, and $-$ to denote the corresponding subcategories consisting of bounded, bounded below and bounded above complexes.  For instance, $\K^+(\A)$ denotes the homotopy category of bounded below complexes.

Recall that the bounded above derived category $\Der^-(\mathcal{A})$ is well-defined whenever $\mathcal{A}$ has enough projectives and, in fact, the canonical composition 
\[
\mathcal{K}^-(\text{Proj-}\mathcal{A})\hookrightarrow\mathcal{K}^{-}(\mathcal{A})\stackrel{q}{\to}\Der^-(\mathcal{A})
\] 
is a triangulated equivalence (see \cite{verdier1996categories}). Dual considerations hold for the bounded below derived category of an Abelian category with enough injectives.  %\textcolor{black}{{\bf (Chequear cuales de estas subcategor\'ias son realmente usadas y borrar las que no lo sean)}}
Finally, let us recall that, when $\A$ has enough projectives, $I$ is a set such that $I$-coproducts exist in $\mathcal{A}$ and $\{V_i\}_{I}$ is  a family of objects in $\mathcal{A}$, then we can compute the coproduct of the stalk complexes $V_i[0]$ in $\Der^{-}(\A)$ as follows. We start by fixing, for each $j\in I$, a {\bf projective resolution} $s_j\colon P^\bullet_j{\to}V_j[0]$, i.e., a quasi-isomorphism of complexes, where $P^\bullet_j$ is a degree-wise projective complex concentrated in degrees $\leq 0$. Then the coproduct  $\coprod_IV_i[0]$ in $\Der^{-}(\A)$ exists and it is precisely the coproduct of complexes $\coprod_{I}P^\bullet_i$ (computed as in $\Ch(\mathcal{A})$ or in $\mathcal{K}(\mathcal{A})$).

\subsection{Torsion pairs}
A\textbf{ torsion pair} $\t=(\mathcal T,\mathcal F)$ in an Abelian category $\A$ is a pair of subcategories, called the \textbf{torsion} and the \textbf{torsion-free class}, respectively, such that $\mathcal T^{\perp}=\mathcal F$,  $\mathcal T={}^{\perp}\mathcal F$  and such that,   for any $A\in\A$, there is a (necessarily functorial) short exact sequence
\[
0\to T_A\longrightarrow A\longrightarrow F_A\to 0,\quad\text{with $T_A\in \mathcal T$ and $F_A\in\mathcal F$.}
\]  
As a consequence, the class $\mathcal T$ (resp., $\mathcal F$) is closed under taking quotients (resp., subobjects), extensions and all small coproducts (resp., small products) that exist in $\A$. %In case $\mathcal{A}$ is (Ab.3) (resp., (Ab.3*)) and well-powered, the converse is also true, in the sense that any class of objects satisfying these closedness properties is a torsion (resp., torsion-free) class. 
We let 
\[
t\colon \mathcal{A}\longrightarrow\mathcal{T}\quad\text{and}\quad(1:t)\colon \mathcal{A}\longrightarrow\mathcal{F}
\]
be the functors acting on objects by  $A\mapsto T_A$ and $A\mapsto F_A$, respectively. They are, respectively, right and left adjoints to the corresponding inclusions. 

Recall that a torsion pair $\mathbf{t}=(\mathcal{T},\mathcal{F})$ in $\A$ is said to be {\bf hereditary} (resp., {\bf cohereditary}, of {\bf finite type}) provided $\mathcal{T}$ is closed under taking subobjects (resp., $\mathcal F$ is closed under taking quotients, $\mathcal F$ is closed under taking direct limits). Furthermore, $\mathcal{T}$ is said to be a {\bf TTF class} if it is both a torsion and a torsion-free class; in this last case, $({}^{\perp}\mathcal{T},\mathcal{T},\mathcal{F})$ is said to be a {\bf TTF triple}.

\textcolor{black}{If $\A$ is a Grothendieck category (see Subsection~\ref{complete_abelian_subs}) and $\t=(\T,\mathcal F)$ is a hereditary torsion pair, then we denote by $Q_\t\colon \A\to \A/\T$ the {\bf Gabriel quotient} of $\A$ over $\T$ (which is always an exact functor), and we let $S_\t\colon \A/\T\to \A$ be the fully faithful right adjoint of $Q_\t$. If $\A=\mod R$ is the category of right $R$-modules over a ring $R$, then we can  consider the {\bf ring of quotients} $\psi\colon R\to R_\t$ of $R$ with respect to the Gabriel topology associated with $\t$ (see \cite[Chapter IX]{S75}). We get the following diagram:
\begin{equation}\label{loc_tor_pair_diagram}
\xymatrix@C=50pt{
\mod R\ar@<.5ex>[rr]^-{Q_\t}\ar@<-.5ex>[dr]_(.4){\psi^!} &&(\mod R)/\T \ar@<.5ex>[ll]^-{S_\t} \ar@<.5ex>[dl]^(.4){ \psi^!\circ S_\t} \\
&\mod {R_\t}\ar@<-.5ex>[lu]_(.4){\psi_*} \ar@<.5ex>[ur]^(.4){Q_\t\circ\psi_*} 
}
\end{equation}
where $\psi_*:=-\otimes_{R_\t}R_\t$ is the restriction of scalars along $\psi$, and $\psi^!:=\hom_R(R_\t,-)$ its right adjoint. This shows that $Q_\t\circ \psi_*$ is  exact  (as it is a composition of two exact functors) and that $\psi^!\circ S_\t$ is its left adjoint, which is fully faithful (see \cite[pages 199 and 217]{S75}, where this functor is denoted by $j$). Hence, letting $\T':=\ker(Q_\t\circ\psi_*)=(\psi_*)^{-1}(\T)$ and $\t':=(\T',(\T')^{\perp})$, we obtain a hereditary torsion pair in $\mod {R_\t}$, an equivalence of categories $(\mod R)/\T\cong (\mod{R_\t})/\T'$ and, identifying these two categories, we also get two natural isomorphisms of functors $Q_{\t'}\cong Q_\t\circ \psi_*$ and $S_{\t'}\cong \psi^!\circ S_\t$. Our original torsion pair $\t$ is said to be {\bf perfect} if $\Q_{\t'}$ is an equivalence of categories or, equivalently, if $\T'=0$ (see \cite[Proposition XI.3.4]{S75} for other equivalent conditions).}

\subsection{Cotorsion pairs}
Cotorsion pairs were first introduced by Salce \cite{Salce_cotorsion} in categories of modules.
 A pair $(\mathcal{X},\mathcal{Y})$ of subcategories of an Abelian category $\A$ is said to be a \textbf{cotorsion pair} when $\mathcal{Y}=\mathcal{X}^{\perp_1}$ and $\mathcal{X}={}^{\perp_1}\mathcal{Y}$. A cotorsion pair is called \textbf{right} (resp., {\bf left}) {\bf complete} when, for each object $A\in\mathcal{A}$, there is a short exact sequence $0\rightarrow A\to Y\to X\rightarrow 0$ (resp.,  $0\rightarrow Y\to X\to A\rightarrow 0$), where $X\in\mathcal{X}$ and $Y\in\mathcal{Y}$. Some authors say instead that  $(\mathcal{X},\mathcal{Y})$ has {\bf enough injectives} (resp., {\bf projectives}) when  $(\mathcal{X},\mathcal{Y})$ is right (resp., left) complete (see \cite{hovey}).    A cotorsion pair is \textbf{complete} if it is both left and right complete.

\subsection{Covers and envelopes}
Let $\A$ be an Abelian category and recall that a morphism $\phi\colon X\to Y$ in $\A$ is said to be {\bf left minimal} if each morphism $\psi \colon Y\to Y$ such that $\psi\circ \phi=\phi$ is an isomorphism. The notion of a morphism being {\bf right minimal} is defined dually. 

Let now $\X$ be a subcategory of $\A$. A morphism
$\phi \colon  M \rightarrow X_M$ with $X_M\in \mathcal{X}$ is:
\begin{itemize}
\item an {\bf $\mathcal{X}$-preenvelope} if, for any
$Y\in \mathcal{X}$, the induced map $\A(X_M,Y)\rightarrow \A (M,Y)$ is surjective;
\item a {\bf semi-special $\X$-preenvelope} if it is a $\X$-preenvelope and $\Coker(\phi)\in{}^{\perp_1}\mathcal{X}$;
\item a {\bf special $\X$-preenvelope} if it is a monomorphic semi-special  $\X$-preenvelope;
\item an {\bf $\X$-envelope} if it is a left minimal $\X$-preenvelope.
\end{itemize}
If  $\mathcal{X}$ is such that
every object has an $\mathcal{X}$-preenvelope (resp., semi-special preenvelope, special preenvelope, envelope), then $\mathcal{X}$ is
said to be {\bf preenveloping} (resp., {\bf semi-special preenveloping}, {\bf special preenveloping}, {\bf enveloping}). The dual notions are those of $\mathcal{X}$-{\bf precover}, {\bf semi-special $\X$-precover}, {\bf special $\mathcal{X}$-precover}, {\bf $\mathcal{X}$-cover} and any subcategory $\X$ for which those notions exist for all objects in $\A$ is called {\bf precovering}, {\bf semi-special precovering}, {\bf special precovering}, and {\bf covering}.

(Pre)envelopes and (pre)covers usually take the name of the classes over which they are constructed. Thus, the notions of injective (pre)envelopes, pure-injective (pre)envelopes, flat (pre)covers, etc.\ appear naturally in the categories where the corresponding classes can be defined.

\begin{example}\label{ex_cotorsion_implies_cover}
Let $(\mathcal X,\mathcal Y)$ be a right (resp., left) complete cotorsion pair in an Abelian category $\A$. For each $A\in \mathcal A$, consider an exact sequence  $0\rightarrow A\stackrel{u}{\to} Y\to X\rightarrow 0$ (resp.,  $0\rightarrow Y\to X\stackrel{p}{\to} A\rightarrow 0$) as above. Then, $u$ (resp., $p$) is a special $\mathcal Y$-preenvelope (resp., special $\X$-precover). 
\end{example}

Let us conclude this subsection with a technical but useful lemma:

\begin{lemma}\label{injective_preenvelope}
Let $\A$ be an Abelian category, $\T\subseteq \A$ a subcategory, and $X\in \A$.
\begin{enumerate}[\rm (1)]
\item If $X$ embeds in a product of objects of $\T$, then any $\T$-preenvelope of $X$ is a monomorphism;
\item when $\T$ is closed under taking coproducts, any coproduct of $\T$-preenvelopes, when it exists, is a $\T$-preenvelope;
%When $\T$ is closed under taking coproducts, a coproduct \textcolor{black}{of} a $\T$-preenvelope, when it exists, is a $\T$-preenvelope;
\item suppose that $\Cogen(\T)=\A$, then the following assertions hold:
 \begin{enumerate}[\rm {(3.}1)]
 \item if  $\phi\colon A\to B$ is a morphism in $\A$ such that the map
\begin{equation}\label{cog_mono_eq}
\phi^*\colon \A(B,T)\to \A(A,T)
\end{equation}
is surjective for all $T\in \T$, then $\phi$ is a monomorphism. 
\item any coproduct of $\T$-preenvelopes, when it exists, is a monomorphism. 
\end{enumerate} 
\end{enumerate}
\end{lemma}
\begin{proof}
(1). Let $\phi\colon X\to T$ be a $\T$-preenvelope and $\psi\colon X\to \prod_{I}T_i$  an embedding, for some family $(T_i)_{I}$ in $\T$. Put $\psi_j:=\pi_j\circ\psi$, where $\pi_j\colon \prod_{I}T_i\to T_j$ is the $j$-th projection, for each $j\in I$. Then,  there is a map $\alpha_j\colon T\to T_j$ such that $\psi_j=\alpha_j\circ\phi$, for all $j\in I$. By the universal property of products, we get a map $\alpha \colon T\to \prod_{I}T_i$ such that $\pi_j\circ\alpha =\alpha_j$, for all $j\in I$. Hence, $\alpha\circ\phi =\psi$, since $\pi_j\circ\alpha\circ\phi=\alpha_j\circ\phi=\psi_j=\pi_j\circ\psi$, for all $j\in I$. Therefore,  $\phi$ is a monomorphism as  so is $\psi$ . 

\smallskip\noindent
(2). Let $(f_i\colon X_i \to T_i)_{I}$ be a family of $\T$-preenvelopes, such that $\coprod_{I}X_i$ and $\coprod_{I}T_i$ exist in $\A$. For each $i$ in $I$, let $\epsilon_{X_i}\colon X_i\to \coprod_{I}X_i$ and $\epsilon_{T_i}\colon T_i\to \coprod_{I}T_i$ be the inclusions in the respective coproducts. Given a morphism $\phi\colon \coprod_{I}X_i\to T'$ with $T'\in\T$, there exists, for each $i\in I$, a morphism $\psi_i\colon T_i\to T'$ such that $\phi\circ\epsilon_{X_i}=\psi_i \circ f_i$. Hence, there is a unique map $\psi\colon \coprod_{I}T_i\to T'$ such that $\psi\circ \epsilon_{T_i}=\psi_i$ for all $i\in I$, and so $\psi\circ (\coprod_{I}f_i)=\phi$. Assertion (2) then  follows, as  $\coprod_{I}T_i\in \T$.

%(2) \textcolor{black}{Let $u:X \to T$ be a $\T$-preenvelope, and f}or each $i$ in $I$, let $\epsilon_{X,i}\colon X\to X^{(I)}$ and $\epsilon_{T,i}\colon T\to T^{(I)}$ be the inclusion in the coproduct. Now, given a morphism $\phi\colon X^{(I)}\to T'$ with $T'\in\T$, there exists, for each $i\in I$, a morphism $\psi_i\colon T\to T'$ such that $\phi\circ\epsilon_{X,i}=\psi_i \circ u$. Hence, there is a unique map $\psi\colon T^{(I)}\to T'$ such that $\psi\circ \epsilon_{T,i}=\psi_i$ for all $i\in I$, and so $\psi\circ u^{(I)}=\phi$. 

\smallskip\noindent
(3).  We just need to prove (3.1), for then (3.2) is clearly a consequence of assertions (1) and (3.1).  Fix a monomorphism $\mu \colon A\hookrightarrow \prod_{I}T_i$ for some family $(T_i)_{I}$ in $\T$. Due to the hypothesis on $\phi$, we then have that the morphism 
\[
\xymatrix{
\A(\phi,\prod_{I}T_i)\colon \A(B,\prod_{I}T_i)\cong\prod_{I}\A(B,T_i)\longrightarrow\prod_{I}\A(A,T_i)\cong\A(A,\prod_{I}T_i)
}
\]
 is surjective since products are exact in $\Ab$. We  get a morphism $\widehat {\mu}\colon B\to \prod_{I}T_i$ such that $\widehat {\mu}\circ \phi=\mu$. Then, $\phi$ is a monomorphism since so is $\mu$. 
\end{proof}

\section{Background on Abelian categories}\label{sec. background}

Throughout this paper we try to work in Abelian categories which are as general as possible but, from time to time, a few additional hypotheses will be needed. In this section we introduce some of these hypotheses and  describe the main classes of examples we are interested in.
The hypotheses we introduce will go in two, essentially opposite, directions: on the one hand we will study hypotheses for ``big'' categories (e.g., (co)completeness, exactness of infinite (co)products, etc.) and, on the other hand, we will study some finiteness conditions (e.g., $\hom$- and $\Ext^1$-finiteness, generators, etc.). Let us start with the following definition:

\begin{definition}
An Abelian category $\A$ is said to be 
\begin{itemize}
\item {\bf $\Ext^1$-small} provided $\Ext^1_{\mathcal A}(A,B)$ is a set (as opposed to a proper class) for each pair of objects $A$ and $B$ in $\A$;
%\item {\bf decent} provided $\Ext^n_{\mathcal A}(A,B)$ is a set (as opposed to a proper class) for each pair of objects $A$ and $B$ in $\A$, and for all $n\in\N$; \textcolor{black}{(PREGUNTA:  Es necesario ahora mantener este concepto?. Creo que apenas se usa, si es que se usa.)}
\item {\bf well-powered} if the lattice $\mathcal{L}_{\A}(A):=\{\text{subobjects of $A$}\}$ is a set for all $A\in\A$. 
\end{itemize}
\end{definition}
The notion of well-powered category is classical (see, e.g., \cite{McL}). Note that both notions introduced in the above definition are self-dual. For this reason, when working in categories satisfying these hypotheses, some concepts and statements will have  equivalent dual versions that, unless necessary for our purposes, will not be explicitly given. 
%The notion of well-powered category is classical (see, e.g., \cite{McL}), while decent Abelian categories were introduced in \cite{vstovivcek2011tilting}. Such categories are quite common, in fact, any Abelian category with enough injectives or enough projectives is decent.

%Note that the three notions introduced in the above definition are self-dual, in the sense that $\A$ satisfies one of these properties if and only if the opposite category $\A^{\op}$ satisfies the same condition. For this reason, when working in categories satisfying these hypotheses, some concepts and statements will have  equivalent dual versions that, unless necessary for our purposes, will not be explicitly given. 

\subsection{Categories with a generator}
In this first subsection we study some notions of ``generator'' for an Abelian category. Usually (say, in categories of modules) these notions coincide but, when we work in full generality, some care is needed.

\begin{definition}\label{def. epi-generator}
A set of objects $\G$ in a category $\A$ is said to be a {\bf set of generators} when the functors $\{\A(G,-)\colon \A\to \Ab\}_{\G}$ are jointly faithful. Furthermore, $\G$ is a {\bf set of} ({\bf finite}) {\bf epi-generators} when $\Gen(\G)=\A$ (resp., $\gen(\G)=\A$). An object $G$ is a ({\bf finite}) ({\bf epi-}){\bf generator} if so is the set $\{G\}$.
The concepts of ({\bf finite}, {\bf mono-}){\bf cogenerator} and of {\bf set of }\mbox{({\bf finite}, {\bf mono-})}{\bf cogenerators} are defined dually. %An epi-generator $G$ will be called a {\bf finite epi-generator} when each object of $\A$ is epimorphic image of a finite coproduct of copies of $G$. One defines dually {\bf finite mono-cogenerator}.
\end{definition}

It is not difficult to verify that an epi-generator is also a generator, but the converse might not be  true in general. Anyway, we will see that the two notions coincide in  our main cases of interest (see Lemmas \ref{lem.AB3-are-AB3*} and \ref{hom-finite-generator}).
Note also that, if $G$ is an epi-generator of an Abelian category $\A$, then $\A=\Gen(G)=\Pres(G)$, so the following lemma applies to categories with an epi-generator.

\begin{lemma}\label{existence_of_coproducts_in_pres_lemma}
Let $\A$ be an Abelian category, $G\in \A$ and $I$  an infinite set for which the coproduct $G^{(I)}$ exists in $\A$. The following assertions hold:

\begin{enumerate}[\rm (1)]
\item  if $X\in\Pres (G)$, then the coproduct $X^{(I)}$ exists in $\A$;
\item  if $\Pres_I(G)$ denotes the subcategory of those objects $X$ that fit into an exact sequence of the form $G^{(I)}\to G^{(I)}\to X\to 0$, then any $I$-indexed family $(X_i)_{I}$ in $\Pres_I(G)$ has a coproduct in $\A$ that belongs to $\Pres_I(G)$;
\item  if $G$ is a finite epi-generator of $\A$, then $\A$ has $I$-coproducts.
\end{enumerate}
\end{lemma}
\begin{proof}
(1).  Consider a presentation 
\[
G^{(J)}\overset{p}{\longrightarrow} G^{(K)}\longrightarrow X\to 0
\] 
and let $\Lambda$ be a set with cardinality $|\Lambda|=\max\{|I|,|J|,|K|\}$ so, by construction, the coproduct $G^{(\Lambda)}$ exists in $\A$. Assume, without loss of generality, that $J=\Lambda=K$. Furthermore, note that $|I\times \Lambda|=|\Lambda|$, so the coproducts in the following exact sequence exist in $\A$:
\[
(G^{(\Lambda)})^{(I)}\overset{p^{(I)}}{\longrightarrow}(G^{(\Lambda)})^{(I)}\longrightarrow \coker(p^{(I)})\to 0.
\]
Finally, the same argument used in the proof of \cite[Lem.\,3.1]{parra2020hrs} shows that $\coker(p^{(I)})$ represents a coproduct of $|I|$-many copies of $X$ in $\A$.

\smallskip\noindent
(2). Let $(G^{(I)}\stackrel{f_i}{\longrightarrow}G^{(I)}\stackrel{p_i}{\longrightarrow}X_i\to 0)_{I}$ be an associated family of presentations. Then  we get a morphism $\coprod_{I}f_i:(G^{(I)})^{(I)}\to(G^{(I)})^{(I)}$. This map is well-defined since $(G^{(I)})^{(I)}\cong G^{(I\times I)}\cong G^{(I)}$. By the mentioned argument of  \cite[Lem.\,3.1]{parra2020hrs}, we get that $\coker (\coprod_{I}f_i)$ is a coproduct of the family $(X_i)_{I}$ in $\A$, and obviously it belongs to $\Pres_I(G)$.
 
 \smallskip\noindent
(3). It is a consequence of the fact that,  when $G$ is a finite epi-generator, $\Pres_I(G)=\A$.
\end{proof}

When an Abelian category has an epi-generator, then there is the following useful characterization of some preenveloping classes:

\begin{lemma} \label{lem.special-preenvelope-generator}
Let $\A$ be an Abelian category with an epi-generator $G$ and  $\T$  a subcategory closed under taking coproducts and quotients (e.g., a torsion class) in  $\A$. Take $A\in\A$ and suppose that  $\mu\colon G\to T_G$ is a (semi-special) $\T$-preenvelope. Fix also an epimorphism $\pi\colon G^{(I)}\twoheadrightarrow A$, for some set $I$, and consider the following pushout diagram (see Lemma \ref{existence_of_coproducts_in_pres_lemma}):
\begin{equation}\label{PO_diagram_lem.special-preenvelope-generator}
\xymatrix@C=35pt{
G^{(I)}\ar@{}[dr]|{\text{P.O.}}\ar[d]_\pi\ar[r]^{\mu^{(I)}}&T^{(I)}_G\ar[r]\ar[d]&\Coker(\mu)^{(I)}\ar[r]\ar@{=}[d]&0\\
A\ar[r]_{\mu_A}&T_A\ar[r]&\Coker(\mu)^{(I)}\ar[r]&0
}
\end{equation}
Then $\mu_A$ is a (semi-special) {$\T$-}preenvelope, so the following assertions are equivalent:
\begin{enumerate}[\rm (1)]
\item $G$ has a (semi-special) $\T$-preenvelope;
\item $\T$ is (semi-special) preenveloping in $\A$.
\end{enumerate}
Moreover, $\T$ is special preenveloping if, and only if, it is  cogenerating  (see Definition~\ref{def.quotX and subX}) and $G$ has a special $\T$-preenvelope. 
\end{lemma}
\begin{proof}
{Note that $T_G^{(I)}$ exists by Lemma \ref{existence_of_coproducts_in_pres_lemma}}. The fact that $\mu_A$ is a (semi-special) $\T$-preenvelope follows from the fact that $\mu^{(I)}$ is a (semi-special) $\T$-preenvelope by Lemma~\ref{injective_preenvelope}(2) (and Lemma~\ref{lem.Ext-V-versus-Ext-VI}(2)), and by the universal property of pushouts. 
Finally, the special preenveloping case now follows by Lemma~\ref{injective_preenvelope}(3). 
\end{proof}

Let us conclude with the following lemma, that shows that having a finite epi-generator which is furthermore a projective object is quite a strong requirement:

\begin{lemma} \label{lem.only-finite-copr-project-epigener}
An Abelian category $\A$ has a projective finite epi-generator if, and only if, it is equivalent to $\fpmod R$ for some right coherent ring $R$. 
\end{lemma}
\begin{proof}
Let $P$ be a projective {finite}  epi-generator of $\A$, then the category $\proj\A$ coincides with $\add(P)$. Consider now the restricted Yoneda functor
\[
\A\longrightarrow [\add(P)^{\op},\Ab]\quad\text{such that}\quad A\mapsto \A(-,A)_{\restriction\add(P)}.
\]
This functor induces an equivalence of categories from $\A$ to the category $\fpmod{\add(P)}$ of finitely presented additive functors $\add(P)^{\op}\to\Ab$, that is, the functors $F$ that admit an exact sequence 
\[
\xymatrix{
\add(P)(-,P_1)\ar[r]&\add(P)(-,P_0)\ar[r]& F\ar[r]& 0,
}
\]
with $P_1,\, P_0\in\add(P)$ (see \cite[Prop.\,2.2]{saorin2020t}). The fact that  $\fpmod{\add(P)}$ is Abelian is equivalent to say that the (skeletally small) additive category $\add(P)$ has weak kernels (see \cite[Lem.\,2.1]{saorin2020t}).  Bearing in mind that $\A(P,-)$ induces an equivalence of categories $\add(P)\cong\add(R_R)=\proj R$  for the ring $R:=\End_\A(P)$, it follows that $\proj R$ has weak kernels, which is equivalent to say that $R$ is right coherent. Since the categories of projectives in $\A$ and in the Abelian category $\fpmod R$ are equivalent it follows, again by \cite[Prop\,2.2]{saorin2020t}, that $\A$ and $\fpmod R$ are equivalent.
\end{proof}

\textcolor{black}{
A strategy to construct non-trivial examples of finite epi-generators is the following:
\begin{example}
Let $R$ be a right coherent ring, $\t=(\T,\mathcal F)$ a hereditary torsion pair of finite type in $\mod R$ and $Q_\t\colon\mod R\to\G:=(\mod R)/\T$ the corresponding Gabriel quotient. Then, the following statements hold true:
\begin{enumerate}[\rm (1)]
\item the category $\A:=\fp(\G)$ of finitely presented objects in $\G$ is Abelian (that is, $\G$ is locally coherent);
\item $Q_\t(R)$ is a finite epi-generator in $\A$;
\item $Q_\t(R)$ is projective in $\A$ if, and only if, $\t$ is perfect.
\end{enumerate}
\end{example}
\begin{proof}
Part (1) follows by \cite[Proposition A.6]{Henning_spectrum_module}, while (2) follows by \cite[Proposition 2.10]{saorin2020t}. For (3), consider the notation of diagram \eqref{loc_tor_pair_diagram}. If $\t$ is perfect, then $S_\t$ is an exact functor, in fact, it 
has both a left and a right adjoint (see \cite[Proposition XI.3.4]{S75}). Therefore, its left adjoint $Q_\t$ sends projectives to projectives, showing that $Q_\t(R)$ is projective. On the other hand, suppose that $Q_\t(R)$ is projective and let us verify that $\T'=0$. Indeed, consider $T\in \T'$ and choose a free presentation of $T$ in $\mod {R_\t}$:
\[
\xymatrix{
R_\t^{(I)}\ar[r]^-{\phi}& R_\t^{(J)}\ar[r] &T\ar[r]& 0.
}
\] 
Since $Q_{\t'}$ is exact and $Q_{\t'}(T)=0$, we have that $Q_{\t'}(\phi)\colon Q_{\t'}(R_\t)^{(I)}\to Q_{\t'}(R_\t)^{(J)}$ is an epimorphism and, since $Q_{\t'}(R_\t)\cong Q_\t(R)$ is projective, $Q_{\t'}(\phi)$ is a split epimorphism. Hence, $\phi\cong (S_{\t'}\circ Q_{\t'})(\phi)$ is a (split) epimorphism, showing that $T=\coker(\phi)=0$.
\end{proof}
}

\subsection{(Co)Complete Abelian categories}\label{complete_abelian_subs}

Recall that an Abelian category $\A$ is said to be {\bf (Ab.3)}, {\bf (Ab.4)}, {\bf (Ab.5)}  (and dually, {\bf(Ab.3*)}, {\bf (Ab.4*)}, {\bf (Ab.5*)}) if, respectively, it is cocomplete, it {is (Ab.3) and it} has exact coproducts, and it {is (Ab.3) and} it has exact directed colimits (it is complete, it {is (Ab.3*) and} it has exact products, and it {is (Ab.3*) and} it has exact inverse limits). Let us remind the reader of the following result in category theory.

\begin{lemma} \label{lem.AB3-are-AB3*}
Let $\A$ be an (Ab.3) Abelian category. Then,
\begin{enumerate}[\rm (1)]
\item an object $G$ in $\A$ is a generator if and only if it is an epi-generator;
\item if $\A$ has a generator, then it is well-powered and (Ab.3*).
\end{enumerate}
\end{lemma}
\begin{proof}
(1). Given a generator $G$ and an object $X\in \A$, the natural homomorphism 
\[
G^{(\A(G,X))}\longrightarrow X
\] 
is an epimorphism. Hence, $G$ is also an epi-generator.

\smallskip\noindent
(2). By \cite[Prop.\,IV.6.6]{S75}, we know that $\A$ is well-powered.  On the other hand, given a set $I$, viewed as a small discrete category, consider the constant diagram functor $\kappa_I\colon\A\to\A^I$. Applying a version of Freyd's Special Adjoint Functor Theorem, e.g.\ the dual of \cite[Thm.\,3.3.4]{borceux1994handbook1}, one can see that $\kappa_I$ has a right adjoint, which is precisely the $I$-product functor $\prod_I \colon\A^I\to\A$. 
\end{proof}

An Abelian category is said to be {\bf Grothendieck} if it is (Ab.5) and it has a generator or, equivalently, a set of generators. Note that, by Lemma~\ref{lem.AB3-are-AB3*}, any generator in a Grothendieck category $\A$ is also an epi-generator and {then }$\A$ is bicomplete {and well-powered}. Furthermore, it is known that Grothendieck categories have enough injectives. %In particular, \textcolor{black}{they} are decent. 

\subsection{$\Hom$- and $\Ext^1$-finite categories}

There are several classes of natural examples of Grothendieck categories, e.g.\ categories of modules (over unitary rings, or even small preadditive categories) or of quasi-coherent sheaves over a scheme. These are, in general, ``big'' categories. On the other hand, in this paper we are also interested in other classes of ``smaller'' Abelian categories like, e.g., those of finitely generated modules over Artin algebras. The finiteness properties introduced in this subsection are designed to mimic the properties of this kind of ``smaller'' examples.

\begin{definition}\label{def_hom_ext_fin}
An Abelian category $\A$ is said to be {\bf $\Hom$-finite} (resp., {\bf $\Ext^1$-finite}) when there is a commutative ring $R$ such that $\A$ is $R$-linear and $\A(A,B)$ (resp., $\Ext^1_\A(A,B)$) is finitely generated as an $R$-module, for all objects $A,\, B \in \A$.
\end{definition}

Of course, all $\Ext^1$-finite categories are also $\Ext^1$-small. As it happens in cocomplete Abelian categories, in $\hom$-finite Abelian categories any generator is also an epi-generator:

\begin{lemma} \label{hom-finite-generator}
Let $\A$ be a $\Hom$-finite Abelian category. Then, 
\begin{enumerate}[\rm (1)]
\item an object $G$ in $\A$ is a generator if and only if it is a finite epi-generator;
\item if $\A$ has a projective generator $P$, then $\A$ is $\Ext^1$-finite.
\end{enumerate}
\end{lemma}
\begin{proof}
Fix a ground commutative ring $R$ with respect to which $\A$ is $\hom$-finite. 

\smallskip\noindent
(1). Given a generator $G$ and an object $X\in \A$, the $R$-module $\A(G,X)$ is finitely generated, so there is a finite subset $\{f_1,\dots, f_n\} \subseteq \A(G,X)$ such that every morphism  $f\in\A(G,X)$ can be written as $f=a_1f_1+\dots +a_nf_n$, for some $a_i\in R$ and $i=1,\dots,n$. Consider the coproduct $G^{(n)}$, and denote by $\iota_i\colon G\to G^{(n)}$ the $i$-th inclusion ($i=1,\dots,n$). Then there is a unique morphism $\phi\colon G^{(n)}\to X$ such that $\phi\circ\iota_i=f_i$, for all $i=1,\dots,n$, giving an exact sequence
\begin{equation}\label{eq_homfin_gen}
\xymatrix{
G^{(n)}\ar[r]^-{\phi}& X\ar[r]^-p&\coker(\phi)\ar[r]&0.
}
\end{equation}
To show that $\phi$ is an epimorphism, we have to verify that $\coker(\phi)=0$ or, equivalently, that $p=0$. Since $G$ is a generator, this is the same as showing that $\A(G,p)\colon \A(G,X)\to \A(G,\coker(\phi))$ is the trivial map. Apply $\A(G,-)$ to the sequence \eqref{eq_homfin_gen} to get the following two maps in $\mod R$:
\[
\xymatrix@C=35pt{\A(G,G^{(n)})\ar[r]^-{\A(G,\phi)}& \A(G,X)\ar[r]^-{\A(G,p)}& \A(G,\coker(\phi)).}
\]
To conclude, note that $\A(G,p)\circ \A(G,\phi)=\A(G,p\circ \phi)=0$, while $\A(G,\phi)$ is surjective by construction. Hence, $\A(G,p)=0$, as desired.

%(1) Given a generator $G$ and an object $X\in \A$, the set $\A(G,X)$ is finite, so the coproduct $G^{(\A(G,X))}$ exists in $\A$. Hence, there is a natural epimorphism $G^{(\A(G,X))}\to X$, showing that $G$ is also an epi-generator.

\smallskip\noindent
(2). Consider two objects $A$ and $B$ in $\A$, and  a short exact sequence
\[
\xymatrix{
0\ar[r]&\ker(\pi)\ar[r]&P^{(\alpha)}\ar[r]^-\pi& A\ar[r]&0.
}
\]
Applying the functor $\A(-,B)$, we get the following long exact sequence in $\mod R$:
\[
\xymatrix{
\cdots\ar[r]&\A(P^{(\alpha)},B)\ar[r]&\A(\ker(\pi),B)\ar[r]&\Ext^1_\A(A,B)\ar[r]&\Ext^1_\A(P^{(\alpha)},B)=0,
}
\]
where $\Ext^1_\A(P^{(\alpha)},B)=0$ since $P^{(\alpha)}$ is projective. Now, $\A(\ker(\pi),B)$ is finitely generated over $R$, so also its epimorphic image $\Ext^1_\A(A,B)$ has to be a finitely generated $R$-module.
%where $\Ext^1_\A(P^{(\alpha)},B)=0$ since $P^{(\alpha)}$ is projective. Now, $\A$ is $\hom$-finite, so there is a commutative ring $R$ for which $\A(P^{(\alpha)},B)$ and $A(\ker(\pi),B)$ are finitely generated over $R$, so that also $\Ext^1_\A(A,B)$ is a finitely generated $R$-module.
\end{proof}

Our next goal is to show that $\Hom$-finite categories  are very special (see Proposition~\ref{prop.Hom-finite-are-coherent}), but let us first remind the reader of the following result, for which it is essential to assume that the ring is commutative (see \cite{Bergman}):

\begin{lemma} \label{lem.no self-reproductive product over commutative}
Let $R$ be a non-zero commutative ring, $M$ a non-zero finitely generated $R$-module and $I$ an infinite set. Then, $M$ is not isomorphic to $M^I$ in $\mod R$.
\end{lemma}
\begin{proof}
$M$ has a maximal submodule, say $N$ (see \cite[Coro.\,10.5]{AF92}). Then,  there   is a  maximal ideal $\frak m$ of $R$, together with an isomorphism $M/N\tilde{\rightarrow}R/\frak m$. It follows that $M\frak m\subseteq N\subsetneq M$.  This implies that $M^I\frak m\subseteq (M\frak m)^I\subsetneq M^I$, so  we have an epimorphism $\pi\colon M^I/M^I\frak{m}  \rightarrow M^I/(M\frak m)^I\cong (M/M\frak m)^I$, where the last isomorphism is due to the (Ab.4$^*$) condition of $\mod R$.
 Suppose, looking for a contradiction, that $M\cong M^I$ in $\mod R$. Then $\pi$ induces an epimorphism $\pi'\colon M/M\frak m \rightarrow (M/M\frak m)^I$ of $R/\frak m$-vector spaces, which is absurd since $M/M\frak m$ is finite dimensional over $R/\frak m$.
\end{proof}

%Note that the commutativity of $R$ is essential in the above lemma. To see an example of this, let $V$ be an infinite dimensional vector space over some field $\frak K$, and  $R:=\End_{\mathfrak K}(V)$. As the dimension of $V$ is infinite, $V\cong V^{(\N)}$ and so $R\cong \End_{\mathfrak K}(V)\cong \Hom_{\mathfrak K}(V^{(\N)},V)\cong R^{\N}$, as left $R$-modules.

\begin{proposition} \label{prop.Hom-finite-are-coherent}
Let $\A$ be a $\Hom$-finite Abelian category. The following assertions hold true:
\begin{enumerate}[\rm (1)]
\item $\A(A,A)$ is a left and right coherent ring, for every object $A\in\A$;
\item $\A(A,B)$ is finitely presented both as a left $\A(B,B)$- and right $\A(A,A)$-module, for all $A,\, B\in\A$;  
\item given an infinite set $I$, there is no non-zero $X\in \A$ such that the coproduct $X^{(I)}$ exists in $\A$. 
\end{enumerate}
\end{proposition}
\begin{proof}
Fix a commutative ring $R$ with respect to which $\A$ is $\hom$-finite. \textcolor{black}{For  $A,\, B,\, C\in \A$ and $E:=\A(A,A)$,  the right $E$-module $\A(A,B)$ is finitely generated  (as any set of generators over $R$ also generates $\A(A,B)$ over $E$). Thus, given $f\colon B\to C$ and denoting $f_*:=\A(A,f)\colon\A(A,B) \to \A(A,C)$, the right $E$-module $\ker(f_*)\cong \A(A,\Ker(f))$ is finitely generated.% This simple idea will be sufficient to prove parts (1) and (2):
%and $\ker(f^*)\cong \A(\Coker(f),X)$
%
%Let us note that if $f:A\longrightarrow B$ is any morphism and $X$ is any object in $\A$, then the kernel $f_*=\A(X,f):\A(X,A) \longrightarrow \A(X,B)$ is isomorphic to $\A(X,\Ker(f))$ and the kernel of $f^*=\A(f,X):\A(B,X) \longrightarrow \A(A,X)$ is isomorphic to $\A(\Coker(f),X)$. Hence both of them are finitely generated as $R$-modules. 
} 

\smallskip\noindent
\textcolor{black}{(1). We just prove that $E$ is right coherent, the proof of left coherence is symmetric. It is well-known (and easily checked) that the functor $\A(A,-)\colon \A\to\mod R$ induces an equivalence $\mathrm{sum}(A)\cong\mathrm{free}\text{-}E$, where $\mathrm{sum}(A)\subseteq \A$ is the subcategory of finite coproducts of copies of $A$ and $\mathrm{free}\text{-}E$ denotes the category of finitely generated free $E$-modules. Hence, any morphism $\varphi \colon E^n\to E$ in $\mod E$ is of the form $\varphi =f_*$, for a unique morphism $f\colon A^n\to A$ in $\A$. By the initial discussion, $\ker(\varphi)$ is a finitely generated right $E$-module. Thus, each finitely generated right ideal of $E$ is finitely presented, that is, $E$ is right coherent.}
%
%Now, we deduce from initial paragraph that $\Ker (\varphi)$ is  finitely generated as $R$-module and note that any finite set of generators as an $R$-module also generates $\Ker (\varphi)$ as a right $E$-module. Therefore, we have proved that any finitely generated right ideal of $E$ is finitely presented, that is, $E$ is right coherent.}
%(1). We just prove that $E:=\A(A,A)$ is right coherent, the proof of left coherence is symmetric. It is well-known (and easily checked) that the functor $\A(A,-)\colon \A\to\mod R$ induces an equivalence $\mathrm{sum}(A)\cong\mathrm{free}\text{-}E$, where $\mathrm{sum}(A)\subseteq \A$ is the subcategory of finite coproducts of copies of $A$ and $\mathrm{free}\text{-}E$ denotes the category of finitely generated free $E$-modules. In particular, any morphism $\varphi \colon E^n\to E$ in $\mod E$ is of the form $\varphi =\A(A,f)$, for a unique morphism $f\colon A^n\to A$ in $\A$. By the left exactness of $\A(A,-)$, we get an isomorphism $\A(A,\Ker (f))\cong\Ker (\varphi)$, so this kernel is a finitely generated $R$-module. Any finite set of generators as an $R$-module also generates $\Ker (\varphi)$ as a right $E$-module. Therefore, we have proved that any finitely generated right ideal of $E$ is finitely presented, that is, $E$ is right coherent.

\smallskip\noindent
\textcolor{black}{(2). We just verify that $\A(A,B)$ is a finitely presented right $E$-module, the rest can be proved similarly. Let $\{f_1,\dots,f_t\}$ be a set of generators of $\A(A,B)$ as an $R$-module and consider the induced map $f:=\begin{pmatrix}f_1,\dots ,f_t \end{pmatrix}\colon A^n\to B$. 
%Using once again the first paragraph together with the last part on the previous paragraph, 
By the initial discussion, $\ker(f_*)$ is a finitely generated right $E$-module and, therefore, $\A(A,B)\cong\mathrm{Im}(f_*)$ is finitely presented over $E$.% we deduce that the image of the morphism $\A(A,f)\colon\A(A,A^n)\cong E^n\to \A(A,B)$ is finitely presented as $E$-module. But, its image coincide with $\A(A,B)$ by construction of $f$.
}
%(2). We are going to verify that $\A(A,B)$ is a finitely presented right $E:=\A(A,A)$-module. The rest of the statement can be proved by similar arguments. Let $\{f_1,\dots,f_t\}$ be a set of generators of $\A(A,B)$ as an $R$-module and consider the induced map $f:=\begin{pmatrix}f_1,\dots ,f_t \end{pmatrix}:A^n\to B$, which gives a short exact sequence 
%\[
%0\rightarrow\A(A,\Ker (f))\longrightarrow\A(A,A^n)\cong E^n\longrightarrow\A(A,B)\rightarrow 0,
%\] 
%where the rightmost non-trivial map is an epimorphism by construction of $f$. Hence, $\A(A,\Ker(f))$ is a finitely generated right $E$-module since it is a finitely generated $R$-module. It follows that $\A(A,B)$ is a finitely presented right $E$-module. 

\smallskip\noindent
(3). Suppose that $X^{(I)}$ exists. Then, replacing $X$ by $X^{(I)}$ if necessary, we can assume that $X\cong X^{(I)}$. This gives  isomorphisms $\A(X,X)\cong\A(X^{(I)},X)\cong\A(X,X)^I$ in ${\A(X,X)}\modl$, which are also isomorphisms of $R$-modules. Since $R$ is commutative, this contradicts Lemma \ref{lem.no self-reproductive product over commutative}.
\end{proof}

%\begin{remark}\label{rem_fin_coproducts}
%No infinite family of non-trivial objects $\{X_i\}_{I}$ in a $\Hom$-finite category $\A$ admits a co/product in $\A$. Fix a commutative ring $R$ over which $\A$ is $\Hom$-finite.  If such an infinite coproduct $X=\coprod_{I}X_i$ existed, then we would have an isomorphism $$\A(X,X)\cong\A(\coprod_{I}X_i,X)\cong\prod_{I}\A(X_i,X),$$ where all $\A(X_i,X)$ are non-zero. 
 %Hence $\A(X,X)$ cannot be finitely generated as an $R$-module, which is a contradiction. 

%On the other hand, there may be $\Ext^1$-finite categories with infinite coproducts (e.g., any category of vectors spaces over a skew field).
%\end{remark}

\subsection{Krull-{Schmidt} categories}
Recall that an additive category with splitting idempotents $\C$ is {\bf Krull-Schmidt} when  each of its objects is a finite coproduct of (indecomposable) objects whose endomorphism rings are local. Equivalently, when $\End_\C(X)$ is a semiperfect ring for each object $X\in\C$ (see \cite[Thm.\,A.1]{chen2008algebras}).
 The following is the key result that makes Krull-Schmidt Abelian categories particularly convenient for our study of preenveloping torsion classes.

\begin{lemma} \label{lem.precovering=covering}
Let $\A$ be a Krull-Schmidt Abelian category and $\T\subseteq \A$ a subcategory closed under direct summands. Then, $\T$ is preenveloping if and only if it is enveloping.
\end{lemma}
\begin{proof}
Under our hypotheses, it is well-known that, for any morphism $f\colon A\to B$ in $\A$, there is a decomposition $B\cong B_1\oplus B_2$ such that $f$ decomposes matricially as 
\[
f=(f_1,0)^{t}\colon A\longrightarrow B_1\oplus B_2,
\]
and such that $f_1\colon A\to B_1$ is left minimal (see, e.g., \cite[Prop.\,1.2]{krause1997minimal}). Hence, if $f\colon A\to B$ is a $\T$-preenvelope, then $f_1$ is a $\T$-envelope.
\end{proof}

Recall that a ring $R$ is {\bf Krull-Schmidt} when its category of finitely presented modules $\fpmod R$ is Krull-Schmidt. This is a two-sided notion by a classical duality, due to Auslander and Gruson-Jensen,  between the categories $\fp(\fpmod{(R^{\op})},\Ab)$ and $\fp(\fpmod R,\Ab)$ of finitely presented functors to $\Ab$ from finitely presented left and right $R$-modules, respectively (see \cite[Sec.\,5]{zbMATH01032187}).

\begin{example}\label{Examples_KS_rings}
 The following are examples of Krull-Schmidt rings (we address the reader to the cited references for the definitions of each of the following types of rings).
 \begin{enumerate}[\rm (1)]
\item Perfect and, more generally, semiperfect rings $R$ such that the descending sequence 
\[
aR\supseteq a^2R\supseteq\ldots\supseteq a^nR\supseteq \ldots
\] 
(or, equivalently, its left version) is stationary, for all $a\in R$ (see \cite[Thm.\,8]{rowen1986finitely}). This includes all one-sided Artinian rings so, in particular, all Artin algebras.
\item Complete semiperfect (two-sided) Noetherian rings (see \cite[Thm.\,B]{rowen1987finitely}).
\item Rings that are finitely generated as modules over complete Noetherian local domains (see \cite{swan1960induced}). 
\end{enumerate}
\end{example}

\section{Co/reflective subcategories of Abelian categories}\label{sub_cogen_sec}

In this section we study some criteria for a subcategory of an Abelian category $\A$ to be reflective, also relating this property with that of being (pre)enveloping. In the last part of the section we apply these results to classes of the form $\Sub(\Gen(V))$, for some object $V\in \A$. As a motivation, let us anticipate that the results of this section will be fundamental for our study of quasi-tilting objects (see, e.g., Proposition \ref{prop.lorthogonal-generates-torsion}).

\subsection{Generalities on co/reflective subcategories}\label{co/reflective_subs}

A subcategory $\B$ of a category $\A$ is said to be \textbf{reflective} (resp., \textbf{coreflective}) when the inclusion $\B\to\A$ has a left (resp., right) adjoint. 

\begin{lemma}\label{reflector_is_epi_lemma}
Let $\B$ be a reflective subcategory of an Abelian category $\A$, and let 
\[
\iota\colon \B\longrightarrow \A,\qquad\tau\colon \A\longrightarrow \B\quad\text{and}\quad\rho\colon \id_\A\Rightarrow \iota\circ\tau
\] 
denote the inclusion, its left adjoint, and the unit of the adjunction, respectively. If $\B$ is closed under quotients in $\A$, then $\rho$ is epimorphic, that is, $\rho_A\colon A\to \iota(\tau(A))$ is an epimorphism for all $A\in\A$.
\end{lemma}
\begin{proof}
Let $A\in\A$ and consider the following exact sequence:
\[
\xymatrix{
A\ar[r]^-{\rho_A}&\iota(\tau(A))\ar[r]&\coker(\rho_A)\ar[r]&0.
}
\]
Since $\tau$ is a left adjoint, it is right exact and  the fact that $\tau(\rho_A)$ is an isomorphism  implies that $\tau(\coker(\rho_A))=0$. On the other hand, since $\B$ is closed under taking quotients, $\coker(\rho_A)\in \B$, implying that $\coker(\rho_A)\cong \iota(\tau(\coker(\rho_A)))=0$.
\end{proof}

Let us recall from \cite[Prop.\,3.1.3]{borceux1994handbook1}, the following useful criterion: a subcategory $\B\subseteq \A$ is reflective if, and only if, any object $A\in \A$ has a $\B$-{\bf reflection} $\rho_A\colon A\to B_A$ in $\B$, that is, $B_A\in \B$ and the following map is an isomorphism for all $B\in\B$:
\[
{\A}(\rho_A,B)\colon \A(B_A,B)\longrightarrow{\A}(A,B).
\]
A $\B$-{\bf coreflection} $B_A\to A$ is just a $\B^{\op}$-reflection in $\A^{\op}$.

\begin{proposition} \label{prop.coreflective}
Let $\A$ be an (Ab.3), well-powered Abelian category and $\iota\colon \B\to \A$ the embedding of a subcategory closed under quotients. Then, $\B$ is coreflective if, and only if, it is closed under taking coproducts in $\A$. In this case, denote  
\[
\sigma \colon\A\longrightarrow\B\quad\text{and}\quad\eta\colon\iota\circ\sigma\Rightarrow \id_\A
\]
the right adjoint to the inclusion and the corresponding counit. The following statements hold true:
\begin{enumerate}[\rm (1)]
\item $\B$ is (Ab.3) and $\eta$ is monomorphic, i.e., $\eta_X$ is a monomorphism for all $X\in\A$;
\item if $\B$ is also closed under taking kernels (or, equivalently, under subobjects), then $\B$ is an (Ab.4) (resp., (Ab.5), Grothendieck) Abelian category whenever $\A$ is so.
\end{enumerate}
\end{proposition}
\begin{proof}
It is a general fact that coreflective subcategories are closed under coproducts. On the other hand, suppose that $\B$ is closed under 
taking coproducts and let us verify that any given object $X\in\A$ admits a coreflection $\sigma(X)\to X$. Indeed, define 
\[
\xymatrix{
\sigma(X):=\sum\{Y:{Y\in \B\cap \mathcal L_\A(X)} \}
}
\] 
(where $\mathcal L_\A(X)$ is a set since $\A$ is well-powered). Then, $\sigma(X)$ is the image of the canonical morphism 
\[
\xymatrix{
\coprod_{Y\in \B\cap \mathcal L_\A(X)}Y\longrightarrow X
}
\] 
induced by the inclusions $Y\hookrightarrow X$, so $\sigma(X)\in\B$. We now verify that the inclusion $\eta_X\colon\sigma(X)\to X$ is the desired coreflection:
%Moreover the inclusion $B\hookrightarrow X$ clearly factors trough the inclusion $\sigma (X)\hookrightarrow X$, for each $B\in \B\cap \mathcal L_\A(X)$, so that $\sigma (X)$ is the largest subobject of $X$ which is in $\B$. 
% As a consequence, if $f\colon X\longrightarrow Y$ is any morphism in 
%$\A$ then we have an inclusion $f(\sigma(X))\subseteq\sigma(Y)$. 
%We then get a well-defined (additive) functor $\sigma :\A\to\B$, where 
%$\sigma(X)$ is the image of the object $X$ and
% $\sigma (f)=f_{| \sigma(X)}:\sigma(X)\to\sigma(Y)$ is the image of any morphism $f:X\to Y$. 
% Even more, i
let $B\in\B$ and consider the following map
\[
\xymatrix{
\B(B,\sigma(X))=\A(B,\sigma(X))\ar[rr]^(.52){\A(B,\eta_X)}&&\A(B,X)=\A(\iota(B),X).
}
\]
Then, $\A(B,\eta_X)$ is injective by the left exactness of $\A(B,-)$, while it is surjective since, given a morphism $f\colon B\to X$, $f(B)$ is clearly a subobject of $\sigma(X)$, so $f$ factors through $\eta_X$.

\smallskip\noindent
(1). As $\B$ is closed under coproducts in $\A$, and $\A$ is (Ab.3), coproducts do exist in $\B$ and they are computed as in $\A$, so  $\B$ is (Ab.3).  Moreover, the counit $\eta\colon \iota\circ\sigma\Rightarrow \id_\A$ is defined by the above inclusions $\eta_X\colon\sigma(X)\hookrightarrow X$ (for all $X\in \A$), so it is clearly monomorphic.

\smallskip\noindent
(2). Suppose now that $\B$ is also closed under taking kernels. Then each subobject $B'\subseteq B$ of an object $B\in\B$ is the kernel of the projection $p\colon B\to B/B'$, which is a morphism in $\B$ since this subcategory is closed under taking quotients. Therefore, $\B$ is also closed under taking subobjects. $\B$ is clearly an Abelian exact subcategory of $\A$. Furthermore, since the inclusion $\B\to\A$ is a left adjoint, it preserves all colimits. Hence, whenever coproducts (resp., directed colimits) are exact in $\A$, these colimits are also exact in $\B$, hence $\B$ is (Ab.4) (resp., (Ab.5)) {whenever $\A$ is (Ab.4) (resp., (Ab.5))}.
It remains to prove that, if $\A$ is Grothendieck, with a generator $G$, then also $\B$ has a generator. For each $B\in\B$, there is a canonical epimorphism $\epsilon \colon G^{(\A(G,B))}\to B$. For each finite subset $F\subset\A(G,B)$, the image $B_F:=\Im (\epsilon_F)$ of the 
 composition 
$\epsilon_F\colon G^{(F)}{\to}G^{(\A(G,B))}{\to} B$, is an object of $\B$. The corresponding family $\{B_F\}_F$ is then directed and, by the (Ab.5) condition, $B=\sum_FB_F$, {and} so  $B$ is a quotient of $\coprod_FB_F$. Let 
\[
\mathcal{S}:=\Quot\{G^{n}:n\in\N\}\cap \B.
\]
Being $\A$ well-powered, $\mathcal S$ is skeletally small, and it generates $\B$ by the above argument.
\end{proof}

%We obtain the following direct consequence of the above proposition:

We now apply the above proposition and its dual in a situation that we will encounter frequently through the paper.

\begin{corollary} \label{cor.subX}
Let $\A$ be a well-powered and $\B$ be a subcategory closed under taking subobjects, quotients and finite coproducts, e.g.\  $\B=\Sub(\X)$, for $\X$ closed under taking quotients and finite coproducts, or $\B=\Quot (\Y)$ for $\Y$ closed under subobjects and finite coproducts. Then:
\begin{enumerate}[\rm (1)]
\item if $\A$ is (Ab.3), then $\B$ is coreflective if, and only if, $\B$ is closed under coproducts. In this case, if $\sigma\colon\A\rightarrow\B$ is right adjoint to the inclusion functor, then the counit $\eta \colon\iota\circ\sigma\Rightarrow\id_\A$ is monomorphic and $\B$ is also (Ab.3). Moreover, $\B$ is (Ab.4), (Ab.5) or Grothendieck whenever $\A$ is so;
\item if $\A$ is (Ab.3*),  then $\B$ is reflective if, and only if, $\B$ is closed under products. In this case, if $\tau \colon\A\rightarrow\B$ is left adjoint to the inclusion functor, then the unit $\rho \colon\id_\A\Rightarrow\iota\circ\tau$ is epimorphic and $\B$ is (Ab.3*).  Moreover, $\B$ is (Ab.4*), (Ab.5*) or (Ab.5*) with an injective cogenerator whenever $\A$ is so;
\item if $\A$ is bicomplete, i.e.\ (Ab.3) and (Ab.3*),  and $\B$ is coreflective (resp., reflective), then $\B$ is bicomplete, but products (resp., coproducts) in $\B$ need not be computed as in $\A$. 
\end{enumerate}
\end{corollary}
\begin{proof}
Assertion (1) follows by Proposition \ref{prop.coreflective} and assertion (2) is just its dual. As for assertion (3), the ``in-brackets'' statement is dual to the one ``not-in-brackets''. We just prove the latter, for which it is enough to prove that if $\B$ is coreflective, then $\B$ has products. Indeed, if $(B_\lambda )_{\lambda\in\Lambda}$ is a family of objects in $\B$ and we consider the product $[P,(\pi_\lambda \colon P\rightarrow\iota (B_\lambda))_{\lambda\in\Lambda}]$ of the family $(\iota (B_\lambda))_{\lambda\in\Lambda}$ in $\A$, then it is well-known, and easy to prove, that $[\sigma (P),(\sigma (\pi_\lambda )\colon \sigma (P)\rightarrow (\sigma\circ\iota)(B_\lambda )\cong B_\lambda)_{\lambda\in\Lambda}]$ is a product  in $\B$.
\end{proof}

\begin{corollary} \label{coro.reflective-versus-productclosedness}
Let $\A$ be a well-powered (Ab.3*) Abelian category, and $\T\subseteq \A$ a subcategory {such that $\mathcal{B}:=\Sub(\mathcal{T})$ is closed} under taking products in $\mathcal{A}$ {(e.g., if $\T$ is closed under taking products)}. Then, $\mathcal{B}$   is reflective in $\A$.
\end{corollary}
\begin{proof}
It is a direct consequence of {the dual of Proposition \ref{prop.coreflective}.}
%By Corollary~\ref{cor.subX}, we know that $\mathcal{B}$ is reflective if, and only if, it is closed under taking products in $\mathcal{A}$, for which it is enough that $\mathcal{T}$ is so, since products are left exact. 
\end{proof}

Let us conclude this subsection with the following example showing that, if instead of completeness and well-powerdness of $\A$, we assume a strong enough finiteness condition on the objects, then we can still reach the same conclusion of Corollary \ref{coro.reflective-versus-productclosedness}.

\begin{example}\label{example_artinian_reflective}
Let $\A$ be an Abelian category and suppose that all the objects of $\mathcal{A}$ are Artinian. If $\T\subseteq \A$ is a subcategory closed under finite co/products, then $\B:=\Sub(\T)$ is reflective.
\end{example}
\begin{proof}
Let $A\in \A$ and consider the following class:
\[
\mathcal{S}_A:=\{K\leq A:K=\ker(\phi) \text{ for some $\phi\colon A\to B$, $B\in\B$}\}.
\]
By the DCC on subobjects of $A$, $\mathcal{S}_A$ contains a minimal element. If $S_1$ and $S_2$ are two such minimal elements, and we fix monomorphisms $A/S_1\hookrightarrow T_1$ and  $A/S_2\hookrightarrow T_2$ with $T_1,\, T_2\in \T$, then we get a monomorphism ${A}/({S_1\cap S_2})\hookrightarrow T_1\oplus T_2$, with $T_1\oplus T_2\in\mathcal T$. It follows that $S_1\cap S_2\in\mathcal{S}_A$ and, by minimality, $S_1=S_1\cap S_2=S_2$, so there is a unique minimal element $S_A$ in $\mathcal{S}_A$ which is then its minimum. Let $\pi\colon A\to B_A:=A/S_A$ be the natural projection, it is routine to check that $\A(\pi,B)\colon\A(B_A,B)\to\A(A,B)$ is an isomorphism for all $B\in\B$, so this is the desired reflection.
\end{proof}

\subsection{The category of subobjects of a (pre)enveloping class}

In this subsection we continue the study of the reflective condition of subcategories of the form $\Sub(\T)$, but specializing to the case when $\T$ is (pre)enveloping.

\begin{proposition} \label{ex.reflective-versus-productclosedness}
Let $\A$ be an Abelian category and $\T\subseteq \A$ a preenveloping subcategory. Then, $\mathcal{B}:=\Sub(\mathcal{T})$ is reflective in $\A$.
\end{proposition}
\begin{proof}
For any object $A\in\A$, we have to construct a $\B$-reflection $\rho_A\colon A\to B_A$. Indeed, take a $\mathcal{T}$-preenvelope $\mu_A\colon A\to T_A$ and consider its epi-mono factorization
\[
\xymatrix@C=40pt@R=8pt{
A\ar@{->>}[dr]_{\rho_A}\ar[rr]^{\mu_A}&&T_A.\\
&B_A\ar@{^(->}[ur]_{\iota_A}
}
\]
Clearly, $B_A\in\B$ and, since $\A(-,B)$ is left exact, ${\A(\rho_A,B)}$ is injective for all $B\in \B$. To prove that it is also surjective, fix an object $B\in\B$, an embedding $\iota\colon B\to T$ with $T\in\mathcal T$ and a morphism $f\colon A\to B$. Since $\mu_A$ is a  $\mathcal T$-preenvelope, there exists a morphism $g\colon T_A\to T$ such that $\iota\circ f=g\circ \mu_A$. Consider then the following diagram:
\[
\xymatrix@R=8pt@C=18pt{
&&&A\ar@{->>}[ld]_{\rho_A}\ar[dr]^{\mu_A}\ar@{-}[d]\\
0\ar[rr]&&B_A\ar[rr]_(.7){\iota_A}\ar@{.>}[dd]_{\bar f}&\ar@/_-8pt/[ddl]^(.35)f&T_A\ar[rr]^(.4){\pi_A}\ar[dd]^g&&\Coker(\mu_A)\ar[dd]^{h}\ar[rr]&&0\\
\\
0\ar[rr]&&B\ar[rr]_{\iota}&&T\ar[rr]_(.4){\pi}&&\Coker(\iota)\ar[rr]&&0
}
\]
which is commutative, and where the two rows are exact. Here the morphism $h$ is constructed noting that $0=\pi\circ \iota\circ f=\pi\circ g\circ \mu_A=\pi\circ g\circ \iota_A\circ \rho_A$, so $\pi\circ g\circ \iota_A=0$ (since $\rho_A$ is an epimorphism) and then, by the universal property of the cokernel, $\pi\circ g$ factors uniquely through $\pi_A$. Once we have $h$, the universal property of kernels gives the unique morphism $\bar{f}$ such that $\iota\circ\bar{f}=g\circ\iota_A$. We then have that $\iota\circ\bar{f}\circ\rho_A=g\circ\iota_A\circ\rho_A=g\circ\mu_A=\iota\circ f$, from which we get that $f=\bar{f}\circ\rho_A$ since $\iota$ is a monomorphism. Therefore, $\A(\rho_A,B)$ is also a surjective map. 
\end{proof}

The following corollary is especially useful in Krull-Schmidt categories, where all the preenveloping classes are enveloping (see Lemma \ref{lem.precovering=covering}).

\begin{corollary} \label{cor.enveloping-implies-semispecialpreenv}
Let $\A$ be an Abelian category and $\T$ a torsion class in $\A$. If $\T$ is enveloping (and cogenerating) then it is semi-special (resp., special) preenveloping.
\end{corollary}
\begin{proof}
Bearing in mind that, when $\T$ is preenveloping, $\mathcal{B}:=\Sub (\T)$ is reflective (see Proposition~\ref{ex.reflective-versus-productclosedness}),   the proof reduces to the case when $\T$ is cogenerating. In this latter case the proof goes as for categories of modules (see \cite[Lem.\,2.1.13]{gobel2006approximations}).
 %{\bf (El argumento que sigue tiene que ser tambi\'en conocido, pero no recuerdo una referencia precisa. Posiblemente el libro de Xu `Relative Homological Algebra')} Let $P$ be a projective epi-generator and let $\mu:P\longrightarrow V$ be its $\T$-envelope. Let us put $T_P:=\Coker(\mu)$ and, by way of contradiction, suppose that $T_P\not\in _{}^{\perp_1}\T$ and fix a non-split exact sequence $0\rightarrow T\longrightarrow T'\stackrel{q}{\longrightarrow}T_P\rightarrow 0$. We then get the following commutative  diagram with exact rows and columns:
% Since $\mu$ is a $\T$-envelope, we have a morphism $\lambda:V\longrightarrow T$ such that $\lambda\circ\mu =\widehat {\mu}$, and then, by the universal property of cokernels, we get another one $\lambda `:T_P\longrightarrow T'$ such that $\lambda'\circ p=\widehat {p}\circ\lambda$. Then we have that $\rho\circ\lambda\circ\mu =\rho\circ\widehat {\mu}=\mu$, which implies that $\rho\circ\lambda$ is an isomorphism by the left minimality of $\mu$. It  follows that $q\circ\lambda'$ is also an isomorphism, and hence $q$ is a retraction, which contradicts the non-split condition of the right vertical column.  
% Therefore $P$ has a semi-special $\T$-preenvelope, which implies that $\T$ is semi-special preenveloping by Lemma \ref{lem.special-preenvelope-generator}.
 \end{proof}

\subsection{The subcategory of objects sub(co)generated by a given  object}

Recall from Section 1 the definition of the subcategories $\Gen(M)$ and $\Cogen(M)$, for a given object $M$ in an Abelian category $\A$. We enlarge these classes to form new subcategories to which 
 we apply the results of the previous subsections.

\begin{definition}
%Suppose that $\A$ is an %(Ab.3) (resp., (Ab.3*)), well-powered 
%Abelian category, and l
Let $X$ and $M$ be two objects in an Abelian category $\A$. We say that  $X$ is {\bf subgenerated} by $M$ if it embeds in an object of $\Gen(M)$. We denote by $\subGen_\A(M):=\Sub(\Gen_\A(M))$ the subcategory of objects subgenerated by $M$.

The objects {\bf  subcogenerated} by $M$ are defined dually. The subcategory of objects  subcogenerated by $M$ is denoted by  $\subCogen_\A(M):=\Quot(\Cogen_A(M))$. When the ambient category $\A$ is clear from the context, the subscript will be generally avoided.
\end{definition}

%In the above definition, we are not assuming that $\A$ is (Ab.3) or (Ab.3*). Hence, for example, $\Gen_\A(M)$ should be understood as the category of quotients of those coproducts of $M$ that exist in $\A$. 

\begin{lemma}\label{example_4_lemma}
Let $\A$ be an Abelian category with a projective epi-generator $P$ and  $\T\subseteq \A$  a subcategory closed under quotients, coproducts and extensions. If $\T=\Gen(V)$,  for an object $V\in \A$ such that $P$ has an $\Add(V)$-preenvelope, then $\B:=\Sub(\T)=\subGen(V)$ is reflective in $\A$. 
\end{lemma}
\begin{proof}
Let $f\colon P\to V^{(I)}$ be an $\Add(V)$-preenvelope. It is enough to prove that it is also a \mbox{$\T$-preenvelope.} In fact, if that happens, then  $\T$ is a preenveloping class in $\A$ by Lemma \ref{lem.special-preenvelope-generator} and so Proposition \ref{ex.reflective-versus-productclosedness} applies. Let $u\colon P\to T$ be any morphism, where $T\in\T=\Gen(V)$. Fix an epimorphism $p\colon V^{(J)}\twoheadrightarrow T$. By the projectivity of $P$, we get a morphism $v\colon P\to V^{(J)}$ such that $p\circ v=u$. Since $f$ is an $\Add(V)$-preenvelope, we then get a morphism $w\colon V^{(I)}\to V^{(J)}$ such that $w\circ f=v$. It follows that $u=p\circ v=p\circ w\circ f$, thus showing that $f$ is a \mbox{$\T$-preenvelope.}
\end{proof}

Note that, putting $\mathcal{X}=\Gen(M)$ (resp., $\mathcal{X}=\Cogen(M)$) in Corollary~\ref{cor.subX}, one gets assertions about $\subGen(M)$ (resp., $\subCogen(M)$) whose statement is left to the reader. Nevertheless, we want to emphasize the following  consequence.

\begin{corollary} \label{cor.subGen-subCogen-Grothendieck}
Let $\A$ be a well-powered Abelian category and $M\in\A$. Then:
\begin{enumerate}[\rm (1)]
\item if $\A$ is (Ab.5), then $\subGen(M)$ is a Grothendieck category;
\item if $\A$ is (Ab.5) and (Ab.3*), then  $\subCogen(M)$ is (Ab.5);
\item if $\A$ is a Grothendieck category, so is $\subCogen(M)$.  
\end{enumerate}
\end{corollary}
\begin{proof}
(1). By Corollary~\ref{cor.subX}(1), we know that $\subGen(M)$ is (Ab.5). We just need to exhibit a set of generators. The strategy is completely analogous to the case when $\A$ is a category of modules (see \cite[Sec.\,15]{wisbauer2018foundations}). Indeed, consider the following class:
\[ 
\mathcal{S}:=\Sub\{M^n:n\in\N\}
\]
which is skeletally small since $\A$ is supposed to be well-powered. Given $T\in\subGen(M)$, fix a monomorphism $\lambda\colon T\to M^{(I)}/K$. Furthermore, let $\lambda_F\colon M^{(F)}/(K\cap M^{(F)})\to M^{(I)}/K$ be the induced monomorphism, for any finite subset $F\subseteq I$. Then $M^{(I)}/K$ is the directed union of the objects $L_F:=\Im(\lambda_F)$. Since $\A$ is supposed to be (Ab.5), we obtain that 
$T$ is the directed union of the  $\lambda^{-1}(L_F)$ (see \cite[Prop.\,V.1.1]{S75}). But each 
$\lambda^{-1}(L_F)$ is isomorphic (via $\lambda$) to an object of $\mathcal{S}$. Therefore $T$ is a quotient of a 
coproduct of 
objects in $\mathcal{S}$, showing that $\mathcal{S}$ is a skeletally small class of generators of $\subGen(M)$. 

\smallskip\noindent
(2) {\em and} (3). Note that $\B:=\subCogen(M)$ is closed under taking subobjects and quotients in $\A$.  By Proposition~\ref{prop.coreflective} it is enough to check that $\B$ is closed under taking coproducts and, since coproducts are exact in $\A$, it is clearly sufficient to prove that $\Cogen(M)$ is closed under taking coproducts. For that, note first that, given a family $\{X_i\}_{I}$ of objects in $\A$, the canonical morphism $\coprod_{I}X_i\to\prod_{I}X_i$ is a monomorphism (see \cite[Coro.\,8.10]{popescu}). 
%since it is the directed colimit of the canonical split monomorphisms 
%\[
%\xymatrix{
%\iota_F\colon\coprod_{i\in F}X_i=\prod_{i\in F}X_i\to\prod_{I}X_i,
%}
%\] 
%where $F$ varies in the set of finite subsets of $I$. 
Let then $\{T_i\}_{I}$ be a family of objects in $\Cogen(M)$ and fix monomorphisms $\lambda_i\colon T_i\hookrightarrow M^{J_i}$, for suitable sets $J_i$ ($i\in I$). We obtain a monomorphism 
\[
\xymatrix@C=12pt{
\coprod_{I}T_i\hookrightarrow\coprod_{I}M^{J_i}\ar[rr]^(.6){\text{can}}&&\prod_{I}M^{J_i},
}
\] 
showing that $\coprod_{I}T_i\in\Cogen(M)$.  
\end{proof}

\section{Semi-special preenveloping classes}\label{sec. semi-special preenveloping}

In this section we introduce the main tools that will be necessary in Section \ref{sec_tilt_pree}. In particular, in Subsection~\ref{unversal_ext_sec}, we develop a theory of universal extensions and $\Ext^1$-universal objects in an Abelian category $\A$. This can be done with relative ease when we impose sufficiently strong finiteness conditions on $\A$ (see, e.g., Proposition \ref{prop.characteriz.Ext-universal}) but, to treat the case of ``big'' Abelian categories (see, e.g., Proposition \ref{prop.derivcoprod=0-implies-projdim=1}), we need a preliminary study of some exactness properties of infinite coproducts, that we carry on in Subsection~\ref{exact_coprod_sec}. Finally, we characterize the semi-special preenveloping classes in an arbitrary Abelian category $\A$ in Subsection~\ref{preenvelopes_main_subs} (see Theorem \ref{thm.specialpreenveloping-versus-cotorsionpair} and Corollary \ref{cor.special-preenveloping-arbitraryA} for the main results in this direction).

\subsection{A formal derived functor of coproducts}\label{exact_coprod_sec}

In a cocomplete Abelian category, all coproducts exist but, in general, their derived product may fail to exist. Nevertheless, we introduce the following notation:

\begin{notation} \label{notation.derived coproduct}
Let $\A$ be an Abelian category and let $I$ be a set such that $I$-coproducts exist in $\mathcal{A}$. Let 
$(V_i)_{I}$ be a family of objects and  $\mathcal{S}$  a subcategory  of $\mathcal{A}$. We shall write 
\[
\xymatrix{
\coprod_{I}^1V_i\in\mathcal{S}
}
\] 
when,  given any family $\{0\rightarrow X_i\stackrel{u_i}{\longrightarrow}Y_i\to V_i\rightarrow 0\}_{I}$ of short exact sequences in $\mathcal{A}$, the kernel of the induced morphism $\coprod_I u_i\colon\coprod_{I}X_i\to\coprod_{I}Y_i$ lies in $\mathcal{S}$. When $\mathcal{S}=\{0\}$, we  write $\coprod_{I}^1V_i=0$. This applies to the case when $V_i=V$, for all $i\in I$ and $V\in \A$, in which case we write  $\coprod_{I}^1V\in\mathcal{S}$ or  $\coprod_{I}^1V=0$, when $\mathcal{S}=\{0\}$.
\end{notation}

We now want to show that, when the first derived functor of the coproduct exists in $\A$, then Notation~\ref{notation.derived coproduct} means exactly what one expects (see Proposition \ref{prop.derived coproduct} for details). Before proceeding, let us fix some terminology. Indeed, let $\A$ be an Abelian category with enough projectives. Given an object $V\in \A$, a {\bf projective presentation} of $V$ is a short exact sequence in $\A$
\[
0\rightarrow\Omega{\longrightarrow} P\longrightarrow V\rightarrow 0\quad\text{with $P\in\Proj(\A)$.}
\]

\begin{proposition} \label{prop.derived coproduct}
Let $\mathcal{A}$ be an Abelian category with enough projectives, $I$ a set for which \mbox{$I$-c}oproducts exist in $\mathcal{A}$,  and let $\{V_i\}_{I}$ and $\mathcal{S}$ be a family of objects and a subcategory  of $\mathcal{A}$, respectively. Consider the following assertions:
\begin{enumerate}[\rm (1)]
\item $\coprod_{I}^1V_i\in\mathcal{S}$;
\item $\ker(\coprod_I u_i)\in \mathcal{S}$, for any family of projective presentations $\{0\rightarrow\Omega_i\stackrel{u_i}{\longrightarrow} P_i\to V_i\rightarrow 0\}_{I}$;
\item $\Ker(\coprod_I u_i)\in\mathcal{S}$, for some family of projective presentations  $\{0\rightarrow\Omega_i\stackrel{u_i}{\longrightarrow} P_i\to V_i\rightarrow 0\}_{I}$;
\item $H^{-1}(\coprod_{I}V_i[0])\in\mathcal{S}$, where $\coprod_{I}V_i[0]$ denotes the coproduct  in $\Der^-(\A)$ (see Section~\ref{subs_complexes}). 
\end{enumerate}
Then the implications ``(1)$\Rightarrow$(2)$\Leftrightarrow$(3)$\Leftrightarrow$(4)'' hold true. Furthermore, when $\mathcal{S}$ is closed under taking quotients in $\mathcal{A}$, all  assertions are equivalent. 
\end{proposition}
\begin{proof}
Consider a family of projective presentations $\{0\rightarrow\Omega_i\stackrel{u_i}{\longrightarrow} P_i\to V_i\rightarrow 0\}_{I}$. For any $j\in I$, one can choose a projective resolution $s_j\colon P^\bullet_j{\to}V_j[0]$, where
\[
\xymatrix@C=20pt@R=0pt{
\cdots\ar[r]&P_{j}^{-2}\ar[rr]^{d_j^{-2}}&&P_j^{-1}\ar[drr]_{\pi_j}\ar[rrrr]^{d_j^{-1}}&&&&P_{j}^0\cong P_j\\
&&&&&\Omega_j\ar[dr]\ar[urr]_{u_j}\\
&&&&0\ar[ur]&&0
}
\]
such that $P_j^0\cong P_j$ and with $\Omega_j$ (necessarily) isomorphic to the $0$-th boundary of  $P^\bullet_j$. As coproducts are right-exact, it follows that $\coprod_I \pi_i$ is an epimorphism and, as $\coprod_Iu_i\circ\coprod_{I}\pi_i=\coprod_Id_{i}^{-1}$, we can deduce that $\ker(\coprod_Iu_i)\cong \ker(\coprod_Id_{i}^{-1})/\ker(\coprod_I\pi_{i})$ and $\Im(\coprod_Iu_i)=\Im(\coprod_Id_{i}^{-1})$. Therefore, the  short exact sequence $0\to \ker(\coprod_Iu_i)\to \coprod_I\Omega_i\to \Im(\coprod_Iu_i)\to 0$, can be rewritten as follows:
\[
\xymatrix{
0\ar[r]& \ker(\coprod_Id_{i}^{-1})/\ker(\coprod_I\pi_{i})\ar[r]& \coprod_I\Omega_i\ar[r]& \Im(\coprod_Id_{i}^{-1})\ar[r]& 0
}
\]
Using again the right exactness of coproducts, we get that $\Im(\coprod_Id_i^{-2})=\ker(\coprod_I\pi_i)$, so we have the following short exact sequence:
\[
\xymatrix@R=4pt{
0\ar[r]& H^{-1}\left(\coprod_{I}P_i^\bullet\right)\ar[r]&\coprod_{I}\Omega_i\ar[r]&\Im \left(\coprod_I d_i^{-1}\right)\ar[r]& 0,
}
\]
where $\coprod_{I}P^\bullet_i$ is computed as in $\Ch(\mathcal{A})$. Using the inclusion $\Im(\coprod_I d_i^{-1})\hookrightarrow\coprod_{I}P_i^0$, we get  
\[
\xymatrix@R=3pt{
0\ar[r]& H^{-1}(\coprod_{I}P^\bullet_i)\ar[r]&\coprod_{I}\Omega_i\ar[r]^{\coprod_I u_i}&\coprod_{I}P_i.
}
\]
Since $H^{-1}(\coprod_{I}P^\bullet_i)\cong H^{-1}(\coprod_{I}V_i[0])$ does not depend on the choice of the projective resolutions, the equivalence of assertions (2), (3) and (4) is clear. Moreover, they are clearly implied by (1).

\smallskip
We next prove the implication ``(2)$\Rightarrow$(1)'', assuming that $\mathcal{S}$ is closed under taking quotients in $\mathcal{A}$. Let $\{0\rightarrow X_i\stackrel{u_i}{\longrightarrow}Y_i\stackrel{p_i}{\longrightarrow}V_i\rightarrow 0\}_{ I}$ be a family of short exact sequences in $\mathcal{A}$, and fix epimorphisms $\pi_i\colon P_i\twoheadrightarrow Y_i$, where each $P_i$ is a projective object. For each $j\in I$,  we get the following commutative diagram with exact rows, where the left square is bicartesian:
\[
\xymatrix@R=20pt@C=45pt{ 
0\ar[r] & \Omega_j \ar[r]^{\widetilde{u}_j} \ar@{->>}[d]_{\rho_j} & P_j \ar[r]^{\widetilde{p}_j} \ar@{->>}[d]^{\pi_j}  & V_j \ar[r] \ar@{=}[d] & 0 \\ 
0 \ar[r] & X_j \ar[r]_{u_j} & Y_j \ar[r] \ar@{}[ul]|{\text{\tiny P.O.\quad P.B.}} & V_j \ar[r] & 0.}
\]
Bearing in mind that the coproduct functor preserves pushouts, we also  get the following commutative diagram with exact rows, where the central square is cocartesian:
\[
\xymatrix@R=20pt@C=45pt{
0 \ar[r] & K \ar[r] \ar[d]_{f} & {\coprod}_{I} \Omega_i \ar[r]^{\coprod_I \widetilde{u}_i} \ar@{>>}[d]_{\coprod_I \rho_i}& {\coprod}_{I} P_i \ar[r] \ar@{>>}[d]^{\coprod_I \pi_i}& {\coprod}_{I} V_i \ar[r] \ar@{=}[d]& 0 \\ 
0 \ar[r] & L \ar[r] & {\coprod}_{I} X_i \ar[r]_{\coprod_I u_i} &\ar@{}[ul]|{\text{\tiny P.O.}}  {\coprod}_{I} Y_i \ar[r] & {\coprod}_{I} V_i \ar[r] & 0
}
\]
By (2) we know that $K\in\mathcal{S}$ so we should prove that the map $f$ in the above diagram is an epimorphism, since then one concludes that $L\in\mathcal{S}$, being $\mathcal S$ closed under quotients. For this, consider the following diagram
\[
\xymatrix@R=20pt@C=45pt{
 {\coprod}_{I} \ar@{->>}[r]\Omega_i \ar@/_-15pt/[rr]|{\coprod_I \widetilde{u}_i} \ar@{->>}[d]_{\coprod_I \rho_i}&M\ \ar@{->>}[d]\ar@{>->}[r]&{\coprod}_{I} P_i \ar@{->>}[d]^{\coprod_I \pi_i}\\
 {\coprod}_{I} X_i\ar[r] \ar@/_15pt/[rr]|{\coprod_I u_i} &N\ar@{}[ul]|{\text{\tiny P.O.}}\ar[r]&\ar@{}[ul]|{\text{\tiny P.O.}}  {\coprod}_{I} Y_i 
}
\]
where the first line is the epi-mono factorization of $\coprod_I\tilde{u}_i$, the rest of the diagram is constructed by two successive pushouts, using that the juxtaposition of two cocartesian squares is cocartesian. Note that square on the right-hand side is bicartesian, so the map $N\to {\coprod}_{I} Y_i $ is a monomorphism. Hence,  the kernel of the map $ {\coprod}_{I} X_i\to N$ is isomorphic to $\ker(\coprod_I u_i)$. Therefore, we obtain the the following commutative diagram with exact rows:
\[
\xymatrix@C=45pt@R=20pt{
0 \ar[r] & K \ar[r] \ar[d]_{f} & {\coprod}_{I} \Omega_i \ar[r] \ar@{>>}[d]_{\coprod_I \rho_i}& M\ar@{>>}[d]\ar[r]&0 \\ 
0 \ar[r] & L \ar[r] & {\coprod}_{I} X_i \ar[r] &\ar@{}[ul]|{\text{\tiny P.O.}} N  \ar[r] & 0 }
\]
We can now conclude by the dual of \cite[Lem.\,2.5.3]{popescu}.
\end{proof}

To further support the intuition that $\coprod^1_I(-)$ is a good substitute for the derived functor of $I$-coproducts, we show in the following example that, if coproducts are ``exact enough'', then $\coprod^1_I(-)$ vanishes.

\begin{example}\label{ex_exact_copr}
Let $\mathcal{A}$ be an Abelian category  and $V\in \A$. If either one of the following two conditions hold, then $\coprod^1_IV=0$, for any set $I$ for which all $I$-coproducts of exact sequences ending at $V$ exist:
\begin{enumerate}[\rm (1)]
\item if $\mathcal{A}$ is (Ab.4);
\item if $V^{\perp_1}$ is a cogenerating subcategory of $\A$.
\end{enumerate}
\end{example}
\begin{proof}
Consider a family  $\{0\rightarrow A_i\to B_i\to V\rightarrow 0\}_{I}$ of short exact sequences in $\mathcal{A}$ whose coproduct exists in $\A$. If $\A$ is (Ab.4), then the coproduct morphism $u\colon\coprod_{I}A_i\to\coprod_{I}B_i$ is a monomorphism and we are done. 
Suppose now that $V^{\perp_1}$ is cogenerating. Then, for each $X\in V^{\perp_1}$, and each $j\in I$, the induced map $\mathcal{A}(B_j,X)\to\mathcal{A}(A_j,X)$ is surjective. Hence, the morphism 
\begin{align}\label{sequence_final_univ_ext}
\xymatrix{
\mathcal{A}\left(\coprod_{I}B_i,X\right)\cong\prod_{I}\mathcal{A}(B_i,X)\longrightarrow\prod_{I}\mathcal{A}(A_i,X)\cong\mathcal{A}\left(\coprod_{I}A_i,X\right)
}
\end{align}
is surjective because products are exact in $\mathrm{Ab}$. Since $ V^{\perp_1}$ is a cogenerating class, the coproduct morphism $u\colon\coprod_{I}A_i\to\coprod_{I}B_i$ is a monomorphism by Lemma \ref{injective_preenvelope}(3.1). 
\end{proof}

\begin{remark} \label{rem.Coupek-Stovicek}
It is not at all uncommon to find cocomplete non-(Ab.4) Abelian categories $\mathcal{A}$ with an object $V$ satisfying condition (2) of the above example. For instance, if $\mathcal{G}$ is any non-(Ab.4$^*$) Grothendieck category, then, for any torsion pair $\mathbf{t}=(\T,\mathcal{F})$ of finite type in $\G$ such that $\mathcal{F}$ is generating, this latter class is of the form $\mathcal{F}=\Cogen(Q)={}^{\perp_1}Q$, for some $(1\text{-})$cotilting object (see \cite[Thm.\,3.10]{coupek-stovicek} and the dual of Definition \ref{def.quasi-tilting torsion pair}). We refer the reader to [Op.Cit.] for a good supply of examples when $\G=\mathrm{Qcoh}\text{-}\mathbb{X}$ is the category of quasi-coherent sheaves over a Noetherian scheme with an ample family of line bundles, that is generally not (Ab.4$^*$).   Choosing $\A:=\G^{\op}$ and taking $V:=Q$ as an object of $\A$, condition (2) of Example \ref{ex_exact_copr} is clearly satisfied. 
\end{remark}

Finally we can show that, when we assume that $\coprod_I^1V_i=0$, then we can say more about the morphism $\Phi$ introduced  in Lemma \ref{lem.Ext-V-versus-Ext-VI}.

\begin{proposition} \label{prop.Ext-contrav-preserves-products}
Let $\A$ be an Abelian category, $X\in \A$, $I$ a set such that $I$-coproducts exist in $\A$, and $(V_i)_{I}\subseteq \A$  a family  such that $\coprod_{I}^1V_i=0$.  Then, the canonical morphism of big Abelian groups 
\[
\xymatrix{
\Phi \colon \Ext_\A^1\left(\coprod_{I}V_i,X\right)\longrightarrow \prod_{I}\Ext_\A^1(V_i,X)
}
\] 
is an isomorphism. In particular, if $V$ is any object such that $\coprod_I^1V=0$, then the canonical morphism $\Ext_\A^1\left(V^{(I)},X\right)\to\Ext_\A^1(V,X)^I$ is an isomorphism. 
\end{proposition}
\begin{proof}
By Lemma \ref{lem.Ext-V-versus-Ext-VI}, we just need to prove that $\Phi$ is surjective. Let $([\epsilon_i])_{I}\in\prod_{I}\Ext_\A^1(V_i,X)$ and let the component $[\epsilon_j]$ be represented by a short exact sequence 
\[
\epsilon_j:\quad 0\rightarrow X\stackrel{u_j}{\longrightarrow}Y_j\stackrel{p_j}{\longrightarrow}V_j\rightarrow 0,
\] 
for each $j\in I$. The condition $\coprod_{I}^1V_i=0$ gives that $\coprod_I u_i\colon X^{(I)}\to\coprod_{I}Y_i$ is a monomorphism. We next consider the pushout of this monomorphism and the co-diagonal map $\nabla \colon X^{(I)}\to X$. Then, for each index $j\in I$, we get the following commutative diagram with exact rows, where the lower left square is the mentioned pushout and the three upper vertical arrows are the $j$-th inclusions into the respective coproducts:
\[
\xymatrix@C=25pt{\epsilon_j:
&0 \ar[rr] && X \ar[rr]^{u_j} \ar[d]_{\iota_j} && Y_j \ar[rr]^{p_j} \ar[d]^{\mu_j} && V_j \ar[rr] \ar[d]^{\lambda_j} && 0\\ 
&0 \ar[rr] && X^{(I)} \ar[rr]^{\coprod_I u_i} \ar[d]_{\bigtriangledown}\ar@{}[rrd]|{\text{P.O.}} && \coprod Y_i \ar[rr]^{\coprod_I p_i} \ar[d] && \coprod V_i \ar[rr] \ar@{=}[d] && 0 \\ 
\epsilon:&0 \ar[rr] && X \ar[rr] && Y \ar[rr] %\pullbackcorner 
&& \coprod V_i \ar[rr] && 0
}
\]
Denote by $[\epsilon]$ the element of $\Ext_\A^1(\coprod_{I}V_i,X)$ represented by the lower row of the diagram. Since $\nabla\circ\iota_j=\id_X$, the juxtaposition of the two right-most squares is a pullback. This just says that the morphism $\lambda_j^*:=\Ext_\A^1(\lambda_j,X)\colon\Ext_\A^1(\coprod_{I}V_i,X)\to\Ext_\A^1(V_j,X)$ takes $[\epsilon]\mapsto[\epsilon_j]$. Since  $\Phi$ is  induced by the $\lambda_j^*$, with $j$ varying in $I$, we see that $\Phi ([\epsilon])=([\epsilon_i])_{I}$. Hence, $\Phi$ is surjective. 
\end{proof}

\subsection{Universal extensions}\label{unversal_ext_sec}

We start extending the following definition, which is usually given in categories of modules.

\begin{definition} \label{def.universal extension}
Let $\A$ be an Abelian category. Given $A$ and $B\in\mathcal{A}$, we say that a short exact sequence of the form
\[
0\rightarrow A\stackrel{u}{\longrightarrow}X\stackrel{p}{\longrightarrow}B^{(J)}\rightarrow 0,
\] 
for some non-empty set $J$,  is a {\bf universal extension} of $B$ by $A$ if the following equivalent (for the equivalence see the comment right below) conditions hold:
\begin{enumerate}
\item[({UE})] the map $\Ext_\mathcal{A}^1(B,u)\colon\Ext_\mathcal{A}^1(B,A)\to \Ext_\mathcal{A}^1(B,X)$ is the zero map; 
\item[({UE}')] the connecting morphism $\mathcal{A}(B,B^{(J)})\to\Ext_\mathcal{A}^1(B,A)$ is surjective; 
\item[({UE}'')]  the map $\Ext_\mathcal{A}^1(B,p)\colon\Ext_\mathcal{A}^1(B,X)\to \Ext_\mathcal{A}^1(B,B^{(J)})$ is injective.  
\end{enumerate}
An object $B$ of $\A$ is said to be (left) {\bf $\Ext^1$-universal} when a universal extension of $B$ by any other object exists in $\A$.
 \end{definition}

Note that, to prove that the conditions (UE), (UE') and (UE'') are all equivalent, it is enough to consider the following long exact sequence:
\[
\xymatrix@C=15pt{
\cdots\ar[r]& \A(B,X)\ar[r]&\A(B,B^{(J)})\ar[r]& \Ext^1_\A(B,A)\ar[r]&\Ext^1_\A(B,X)\ar[r]&\Ext^1_\A(B,B^{(J)})\ar[r]&\cdots
}
\]
In the following proposition and its corollary we give an easy characterization of $\Ext^1$-universal objects in categories that have suitable finiteness conditions on $\Hom$-sets.

\begin{proposition} \label{prop.characteriz.Ext-universal}
Let $\A$ be an Abelian category, $V\in\A$ and consider the following assertions:
 \begin{enumerate}[\rm (1)]
 \item $\Ext_\A^1(V,A)$ is finitely generated as a right $\End_\A(V)$-module, for all $A\in\A$;
 \item  $V$ is $\Ext^1$-universal.
 \end{enumerate} 
Then, the implication ``(1)$\Rightarrow$(2)'' holds true and, whenever  $\A(V,A)$ is finitely generated as a right $\End_\A(V)$\mbox{-m}odule for all $A\in\A$, also the converse implication is verified. 
\end{proposition}
\begin{proof}
(1)$\Rightarrow$(2) is an easy adaptation of a well-known argument from the Representation Theory of Artin Algebras (see, for example, the proof of \cite[Lem.\,2.1]{Bo81}). %Indeed, let $A\in \A$ and fix a finite set $\{[e_1],\dots,[e_n]\}$ which generates $\Ext_\mathcal{A}^1(V,A)$ as a
% $R$-module in assertion 1 or as a 
%right $\End_\A(V)$-module.  For each $k=1,\dots,n$, fix a short exact sequence representing $[e_k]$:
%\[
%e_k:\quad 0\rightarrow A\stackrel{u_k}{\longrightarrow}Y_k\stackrel{p_k}{\longrightarrow}V\rightarrow 0.
%\]  
%We then get an element $[\epsilon]\in \Ext_\mathcal{A}^1(V^n,A)$ represented by the lower horizontal row of the following commutative diagram with exact rows
%\[
%\xymatrix@C=40pt{ 
%&0 \ar[r] & A^{n}\ar@{}[dr]|{\text{\tiny P.O.}} \ar[r]^{\coprod u_k \hspace{0.4 cm}} \ar[d]_{\nabla}& \coprod^{n}_{k=1} Y_k \ar[r]^{\hspace{0.5 cm}\coprod p_k } \ar[d]^{\delta} & V^{n} \ar[r] \ar@{=}[d]& 0\\ 
%\epsilon:&0 \ar[r] & A \ar[r]_{u} & Y \ar[r]_{q} & V^{n} \ar[r] & 0
%}
%\]
%where $\nabla$ is the codiagonal. It is routinary to see, applying the long exact sequence of $\Ext_\mathcal{A}^*(V,-)$, that the connecting morphism $w\colon\mathcal{A}(V,V^n)\to\Ext_\mathcal{A}^1(V,A)$ takes $\iota_k\rightarrow [e_k]$, where $\iota_k:V\to V^n$ is the $k$-th inclusion into the coproduct, for each $k=1,\dots,n$. Then $w$ is an epimorphism and, hence, $\epsilon$ is a universal extension of $V$ by $A$. 

\smallskip\noindent
(2)$\Rightarrow$(1){\em \ under the extra hypothesis}.  Let $A\in \A$ be any object.  By  condition (UE') in the definition of universal extension, with $B=V$, we can deduce that $\Ext_\A^1(V,A)$ is  finitely generated as a right $\End_\A(V)$-module, as it is a quotient of $\A(V,V^{(J)})$.  
\end{proof}

%\begin{lemma}
%Let $\A$ be an Abelian category and $V\in \A$ an $\Ext^1$-finite object.  Then, $V$ is $\Ext^1$-universal.
%\end{lemma}

\begin{corollary}\label{coro_ext_univ_finite}
Let $\A$ be a $\Hom$-finite Abelian category over a commutative ring $R$. Then, an object $V\in \A$ is $\Ext^1$-universal if, and only if, $\Ext_\A^1(V,A)$ is finitely generated as an $R$-module for all $A\in\A$.
\end{corollary}
\begin{proof}
As $\A$ is $\Hom$-finite over the commutative ring $R$, then the $R$-algebra $\End_\A(V)=\A(V,V)$ is finite (i.e., finitely generated as an $R$-module). Hence, an $\A(V,V)$-module is finitely generated if and only if it is finitely generated as an $R$-module. Now apply Proposition \ref{prop.characteriz.Ext-universal}.
\end{proof}

Our following result shows that $\Ext^1$-universality is related with smallness of some $\Ext^1$ groups. The study of the formal derived functors of coproducts carried on in the previous subsection turns out to be essential to understand that relation.
%\textcolor{black}{In the following proposition we characterize the $\Ext^1$-universal objects in cocomplete Abelian categories. The study of the formal derived functors of coproducts carried on in the previous subsection turns out to be extremely useful in this context.}

\begin{proposition}\label{prop_ext_univ_infinite}
Let $\A$ be an (Ab.3) Abelian category. If $V\in \A$ is an $\Ext^1$-universal object then $\Ext_\A^1(V,A)$ is a set (as opposed to a proper class), for all $A\in\A$. Moreover, if $\coprod^1_IV=0$, for any set $I$, the converse implication also holds. 
\end{proposition}
\begin{proof}
Suppose that $V$ is $\Ext^1$-universal and let $A\in \A$. Then,  by condition (UE') in the definition of universal extension, the big Abelian group $\Ext_\A^1(V,A)$ is a quotient of $\A(V,A)$, and it is therefore ``small''. On the other hand, assume that $\Ext_\A^1(V,A)$ is ``small'' for all $A\in \A$ and that $\coprod_I^1V=0$ for all sets $I$. Then, for each element $[e]\in E:=\Ext_\mathcal{A}^1(V,A)$, we choose a short exact sequence 
\[
e:\quad 0\rightarrow A\longrightarrow B_e\longrightarrow V\rightarrow 0
\] representing $[e]$. By the  hypotheses on $V$, the coproduct of these sequences, with $[e]$ varying in $E$, gives a short exact sequence 
\[
\xymatrix@C=32pt{
0\ar[r]& A^{(E)}\ar[r]^(.4){u}&\coprod_{[e]\in E}B_e\ar[r]^(.55){p}&V^{(E)}\ar[r]& 0.
}
\] 
By taking the pushout of $u$ and the codiagonal $\nabla\colon A^{(E)}\to A$, and using standard properties of pushouts, we get an exact sequence 
\[
\xymatrix@C=32pt{
0\ar[r]& A\ar[r]^{v}&B\ar[r]^(.45){q}&V^{(E)}\ar[r]& 0,
}
\] 
which we claim is a universal extension of $V$ by $A$.  Indeed, for each $[e']\in E$,  we have the following commutative diagram with exact rows, 
\[
\xymatrix@R=18pt@C=45pt{0 \ar[r] & A \ar[r]^{u_{e'}} \ar[d]_{\iota_{e'}} & B_{e'} \ar[r]^{p_{e'}} \ar[d]^{\iota_{e'}} & V \ar[r] \ar[d]^{\iota_{e'}} & 0 \\ 0 \ar[r] & A^{(E)} \ar[r]^{u}  \ar[d]_{\nabla} & B^{(E)} \ar[r]^{p} \ar[d]^{\phi} & V^{(E)}  \ar[r] \ar@{=}[d] & 0 \\ 0 \ar[r] & A \ar[r]_{v} & B \ar[r]_{q} \ar@{}[ul]|{\text{\tiny P.O.}} & V^{(E)} \ar[r]  & 0}
\]
where $\iota_{e'}$ is the inclusion into the coproduct and where the lower left square is cocartesian. Note also that $\nabla\circ\iota_{e'}=\id_A$. This implies that the following square is cartesian
\[
\xymatrix{ B_{e'} \ar[r]^{p_{e'}} \ar[d]_{\phi \ \circ \ \iota_{e'}} & V  \ar[d]^{\iota_{e'}} \\ B \ar[r]_{q} & V^{(E)} \ar@{}[ul]|{\text{\tiny P.B.}}}
\]
Hence, the connecting morphism $w\colon\mathcal{A}(V,V^{(E)})\to\Ext_\mathcal{A}^1(V,A)$ takes $\iota_{e'}\mapsto [e']$, for each $[e']\in E$. Then $w$ is surjective and the universal extension of $V$ by $A$ always exists. 
\end{proof}

The following result shows that, under suitable hypotheses on $V$, the condition $\coprod^1_IV=0$ has strong consequences. 

\begin{proposition} \label{prop.derivcoprod=0-implies-projdim=1}
Let $\A$ be an Abelian category and $V\in \A$ such that  $\Ext_\mathcal{A}^1(V,V^{(I)})=0$, for all sets $I$ for which $I$-coproducts exist in $\mathcal{A}$. Then,
\begin{enumerate}[\rm (1)]
\item if $\A$ is $\Hom$-finite, the following are equivalent:
\begin{enumerate}
\item[\rm (1.1)] $V$ is $\Ext^1$-universal;
\item[\rm (1.2)] $V^{\perp_1}$ is a cogenerating subcategory of $\A$;
\end{enumerate}
\item if $\A$ is (Ab.3), the following  conditions are equivalent:
\begin{enumerate}[\rm (2.1)]
\item $\Ext_{\A}^1(V,A)$ is a set (as opposed to a proper class), for all $A\in\A$, and  $\coprod^1_IV=0$;
\item $V^{\perp_1}$ is a cogenerating subcategory of $\A$.
\end{enumerate}
If  these equivalent conditions hold, then $V$ is $\Ext^1$-universal.
\end{enumerate}
\end{proposition}
\begin{proof}
%We prove simultaneously the implications ``(1.2)$\Rightarrow$(1.1)'' and ``(2.2)$\Rightarrow$(2.1)''. Indeed, by Corollary \ref{coro_ext_univ_finite} and Example \ref{ex_exact_copr}, we  need to prove that $\Ext_\A^1(V,A)$ is a finitely generated $R$-module in part (1), and a set (as opposed to a proper class) in part (2), for all $A\in \A$.  To see this, using that $V^{\perp_1}$ is cogenerating, we choose an exact sequence $0\rightarrow A\to T_A\to B\rightarrow 0$, where $T_A\in V^{\perp_1}$. This gives an exact sequence 
We prove simultaneously the implications ``(1.2)$\Rightarrow$(1.1)'' and ``(2.2)$\Rightarrow$(2.1)''. Indeed, {by Example \ref{ex_exact_copr}, we know that $\coprod_I^1V=0$ for all sets $I$, when  condition (2.2) holds. Hence, from Corollary \ref{coro_ext_univ_finite} and Proposition \ref{prop_ext_univ_infinite}, we  only need to prove that $\Ext_\A^1(V,A)$ is a finitely generated $R$-module in part (1), and a set (as opposed to a proper class) in part (2), for all $A\in \A$.} To see this, using that $V^{\perp_1}$ is cogenerating, we choose an exact sequence $0\rightarrow A\to T_A\to B\rightarrow 0$, where $T_A\in V^{\perp_1}$. This gives an exact sequence 
\[
\A(V,T_A)\to\A(V,B)\to\Ext_\A^1(V,A)\rightarrow\Ext_\A^1(V,T_A)=0,
\] 
%and the result follows. Let us now verify the implications ``(1.1)$\Rightarrow$(1.2)'' and ``(2.1)$\Rightarrow$(2.2)''.  By Corollary \ref{coro_ext_univ_finite} and Proposition \ref{prop_ext_univ_infinite}, we know that, in both situations,  the universal extension of $V$ by $A$ exists, for all objects $A\in\mathcal{A}$. Let then 
and the result follows. Let us now verify the implications ``(1.1)$\Rightarrow$(1.2)'' and ``(2.1)$\Rightarrow$(2.2)''.  By Proposition \ref{prop_ext_univ_infinite}, we know that, in both situations,  the universal extension of $V$ by $A$ exists, for all objects $A\in\mathcal{A}$. Let then 

\[
0\rightarrow A\stackrel{u}{\longrightarrow}B\stackrel{p}{\longrightarrow}V^{(J)}\rightarrow 0
\] 
be  a universal extension, so that $\Ext_\mathcal{A}^1(V,p)$ is a monomorphism.  Then we have that $B\in V^{\perp_1}$ since $\Ext_\mathcal{A}^1(V,V^{(J)})=0$. \\ The final statement in part (2) follows by Proposition \ref{prop_ext_univ_infinite}.
%be  a universal extension, with $J$  finite if we are in the setting of part (1). Then,  $\Ext_\mathcal{A}^1(V,p)$ is a monomorphism, and so $B\in V^{\perp_1}$ since $\Ext_\mathcal{A}^1(V,V^{(J)})=0$. \\ 
%The final statement in part (2) follows by Proposition \ref{prop_ext_univ_infinite}.
\end{proof}

\subsection{Special preenvelopes and cotorsion pairs}\label{preenvelopes_main_subs}
In this subsection we study the relation between cotorsion pairs and semi-special preenveloping torsion classes. Let us start with the following technical lemmas: 

\begin{lemma} \label{lem.extension-in-reflective-subcats}
Let $\A$ be an Abelian category and   $\mathcal T$   a subcategory closed under extensions, such that $\B:=\Sub(\T)$ is closed under quotients in $\A$ (or, equivalently, it is an Abelian exact subcategory). Then, the following statements hold true:
\begin{enumerate}[\rm (1)]
\item for any object $X\in\gen (\T)={\Quot}(\T)$, the inclusion functor $\iota:\mathcal{B}\hookrightarrow\A$ induces an isomorphism $\Ext_\mathcal{B}^1(X,B)\cong\Ext_\A^1(X,B)$, for all $B\in\mathcal{B}$;
\item ${}^{\perp_1}\T\cap\mathcal{B}\subseteq \Ker(\Ext_\mathcal{B}^1(-,\T))$;
\item ${}^{\perp_1}\T\cap{\gen(\T)}=\Ker(\Ext_\mathcal{B}^1(-,\T))\cap{\gen(\T)}$;
\item a morphism $\mu \colon B\to T$, with $B\in\mathcal{B}$, is a special $\T$-preenvelope in $\mathcal{B}$ if, and only if, it is a semi-special $\T$-preenvelope in $\A$. 
\end{enumerate}
\end{lemma}
\begin{proof}
(1). For all $B,\, B'\in\mathcal{B}$, the canonical morphism $\Ext_\mathcal{B}^1(B',B)\to\Ext_\A^1(B',B)$ is a monomorphism. Let us show that it is also surjective when $B'=X\in\gen (\T)$. Let $0\rightarrow B\stackrel{u}{\longrightarrow}A\stackrel{p}{\longrightarrow}X\rightarrow 0$ be an exact sequence and consider first the case when $X=T\in\T$.  We fix a monomorphism $v\colon B\hookrightarrow T'$ with $T'\in\mathcal{T}$,  taking the pushout of $u$ and $v$ and, using that $\mathcal{T}$ is closed under extensions, we see that $A\in\mathcal{B}$, so that the above sequence also belongs in $\Ext_{\mathcal{B}}^1(T,B)$.
For a generic $X\in\gen (\T)={\Quot}(\T)$, we fix an epimorphism $\pi \colon\tilde{T}\to X$, with $\tilde{T}\in\T$. By taking the pullback of $p$ and $\pi$, we get an exact sequence $0\rightarrow B\to\tilde{B}\to\tilde{T}\rightarrow 0$ together with an epimorphism $\tilde{B}\twoheadrightarrow A$. By the previous paragraph we get that $\tilde{B}\in\mathcal{B}$, so that $A\in\mathcal{B}$, since $\mathcal{B}$ is closed under quotients.
% Hence, given $V\in\mathcal{T}$,
%\begin{equation}\label{Ext_orth_in_B}
%\mathcal{B}\cap V^{\perp_{1}}=\Ker(\Ext_\mathcal{B}^1(V,-)).
%\end{equation}
 
 \smallskip\noindent
(2). It follows from the fact that  the canonical morphism $\Ext_\mathcal{B}^1(B',B)\to\Ext_\A^1(B',B)$ is a monomorphism, for all $B,\, B'\in\mathcal{B}$.
%(2) follows from assertion (1).   
% By applying the right exact functor $\iota\circ\tau :\A\longrightarrow\A$, we get a commutative diagram with exact rows:
 
%$$\xymatrix{0 \ar[r] & B' \ar[r]^{u} \ar[d]_{\eta_{B'}}^{\wr}& A \ar[r]^p \ar[d]_{\eta_A} & B \ar[r] \ar[d]_{\eta_B}^{\wr} & 0\\ &(\iota \circ \tau)(B') \ar[r] & (\iota \circ \tau)(A) \ar[r] & (\iota \circ \tau)(B) \ar[r] & 0}$$
  
 % The outer vertical arrows are isomorphisms. Then an appropriate use of Ker-Coker lemma gives that $\eta_A$ is a monomorphism, and hence, by assertion 1, it is an isomorphism. This proves that $A\in\mathcal{B}$ and hence $\epsilon\in\Ext_\mathcal{B}^1(B,B')$. 
 
 \smallskip\noindent
(3). The inclusion ``$\subseteq$'' is a consequence of (2), while  ``$\supseteq$'' is a consequence of (1).%follows from the fact that $\T$ is closed under extensions in $\A$. 
%(3) The inclusion ``$\subseteq$'' is a consequence of (2), while the inclusion ``$\supseteq$'' follows from the fact that $\T$ is closed under extensions in $\A$. 
 
 \smallskip\noindent
(4). The morphism $\mu$ is a semi-special $\T$-preenvelope in $\A$  (resp., in $\mathcal{B}$) if, and only if, we have that $\Coker (\mu )\in {}^{\perp_1}\T\cap{\gen (\T)}$ (resp., $\in \Ker(\Ext_\mathcal{B}^1(-,\T))\cap{\gen (\T)}$). Now, apply part (3)  and use that each $\T$-preenvelope in $\B=\Sub(\T)$ is monomorphic.
\end{proof}

In the following example we show that there are some subtleties when taking $\Ext^1$-orthogonals in a subcategory:

\begin{example}
In the setting of Lemma \ref{lem.extension-in-reflective-subcats},  the inclusion $_{}^{\perp_1}\T\cap\mathcal{B}\subseteq \Ker(\Ext_\mathcal{B}^1(-,\T))$ may be strict.
\end{example}
\begin{proof}
Combining  \cite[Thm.\,6.2]{parra2016hearts} and its proof with Corollary 5.2 in [Op.Cit.], we obtain an example of a finite dimensional algebra $R$ with a finitely generated $R$-module $V$ such that the stalk complex $V[0]$ is a projective generator of the heart of the Happel-Reiten-Smal\o\ $t$-structure in $\Der(R)$ associated with the torsion pair $(\Gen(V),V^{\perp})$. Then, by \cite[Prop.\,3.8]{parra2020hrs}, we know that $V$ is quasi-tilting (see Section \ref{sec_tilt_pree}). Moreover, it is shown in [Op.Cit.] that the trace $\tr_V(R)=:I$ is a $2$-nilpotent ideal of $R$ that coincides with $\mathrm{ann}_R(V)$. Hence, $\mathcal{B}:=\subGen(V)=\mod{R/I}$, so that $R/I\in\Ker(\Ext_\mathcal{B}^1(-,\T))$. However, the canonical exact sequence $0\rightarrow I\hookrightarrow R\to R/I\rightarrow 0$ does not split in $\mod R$, thus proving that $R/I\not\in {}^{\perp_1}\T\cap\mathcal{B}$, since $I\in\Gen(V)=\T$. 
\end{proof}

\begin{lemma}\label{2-4implies1_lemma}
Let $\A$ be an Abelian category,  $\T$  a subcategory closed under  quotients and  {under taking coproducts in  $\B:=\Sub(\T)$},  and suppose that  $Q$ is an epi-generator of $\B$. 
% such that one of the following conditions is satisfied:
%\begin{enumerate}[(i)] 
%\item $\B=\gen(Q)$ (e.g., when $\A$ is $\hom$-finite);
%\item $\B$ is (Ab.3).
%\end{enumerate}
If there is a short exact sequence as follows
\[
\xymatrix{
0\ar[r]&Q\ar[r]^{\lambda}&T_Q\ar[r]&V_Q\ar[r]&0,
}
\]
such that $\lambda$ is a special $\T$-preenvelope of $Q$ in $\B$, then the following statements hold true:
\begin{enumerate}[\rm (1)]
\item $\mathrlap{\text{\tiny \hspace{3.8pt}$\B$}}\coprod_I \lambda: \mathrlap{\text{\tiny \hspace{3.8pt}$\B$}}\coprod_I Q \longrightarrow \mathrlap{\text{\tiny \hspace{3.8pt}$\B$}}\coprod T_Q$ is a special $\T$-preenvelope in $\B$, for each non-empty set $I$ for which the needed coproducts exist, where $\mathrlap{\text{\tiny \hspace{3.8pt}$\B$}}\coprod$ denotes the coproduct in $\B$;%$\mathrlap{\text{\tiny \hspace{4.5pt}$\B$}}\coprod_{I}\lambda\colon \mathrlap{\text{\tiny \hspace{4.5pt}$\B$}}\coprod_{I}Q\to T_Q^{(I)}$ is a special $\T$-preenvelope for each non-empty set $I$ for which the needed coproducts exist, where $\mathrlap{\text{\tiny \hspace{4.5pt}$\B$}}\coprod$ denotes a coproduct in $\B$;
\item $V_Q$ is $\Ext^1$-universal in $\B$;
\item $\mathcal{T}=\Ker(\Ext^1_\B(V_Q,-))=\Gen_{{\mathcal{B}}}(T_Q)$.
\end{enumerate}
\end{lemma}
\begin{proof}
(1) follows from Lemmas \ref{injective_preenvelope} and  \ref{lem.special-preenvelope-generator}. To prove parts (2) and (3) we will use the same construction. Indeed, given $B\in\B$,  there is a non-empty set $I$ and an epimorphism $\pi\colon \mathrlap{\text{\tiny \hspace{3.8pt}$\B$}}\coprod_I Q\twoheadrightarrow B$. We then get the following commutative diagram with exact rows:
%(1) follows from Lemma \ref{injective_preenvelope} and Lemma \ref{lem.special-preenvelope-generator}. To prove parts (2) and (3) we will use the same construction. Indeed, given $B\in\B$,  there is a non-empty set $I$ and an epimorphism $\pi\colon \mathrlap{\text{\tiny \hspace{4.5pt}$\B$}}\coprod_{I}Q\twoheadrightarrow B$. We then get the following commutative diagram with exact rows:
%\begin{equation}\label{universal_PO}
%\xymatrix@R=15pt@C=38pt{ 
%0 \ar[r] & \mathrlap{\text{\tiny \hspace{3.8pt}$\B$}}$\coprod_{I}Q$ \ar[r]^-{\mathrlap{\text{\tiny \hspace{2.5pt}$\B$}}$\coprod_I \lambda$} \ar@{>>}[d]_{$\pi$} &  \mathrlap{\text{\tiny \hspace{3.8pt}$\B$}}$\coprod T_{Q}$ \ar[r] \ar@{>>}[d]&  \mathrlap{\text{\tiny \hspace{3.8pt}$\B$}}$\coprod V_{Q}$ \ar[r] \ar@{=}[d] & 0 \\ 
%0 \ar[r] & B \ar[r] & T' \ar[r] \ar@{}[ul]|{\text{\tiny P.O.}} &  \mathrlap{\text{\tiny \hspace{3.8pt}$\B$}}$\coprod V_{Q}$ \ar[r] & 0}
%\end{equation}

\begin{equation}\label{universal_PO}
\xymatrix@R=15pt@C=38pt{ 
0 \ar[r] & \mathrlap{\text{\tiny \hspace{3.8pt}$\B$}}\coprod_{I}Q \ar[r]^-{\mathrlap{\text{\tiny \hspace{2.5pt}$\B$}}\coprod_I \lambda} \ar@{>>}[d]_{\pi} &  \mathrlap{\text{\tiny \hspace{3.8pt}$\B$}}\coprod T_{Q} \ar[r] \ar@{>>}[d]&  \mathrlap{\text{\tiny \hspace{3.8pt}$\B$}}\coprod V_{Q} \ar[r] \ar@{=}[d] & 0 \\ 
0 \ar[r] & B \ar[r] & T' \ar[r] \ar@{}[ul]|{\text{\tiny P.O.}} &  \mathrlap{\text{\tiny \hspace{3.8pt}$\B$}}\coprod V_{Q} \ar[r] & 0}
\end{equation}
%\xymatrix@R=15pt@C=38pt{ 
%0 \ar[r] & \mathrlap{\text{\tiny \hspace{4.pt}$\B$}}\coprod_{I}Q \ar[r]^{\hspace{0.4 cm}\mathrlap{\text{\tiny \hspace{2.5pt}$\B$}}\coprod_{I}\lambda} \ar@{>>}[d]_{\pi} & T_{Q}^{(I)} \ar[r] \ar@{>>}[d]& V_{Q}^{(I)} \ar[r] \ar@{=}[d] & 0 \\ 
%0 \ar[r] & B \ar[r] & B' \ar[r] \ar@{}[ul]|{\text{\tiny P.O.}} & V_{Q}^{(I)} \ar[r] & 0
%}
where the left square is a pushout, and the coproducts in the first row exist since $\B$ is closed under taking subobjects, so $\B=\Gen(Q)=\Pres(Q)$  and Lemma \ref{existence_of_coproducts_in_pres_lemma} applies. Now, since $\T$ is closed under coproducts in $\mathcal{B}$ and quotients and $T'\in \T$ we get that $\Ext^1_\B(V_Q,T')=0$. This implies that the bottom row is a universal extension in $\B$, thus proving (2). It remains to verify part (3): since $\lambda$ is a special $\T$-preenvelope of $Q$, we know that $\mathcal{T}\subseteq\Ker(\Ext^1_\B(V_Q,-))$. On the other hand, let $B\in \Ker(\Ext^1_\B(V_Q,-))$ and note that, when taking the diagram in \eqref{universal_PO},  the lower row splits by the choice of $B$ and Lemma~\ref{lem.Ext-V-versus-Ext-VI}(1). In particular, $B$ is an epimorphic image of  $\mathrlap{\text{\tiny \hspace{3.8pt}$\B$}}\coprod T_Q$, so that $B\in \Gen_{\mathcal{B}}(T_Q)$. Consequently, we obtain     $\Ker(\Ext^1_\B(V_Q,-)) \subseteq \Gen_{\mathcal{B}}(T_Q)$, and so assertion (3)  follows from the following inclusions: $\T \subseteq \Ker(\Ext^1_\B(V_Q,-))\subseteq \Gen_{\mathcal{B}}(T_Q)\subseteq \T.$
%since $\T$ is closed under taking quotients, $V_Q\in \Ker(\Ext_\mathcal{B}^1(-,\T))\cap\T$. Thus, $\mathcal{T}\subseteq\Ker(\Ext^1_\B(V_Q,-))$. Conversely, let $B\in \Ker(\Ext^1_\B(V_Q,-))$ and note that, when taking the diagram in \eqref{universal_PO},  the lower row splits by the choice of $B$ and Lemma~\ref{lem.Ext-V-versus-Ext-VI}(1). In particular, $B$ is an epimorphic image of $T_Q^{(I)}$ and, hence, $B\in\mathcal{T}$. This proves at once that $\mathcal{T}=\Ker(\Ext^1_\B(V_Q,-))$ and $\Ker(\Ext^1_\B(V_Q,-))\subseteq \Gen(T_Q)$; the converse inclusion  follows by the inclusions 
%$\Ker(\Ext^1_\B(V_Q,-))\subseteq \Gen(T_Q)\subseteq \T=\Ker(\Ext^1_\B(V_Q,-))$.
% As $\T$ is closed under coproducts, the coproduct of $\mathcal{T}$-preenvolopes in $\B$ is a $\mathcal{T}$-preenvolope in $\B$, and $\mathrlap{\text{\tiny \hspace{4.5pt}$\B$}}\coprod_{I}T_Q=T_Q^{(I)}$, so we have a $\mathcal{T}$-preenvelope $\mathrlap{\text{\tiny \hspace{4.5pt}$\B$}}\coprod_{I}\lambda\colon\mathrlap{\text{\tiny \hspace{4.5pt}$\B$}}\coprod_{I}Q\to T_Q^{(I)}$,
%which is necessarily monomorphic since $\mathrlap{\text{\tiny \hspace{4.5pt}$\B$}}\coprod_{I}Q$ embeds in an object of $\mathcal{T}$ (see Lemma \ref{injective_preenvelope}).
\end{proof}

We are now ready for the main result of this section:

\begin{theorem} \label{thm.specialpreenveloping-versus-cotorsionpair}
Let $\mathcal{A}$ be an Abelian category,  $\mathcal{T}$  a subcategory closed under extensions {and direct summands and such that $\mathcal{B}:=\Sub(\mathcal{T})$ is closed under quotients.} Consider the following assertions:
\begin{enumerate}[\rm (1)]
\item $\mathcal{B}$ is reflective in $\mathcal{A}$ and $\mathcal{T}=\mathcal{B}\cap V^{\perp_1}$ for some {object $V$}which is $\Ext^1$-universal in $\B$ and satisfies that $\mathrlap{\text{\tiny \hspace{3.8pt}$\B$}}\coprod_I V\in V^{\perp_1}$, for all sets $I$ for which $\mathrlap{\text{\tiny \hspace{3.8pt}$\B$}}\coprod_I V$ exists;
\item $\T$ is semi-special preenveloping in $\A$;
\item $\mathcal{B}$ is reflective in $\mathcal{A}$ and $\mathcal{T}$ is special preenveloping in $\mathcal{B}$;
\item $\B$ is reflective in $\A$ and $(\Ker(\Ext^1_\B(-, \mathcal T )),\mathcal T)$ is a right complete cotorsion pair in $\B$.
\end{enumerate}
Then, the implications ``(1)$\Rightarrow$(2)$\Leftrightarrow$(3)$\Leftrightarrow$(4)'' hold true.
{When, in addition, $\T$ is closed under  quotients,  $\mathcal{B}$ has an epi-generator and $\T$ is closed under coproducts in $\mathcal{B}$,}
  all  assertions are equivalent.

% such that one of the following conditions is satisfied:
%\begin{enumerate}[(i)] 
%\item $\B=\gen(Q)$ (e.g., when $\A$ is $\hom$-finite);
%\item $\B$ is (Ab.3);
%\end{enumerate} 
\end{theorem}
\begin{proof}
(2)$\Rightarrow$(3). It follows from Lemma \ref{lem.extension-in-reflective-subcats}(4) and Proposition \ref{ex.reflective-versus-productclosedness}.% Given $B\in\B$, consider a semi-special $\mathcal{T}$-preenvelope $\phi\colon B\to T_B$ in $\A$. Clearly, this is also a semi-special preenvelope in $\B$, as $\T$ is closed under quotients, and it is also special by Lemma \ref{injective_preenvelope}. The rest follows by Proposition \ref{ex.reflective-versus-productclosedness} (2).
%(2)$\Rightarrow$(3). Given $B\in\B$, consider a semi-special $\mathcal{T}$-preenvelope $\phi\colon B\to T_B$ in $\A$. Clearly, this is also a semi-special preenvelope in $\B$, as $\T$ is closed under quotients, and it is also special by Lemma \ref{injective_preenvelope}. The rest follows by Proposition \ref{ex.reflective-versus-productclosedness} (2).

\smallskip\noindent
(3)$\Rightarrow$(4). %By \eqref{Ext_orth_in_B}, ${}^{\perp_1}\mathcal{T}\cap \B$ is the left $\Ext$-orthogonal of $\mathcal T$ in $\B$. 
To show that $(\Ker(\Ext^1_\B(-, \mathcal T )),\mathcal T)$ is a cotorsion pair we have to show that, given $B\in \B$ such that $\Ext_\mathcal{B}^1(X,B)=0$, for all $X\in \Ker(\Ext^1_\B(-, \mathcal T ))$, then $B\in\mathcal{T}$. For this, take  a special $\mathcal T$-preenvelope $\mu_{B}\colon B\to T_B$, so that $\Coker(\mu_{B})\in\Ker(\Ext^1_\B(-, \mathcal T ))$. Thus, the following sequence is split exact
\[
0\rightarrow B\stackrel{\mu_{B}}{\longrightarrow}T_{B}\longrightarrow\Coker(\mu_{B})\rightarrow 0.
\]
Hence $B$ is a summand of $T_B$ and so it belongs to $\mathcal T$. 
%To show that the cotorsion pair $({}^{\perp_1}\mathcal{T}\cap \B,\mathcal T)$ is right complete one proceeds similarly: given $B'\in \B$, one can consider a short exact sequence $0\to B'\overset{\mu_{B'}}\longrightarrow T_{B'}\to C\to 0$ where $\mu_{B'}$ is a special $\mathcal T$-preenvelope, and so $T_{B'}\in \mathcal T$ while $C\in {}^{\perp_1}\mathcal{T}\cap \B$.  

\smallskip\noindent
(4)$\Rightarrow$(2). Let $\tau\colon\mathcal{A}\to\mathcal{B}$ be the left adjoint to the inclusion $\iota\colon \mathcal{B}\hookrightarrow\mathcal{A}$, and denote by $\eta\colon \id_\mathcal{A}\Rightarrow \iota\circ\tau$ the unit of the adjunction. Then,  $\eta$ is an epimorphism by Lemma \ref{reflector_is_epi_lemma}.
Take now an object $T\in\mathcal{T}$ and an object $A\in \A.$ Consider the natural isomorphism
\[
\eta_A^*:=\mathcal{A}(\eta_A,\iota (T))\colon \mathcal{A}(\iota(\tau(A)),\iota(T))\tilde{\longrightarrow}\mathcal{A}(A,\iota (T)).
\] 
It shows that any morphism $f\colon A\to T$ in $\A$ factors through $\eta_A$. Thus, to prove that $A$ has a $\mathcal{T}$-preenvelope, it is enough to check that any object of $\mathcal{B}$ has a $\mathcal{T}$-preenvelope. But this is a direct consequence of the right completeness of the cotorsion pair $(\Ker(\Ext^1_\B(-, \mathcal T )),\mathcal T)$ (see Example \ref{ex_cotorsion_implies_cover}). Furthermore, for each short exact sequence
\[
0\rightarrow (\iota\circ\tau)(A)\stackrel{\lambda_A}{\longrightarrow} T_A\stackrel{\pi_A}{\longrightarrow}X_A\rightarrow 0,
\] 
with $T_A\in\mathcal{T}$ and $X_A\in \Ker(\Ext^1_\B(-, \mathcal T ))$, the map $\mu_A:=\lambda_A\circ\eta_A\colon A\to T_A$ is a $\mathcal{T}$-preenvelope such that  $\Coker(\mu_A)=\Coker(\lambda_A)\in \Ker(\Ext^1_\B(-, \mathcal T ))\cap{\gen}(\mathcal T)\subseteq {}^{\perp_1}\mathcal{T}$ (see Lemma \ref{lem.extension-in-reflective-subcats}(3)).
%To prove that $\mu_A$ is actually a special $\mathcal{T}$-preenvelope in $\mathcal{A}$ we still need to check that $V_A\in {}^{\perp_1}\mathcal{T}=\Ker(\Ext_\mathcal{A}^1(?,\mathcal{T}))$.  But we have $V_A\in\mathcal{T}$ since $\mathcal{T}$ is closed under taking quotients. By the first paragraph of this proof, we then get that $\Ext_\mathcal{A}^1(V_A,T)=\Ext_\mathcal{B}^1(V_A,T)=0$, for all $T\in\mathcal{T}$. 

\smallskip\noindent
(1)$\Rightarrow$(3). %First note that $V=\mathrlap{\text{\tiny \hspace{3.8pt}$\B$}}\coprod_{I} V\in V^{\perp_1}$, where $|I|=1$, and so $V\in \mathcal{B}\cap V^{\perp_1}=\Tcal$. 
%Let us prove that any  $B\in\mathcal{B}$ has a special $\mathcal{T}$-preenvelope in $\mathcal{B}$. 
Note that  $\mathrlap{\text{\tiny \hspace{3.8pt}$\B$}}\coprod_I V\in\mathcal{T}$, for any set $I$ for which the coproduct exists, and then, by Lemma~\ref{lem.extension-in-reflective-subcats}(1), we have a natural isomorphism  $\Ext_\mathcal{B}^1(\mathrlap{\text{\tiny \hspace{3.8pt}$\B$}}\coprod_I V,-)\cong\Ext_\mathcal{A}^1(\mathrlap{\text{\tiny \hspace{3.8pt}$\B$}}\coprod_I V,-)_{\restriction\mathcal{B}}$ of functors $\mathcal{B}\to\Ab$. Consider a universal extension:
%(1)$\Rightarrow$(3). We  need to prove that each $B\in\mathcal{B}$ has a special $\mathcal{T}$-preenvelope in $\mathcal{B}$. So let $B\in \B$ and  consider a universal extension in $\mathcal{B}$:
\[
\xymatrix{
0\rightarrow B\stackrel{\lambda_B}{\longrightarrow}T_B\stackrel{p_B}{\longrightarrow}\mathrlap{\text{\tiny \hspace{3.8pt}$\B$}}\coprod_J V\rightarrow 0\qquad\text{ in $\mathcal{B}$}.}
\] 
By Definition \ref{def.universal extension} and the mentioned natural isomorphism, the map 
\[
\xymatrix{
\Ext_\mathcal{A}^1(V,p)\colon \Ext_\mathcal{A}^1(V,T_B)\hookrightarrow\Ext_\mathcal{A}^1(V,\mathrlap{\text{\tiny \hspace{3.8pt}$\B$}}\coprod_J V)=0
}
\]
is injective. Therefore $T_B\in\mathcal{B}\cap V^{\perp_1}=\mathcal{T}$. Hence, $\lambda_B$ is a special $\mathcal{T}$-preenvelope in $\mathcal{B}$ since $\Ext_\B^1(\mathrlap{\text{\tiny \hspace{3.8pt}$\B$}}\coprod_J V,-)$ vanishes on $\T$ (see Lemma \ref{lem.Ext-V-versus-Ext-VI}). 

\smallskip\noindent
The implication ``(2--4)$\Rightarrow$(1)'', under the stronger assumptions, is a consequence of Lemma~\ref{2-4implies1_lemma}, using also Lemma \ref{lem.extension-in-reflective-subcats}(1) to see that $\Ker (\Ext_\mathcal{B}^1(V_Q,-))=V_Q^{\perp_1}\cap\mathcal{B}$, since $V_Q\in\T$.
\end{proof}

\begin{remark} \label{rem.to-main-theorem}
Let $\A$ be an Abelian category with a (projective) epi-generator $G$. If $\mathcal{B}$ is a reflective subcategory of $\A$, and $\tau\colon\A\rightarrow\mathcal{B}$ is the left adjoint to the inclusion $\iota \colon\mathcal{B}\hookrightarrow\A$, then $\tau (G)$ is a (projective) epi-generator of $\mathcal{B}$. In this case, if $\T\subseteq\mathcal{B}$ is a subcategory which is closed under coproducts in $\A$, then it is also closed under coproducts in $\mathcal{B}$.

In particular, if $\A$ is an Abelian category with an epi-generator and $\T$ is a subcategory closed under coproducts, quotients and extensions (e.g., a torsion class), then $\T$ satisfies one of the assertions of the above theorem if, and only if, it satisfies all the others.
\end{remark}

%
%Then assertions (2), (3) and (4) are equivalent. Furthermore, (1) implies (2)--(4)  if $\A$ satisfies either of the following sets of conditions:
%\begin{enumerate}[(i)]
%\item $\A$ is $\hom$- and $\Ext^1$-finite; 
%\item $\A$ is $\Ext^1$-small and (Ab.3). 
%\end{enumerate}

If in the above theorem the class $\T$ is cogenerating in $\A$, that is, $\A=\Sub (\T)$ then some of the conditions get simplified and we get the following straightforward corollary:

\begin{corollary} \label{cor.special-preenveloping-arbitraryA}
Consider the  following assertions for an Abelian category $\A$ and  a subcategory $\T$:  
\begin{enumerate}[\rm (1)]
\item $\T=V^{\perp_1}$ for some object $V$ that is  $\Ext^1$-universal in $\A$ and satisfies that $V^{(I)}\in V^{\perp_1}$, for all sets $I$ for which the coproduct exists in $\A$;
\item $\T$ is special preenveloping in $\A$ and it is closed under extensions and direct summands;
\item $({}^{\perp_1}\T,\T)$ is a right complete cotorsion pair in $\A$. 
\end{enumerate}
The implications ``(1)$\Rightarrow$(2)$\Leftrightarrow$(3)'' hold true. When, in addition,   $\A$ has an epi-generator and $\T$ is closed under coproducts and quotients (e.g., when $\T$ is a torsion class), all assertions are equivalent.
\end{corollary}

Our final result in this section shows that preenveloping hereditary torsion classes are familiar:

\begin{corollary} \label{cor.special-preenveloping-hereditary}
Let $\mathcal{A}$ be an Abelian category and $\mathcal{T}$  a subcategory. The following  are equivalent:
\begin{enumerate}[\rm (1)]
\item $\mathcal{T}$ is a (semi-special) preenveloping hereditary torsion class;
\item $\mathcal{T}$ is a TTF class in $\mathcal{A}$;
\item $\mathcal{T}$ is a (semi-special) precovering cohereditary torsion-free class.
\end{enumerate}
\end{corollary}
\begin{proof}
We just prove the equivalence ``(1)$\Leftrightarrow$(2)'' since the equivalence ``(2)$\Leftrightarrow$(3)'' will then follow by duality.
Both (1) and (2) imply that $\T=\Sub(\T)$ and $\T$ is a hereditary torsion class, something that we assume in the sequel. According to Theorem~\ref{thm.specialpreenveloping-versus-cotorsionpair} and Proposition \ref{ex.reflective-versus-productclosedness}, the subcategory $\mathcal{T}$ is (semi-special) preenveloping if and only if it is reflective in $\mathcal{A}$. The task is hence reduced to check that, if $\mathcal{T}$ is a reflective subcategory, then it is a TTF class. Suppose then that $\mathcal{T}$ is reflective and let $\sigma\colon \mathcal{A}\to\mathcal{T}$ be the left adjoint to the inclusion $\iota\colon \mathcal{T}\hookrightarrow\mathcal{A}$. By Lemma \ref{reflector_is_epi_lemma}, the unit $\eta\colon \id_\mathcal{A}\to\iota\circ\sigma$ is an epimorphism and the induced map $\mathcal{A}(\eta_A,T)\colon\mathcal{A}(\sigma(A),T)\to\mathcal{A}(A,T)$ is an isomorphism for all $A\in \A$ and $T\in \mathcal T$. Given an object $A\in\mathcal{A}$, consider the exact sequence 
\[
0\rightarrow c(A)\stackrel{j_A}{\longrightarrow} A\stackrel{\eta_A}{\longrightarrow}(\iota\circ\sigma )(A)\rightarrow 0,
\] 
where $c(A):=\Ker(\eta_A)$. We shall prove that $c(A)\in {}^{\perp}\mathcal{T}$, from which we immediately conclude that $({}^{\perp}\mathcal{T},\mathcal{T})$ is a torsion pair in $\mathcal{A}$. Indeed, let $f\colon c(A)\to T$ be a morphism with $T\in\mathcal{T}$ and take the pushout of $j_A$ and $f$:
\[
\xymatrix@R=20pt{
0\ar[r]&c(A)\ar@{}[dr]|{\text{\tiny P.O.}}\ar[r]^{j_A}\ar[d]_{f}&A\ar[r]^(.4){\eta_A}\ar[d]^g&(\iota\circ \sigma)(A)\ar[r]\ar@{=}[d]&0\\
0\ar[r]&T\ar[r]_v&T'\ar[r]_(.4){\omega}&(\iota\circ\sigma)(A)\ar[r]&0
}
\]
Note that $T'$ is in $\mathcal{T}$ since $\mathcal T$ is closed under extensions. In particular, since $\mathcal{A}(\eta_A,T')$ is an isomorphism, there exists a morphism $\widehat {g}\colon (\iota\circ\sigma)(A)\to T'$ such that $g=\widehat {g}\circ\eta_A$. Hence, we get that \mbox{$v\circ f=g\circ j_A=\widehat {g}\circ\eta_A\circ j_A=0$} and, being $v$ a monomorphism, we conclude that $f=0$. 
\end{proof}

Let us conclude the section with the following remark:

\begin{remark}\label{remark_complete_cotorsion}
In the setting of Theorem \ref{thm.specialpreenveloping-versus-cotorsionpair}, suppose that $\mathcal{B}$ is a reflective subcategory of $\A$ such that $\mathcal{B}$ has an epi-generator. Then, the cotorsion pair $(\Ker(\Ext^1_\B(-, \mathcal T )),\mathcal T)$ in part (4) of the theorem is complete if, and only if, such an epi-generator can be chosen in $\Ker(\Ext^1_\B(-,\T))$. 
\end{remark}
\begin{proof}
Suppose first that the cotorsion pair is complete and fix an epi-generator $Q$ of $\mathcal{B}$. Then we have an exact sequence $0\rightarrow T_Q\to W_Q\to Q\rightarrow 0$ in $\mathcal{B}$, with $T_Q\in\T$ and $W_Q\in\Ker(\Ext^1_\B(-, \mathcal T ))$. It follows that $W:=W_Q$ is an epi-generator of $\mathcal{B}$ belonging to $\Ker(\Ext^1_\B(-,\T))$. Conversely, assume that $Q\in\Ker(\Ext^1_\B(-,\T))$. The same argument in the proof of assertion (5) in \cite[Thm.\,2.13]{SS} also works here since, by Lemma~\ref{lem.Ext-V-versus-Ext-VI}, we know that the class $\Ker(\Ext^1_\B(-,\T))$ is closed under taking coproducts in $\mathcal{B}$. Then the cotorsion pair $(\Ker(\Ext^1_\B(-,\T)),\mathcal{T})$ is complete.
\end{proof}

\section{(Quasi-)Tilting preenvelopes}\label{sec_tilt_pree}

In module categories, quasi-tilting torsion pairs are very much related to preenveloping torsion classes (see \cite{hugel2001tilting}). A similar relation  still holds in our general setting, but with some more subtleties. Let us start by defining (quasi-)tilting objects in Abelian categories:

\begin{definition} \label{def.quasi-tilting torsion pair}
Let $\mathcal{A}$ be an Abelian category. A torsion pair $\mathbf{t}=(\mathcal{T},\mathcal{F})$ is called \textbf{quasi-tilting} when there is an object $V$ such that the following condition holds:
\begin{enumerate}[\rm (QT)]
\item $\mathcal{T}=\Gen(V)=\Pres(V)=\overline{\Gen}(V)\cap V^{\perp_1}$.
\end{enumerate} 
In this case $V$ is said to be a {\bf quasi-tilting} object. Furthermore, the torsion pair $\mathbf{t}$ is \textbf{tilting} when 
\begin{enumerate}[\rm (T.1)]
\item there is a quasi-tilting object $V$ such that  $\mathcal{T}=\Gen(V)$;
\item $\T$ is cogenerating.
\end{enumerate}
In this case, $V$ is said to be a (1-){\bf tilting object}.
\end{definition}

Note that, for a tilting object $V$, one easily deduces from the definition that:
\[
\Gen(V)=\Pres(V)= V^{\perp_1}.
\]
and this is a torsion class. Anyway, there are some subtleties in the above definition that we discuss in the following remark:

\begin{remark} \label{rem.usual definition of quasi-tilting}
The usual way to define quasi-tilting objects in categories of modules is the following: given a ring $R$, a right $R$-module $V$ is quasi-tilting provided 
\begin{equation}\label{QT_def_in_modules}
\Gen(V)=\overline{\Gen}(V)\cap V^{\perp_1}.
\end{equation}
In that particular setting, it is then a consequence of \eqref{QT_def_in_modules} that $\Gen(V)$ is a torsion class, and that the equality $\Pres(V)=\Gen(V)$ holds. 
On the other hand, when working in an arbitrary Abelian category $\A$, if we take an object $V\in \A$ that just satisfies the equality \eqref{QT_def_in_modules}, we cannot guarantee in general that $\Gen(V)$ is a torsion class  (\textcolor{black}{see Example~\ref{tilt_not_torsion_ex} below}) nor the equality $\Gen(V)=\Pres(V)$. That is our motivation for explicitly including these facts in the above definition. However, the definition can be given as in modules
 in many Abelian categories that appear ``in nature'' (see Corollary  \ref{cor.different-definitions-tilting}).% as a consequence of Propositions \ref{prop.properties of quasi-tilting}(1) and \ref{prop.Ext-universal tilting} below, both facts are guaranteed when the subcategory $\mathrm{sum}(V)$ of (existing) coproducts $V^{(I)}$ is precovering in $\A$. This includes  when $\A$ is either (Ab3) or  $\Hom$-finite, in which case $V$ is a quasi-tilting object if, and only if, $\Gen(V)=\overline{\Gen}(V)\cap V^{\perp_1}$.
\end{remark}

\textcolor{black}{
\begin{example}\label{tilt_not_torsion_ex}
\textcolor{black}{Let $\G$ be a locally coherent Grothendieck category, $V$ a finitely presented tilting object in $\G$, and $(\Gen(V),V^{\perp})$ the associated torsion pair. Letting $\A:=\fp(\G)$, we have that $V\in \A$ and $\Gen_\A(V)=\ker(\Ext^1_\A(V,-))$ but there are cases in which $\Gen_\A(V)$ is not a torsion class in $\A$.}
\end{example}
\begin{proof}
%\textcolor{black}{A situation of this kind is described in \cite[Proposition 8.19]{parra2019locally}, let us briefly recall it here. If $R$ is a left semihereditary ring that is not right coherent (see \cite[Example 8.18]{parra2019locally}) for an explicit example) then the class $\Fcal$ of flat objects in $\mod R$ is a generating torsion-free class \textcolor{black}{and the associated torsion pair $\t:=(\T={}^\perp\Fcal, \Fcal)$ in $\mod R$ restricts to $\fpmod R$}. Let $\tau_\t=(\Der^{\leq -1}(\mod R))*\T[0], \Fcal[1] *\Der^{\geq 0}(\mod R)$ be the associated Happel-Reiten-Smal\o \ $t$-structure in $\Der(\mod R)$, and $\G:=\Ht=\Fcal[1]*\T[0]$ be its heart\textcolor{black}{, which is a locally coherent Grothendieck category by \cite[Theorem 7.3]{parra2019locally}}. Then $V:=R[1]$ is a finitely presented tilting object in $\G$ \textcolor{black}{with associated torsion pair $\mathbf{t}':=(\Fcal[1],\T[0])$}. If $\Gen_{\A}(V)$ were a torsion class in $\A$, then $\mathbf{t}'$ would restrict to $\A$ \textcolor{, which} }
A situation of this kind is described in \cite[Proposition 8.19]{parra2019locally}, let us briefly recall it here. If $R$ is a left semihereditary ring that is not right coherent (see \cite[Example 8.18]{parra2019locally} for an explicit example) then the class $\mathcal F$ of flat objects in $\mod R$ is a generating torsion-free class \textcolor{black}{and the associated torsion pair $\t:=(\T:={}^\perp\mathcal F,\mathcal F)$ in $\mod R$ restricts to $\fpmod R$}. Consider now  the associated Happel-Reiten-Smal\o \ $t$-structure   $\tau_\t=(\Der^{\leq-1}(\mod R)*\T[0],\mathcal F[1]*\Der^{\geq0}(\mod R))$ in $\Der(\mod R)$, and let $\G:=\Ht=\mathcal F[1]*\T[0]$ be its heart\textcolor{black}{, which is a locally coherent Grothendieck category by  \cite[Theorem 7.3]{parra2019locally} }. Then $V:=R[1]$ is a finitely presented tilting object in $\G$ \textcolor{black}{whose associated torsion pair is $\mathbf{t}':=(\Fcal[1],\T[0])$}. If $\Gen_{\A}(V)$ were a torsion class in $\A$, then $\mathbf{t}'$ would restrict to $\A$\textcolor{black}{, which, again by \cite[Theorem 7.3]{parra2019locally}, would imply that $\mod R$ is locally coherent}.  This is false since  $R$ is not right coherent by assumption.
%%%%%%%%%%%%%%AQUI%%%%%%%%%%%%%%%%%%%%%%%%%%%%%%%%%%%%%%%%%%%%%%%
 %Furthermore, the associated torsion pair $\t':=(\mathcal F[1],\T[0])$ does not restrict to $\A:=\fp(\G)$, that is, $\Gen_\A(V)=\Gen(V)\cap \A$ is not a torsion class in $\A$ because, otherwise, $\H_{\t'}\cong \mod R$ would be locally coherent (by \cite[Corollary~7.4]{parra2019locally}) and we have assumed that $R$ is not right coherent.
\end{proof}
}

\begin{proposition} \label{prop.properties of quasi-tilting}
Let $\A$ be an Abelian category, $V$ a quasi-tilting object of $\A$ and $\T=\Gen(V)$ the associated torsion class. The following assertions hold:
\begin{enumerate}[\rm (1)]
\item $\mathcal{B}:=\Sub(\T)$ is an Abelian exact subcategory of $\A$ and $V$ is a tilting object of $\B$;
\item $\Add(V)={}^{\perp_1}\T\cap\T$;
\item if $V'$  is a second quasi-tilting object in $\A$, the following are equivalent:
\begin{enumerate}[\rm (3.1)]
\item $\Gen(V)=\Gen(V')$;
\item $\Add(V)=\Add(V')$.
\end{enumerate}
\end{enumerate}
\end{proposition}
\begin{proof}
(1). By Example \ref{ex.Sub(X)-Gen(X) Abelian exact}, we know that $\B$ is an Abelian exact subcategory. Furthermore, the pair $(\T,\T^{\perp}\cap \B)$ is a torsion pair in $\B$ and the coproducts of copies of $V$ in $\A$ and in $\B$ are the same, so that, $\Pres_{\B}(V)=\Pres(V)$. On the other hand, by Lemma \ref{lem.extension-in-reflective-subcats}(1) and condition (QT), we have that $\Ker (\Ext_\mathcal{B}^1(V,-))=\mathcal{B}\cap V^{\perp_1}=\T$. It follows that $\T=\Pres(V)=\Pres_{\B}(V)\subseteq\Gen_\B(V)=\Gen(V)=\Ker (\Ext_\mathcal{B}^1(V,-))=\T$, so that $V$ is a tilting object of $\B$. 
%(1). By Example \ref{ex.Sub(X)-Gen(X) Abelian exact}, we know that $\B$ is an Abelian exact subcategory. Furthermore, the coproducts of copies of $V$ in $\A$ and in $\B$ are the same, so that, $\Pres_{\B}(V)=\Pres_{\A}(V)$. On the other hand, by Lemma \ref{lem.extension-in-reflective-subcats}(1) and condition (QT), we have that $\Ker (\Ext_\mathcal{B}^1(V,-))=\mathcal{B}\cap V^{\perp_1}=\T$. It follows that $\Pres_{\B}(V)=\Gen_\B(V)=\Gen(V)=\Ker (\Ext_\mathcal{B}^1(V,-))$, so that $V$ is a tilting object of $\B$. 

\smallskip\noindent
(2). Since $V$ is a tilting object of $\B$, we have that $\Add(V)\subseteq \Ker (\Ext_\mathcal{B}^1(V,-))\cap\T={}^{\perp_1}\T\cap\T$ (see  Lemma \ref{lem.extension-in-reflective-subcats}(3)). On the other hand, if $X\in {}^{\perp_1}\T\cap\T$ then, due to the equality $\Pres(V)=\T$, we have an exact sequence $0\rightarrow T\to V^{(I)}\to X\rightarrow 0$, for some set $I$ and some $T\in\Gen(V)=\T$. But this sequence splits since $X\in {}^{\perp_1}\T$, so that $X\in\Add(V)$.  %When in addition $V$ is $\Ext^1$-universal in $\mathcal{B}$ and we take $X\in {}^{\perp_1}\T\cap\T$, then, due to Proposition \ref{prop.Ext-universal tilting} below, we have that $\T=\Pres_\mathcal{B}(V)$, so that we have an exact sequence $0\rightarrow T\to V^{(I)}\to X\rightarrow 0$, for some set $I$ and some $T\in\Gen(V)=\T$. But this sequence splits since $X\in_{}^{\perp_1}\T$, so that $X\in\Add(V)$. 

\smallskip\noindent
(3). It follows  from assertion (2).
\end{proof}

The above proposition motivates us to give the following definition:
\begin{definition}
Let $\A$ be an Abelian category. Then, two (quasi-)tilting objects $V$ and $V'$ in $\A$ are said to be {\bf equivalent} provided $\Add(V)=\Add(V')$.
\end{definition}

\subsection{Properties of tilting objects}\label{ext_universal_tilting_objects_Sec}

In view of Theorem \ref{thm.specialpreenveloping-versus-cotorsionpair} and Proposition \ref{prop.properties of quasi-tilting}(1), it is natural to ask  whether tilting objects are $\Ext^1$-universal. In fact, this is true in full generality, as we show in the following lemma:

\begin{lemma}\label{tilt_implies_ext_univ_lemma}
Let $\A$ be an Abelian category and $V\in \A$ a tilting object. Then, $V$ is $\Ext^1$-universal. Furthermore, given $A\in \A$ and a universal extension $0\to A\to T\to V^{(I)}\to 0$ of $V$ by $A$, we have that $T\in V^{\perp_1}=\T$. 
\end{lemma}
\begin{proof}
Let $A$ be an object in $\A$. From the fact that $\Gen(V)$ is a cogenerating class, we obtain an exact sequence in $\A$ of the form:
\[
0\rightarrow A\stackrel{v}{\longrightarrow}T\stackrel{p}{\longrightarrow}T'\rightarrow 0,
\] 
for some $T\in \Gen(V)$. Now, since $\Gen(V)$ is a torsion class, also $T'\in \Gen(V)=\Pres(V)$, so that, there is an epimorphism $q\colon V^{(I)} \twoheadrightarrow  T'$  in $\A$ with $\Ker(q)\in \Gen(V)$, for some $I$ set. Hence, we obtain the following commutative diagram in $\A$:
\[
\xymatrix@C=50pt{ & & \Ker(q) \ar@{^(->}[d] \ar@{=}[r] & \Ker(q) \ar@{^(->}[d] \\
0 \ar[r] & A \ \ar@{=}[d] \ar@{^(->}[r] & Z \ar@{>>}[r] \ar@{}[dr]|{\text{P.B.}} \ar@{>>}[d] & V^{(I)} \ar[r] \ar@{>>}[d]^{q} & 0 \\ 
0 \ar[r] & A \  \ar@{^(->}[r]_v  & T  \ar@{>>}[r]_p  & T' \ar[r] & 0
}
\]
From the fact that $\Gen(V)$ is a torsion class in $\A$, we obtain that $Z\in \Gen(V)=V^{\perp_1}$, as it is an extension of $T$ and $\ker(q)$ and both objects belong in $\Gen(V)$. Hence, the second row in the above diagram is a universal extension of $V$ by $A$. 

For the final statement, let $A\in \A$ and consider a universal extension $0\to A\to T\to V^{(I)}\to 0$ of $V$ by $A$. Then, applying $\A(V,-)$, one gets:
\[
0\to \A(V,G)\to \A(V,T)\to \A(V,V^{(J)}) \twoheadrightarrow\Ext_\A^1(V,G)\overset{0}\longrightarrow \Ext_\A^1(V,T)\hookrightarrow \Ext_\A^1(V,V^{(J)})=0,
\]
so  $T\in V^{\perp_1}=\T$. 
\end{proof}

We can now reformulate the definition of tilting object in an Abelian category. In particular, when our ground category $\A$ has an epi-generator, we can characterize tilting objects by properties that are very similar to those usually taken as a definition in categories of modules.

\begin{proposition} \label{prop.Ext-universal tilting}
Let $\A$ be an Abelian category, $V\in \A$ and consider the following assertions:
\begin{enumerate}[\rm (1)]
%\item[(1)] \textcolor{black}{$V$ is an $\Ext^1$-universal object such that $\Gen(V)=V^{\perp_1}$ and the latter is a torsion class;}
\item $V$ is a tilting object;
\item the following conditions hold:
\begin{enumerate}[\rm (2.1)]
\item $\Gen(V)=V^{\perp_1}$;
\item $\Gen(V)$ is   cogenerating;
\item $\Add(V)$ is a precovering class in $\A$;
\end{enumerate}
\item the following conditions hold:
\begin{enumerate}[\rm (3.1)]
\item $\Ext_\A^1(V,V^{(I)})=0$, for all sets $I$ such that $V^{(I)}$ exists in $\A$;
\item $\pd_\A(V)\leq 1$;
\item $\Add(V)$ is a precovering class in $\A$;
\item there is an epi-generator $G$ of $\A$ and an exact sequence 
\[
0\rightarrow G\to V_0\to V_1\rightarrow 0\quad\text{such that}\quad V_0,\, V_1\in\Add(V),
\] 
and all coproducts  of this sequence that exist in $\A$ are exact.
\end{enumerate}
\end{enumerate}
Then, the equivalence ``(1)$\Leftrightarrow$(2)'' holds and, if $\A$ has an epi-generator, all assertions are equivalent.
\end{proposition}
\begin{proof}
(1)$\Rightarrow$(2). Let $A\in \A$ and consider the following short exact sequence
\[
0\to T\overset{\iota}\longrightarrow A\longrightarrow A/T\to 0
\]
where $T\in \T:=\Gen(V)$ and $A/T\in\mathcal{F}$. Using that $\T=\Pres(V)$, choose an exact sequence 
\begin{equation*}
0\rightarrow T'\longrightarrow V^{(J)}\stackrel{q_T}{\longrightarrow}T\rightarrow 0,\qquad \text{where }T'\in\T.
\end{equation*}
We get an epimorphism $\A(X,V^{(J)})\stackrel{\A(X,q_T)}{\longrightarrow}\A(X,T)\to\Ext_\A^1(X,T')=0$, for all \mbox{$X\in \Add(V)\subseteq {}^{\perp_1}\T$,} which shows that $q_T$ is a $\Add(V)$-precover of $T$. Now let $q_A:=\iota \circ q_T\colon V^{(J)}\to T\hookrightarrow A$. Then, for each $X\in \Add(V)$, the map $\A(X,q_A)=\A(X,\iota)\circ \A(X,q_T)$ is surjective as $\A(X,q_T)$ is surjective and $\A(X,\iota)$ is an isomorphism. Hence, $q_A$ is the desired $\Add(V)$-precover.

\smallskip\noindent
(2)$\Rightarrow$(1). We first prove that $\T=\Gen(V)$ is a torsion class. Let $A\in\A$ and fix an $\Add(V)$-precover  $p\colon X\to A$. It easily follows that $\Im(p)=\tr_V(A)$ is the trace of $V$ in $A$. We then get a short exact sequence 
\[
0\rightarrow \tr_V(A)\stackrel{\iota}{\longrightarrow}A\stackrel{\pi}{\longrightarrow}A/\tr_V(A)\rightarrow 0.
\] 
Applying the functor $\A(V,-)$, we then get the following exact sequence 
\[
0\rightarrow\A(V,\tr_V(A))\stackrel{\cong}{\longrightarrow}\A(V,A)\stackrel{0}{\longrightarrow}\A(V,A/\tr_V(A))\longrightarrow\Ext_\A^1(V,\tr_V(A))=0,
\] 
where the last term is trivial since $\tr_V(A)\in\Gen(V)=V^{\perp_1}$. Hence, $A/\tr_V(A)\in V^\perp=\Gen(V)^\perp$,  showing that $(\Gen(V),V^\perp)$ is a torsion pair. It  remains to check that $\Gen(V)=\Pres(V)$. Indeed, let $T\in\T$, choose an (epimorphic) $\Add(V)$-precover $q\colon X\to T$, take the following exact sequence
\begin{equation}\label{presentation_wannabe_eq}
0\to \ker(q)\longrightarrow X\longrightarrow T\to0,
\end{equation}
and apply $\A(V,-)$ to get  the following induced long exact sequence:
\[
0\to \A(V,\ker(q))\to \A(V,X)\twoheadrightarrow \A(V,T)\overset{0}\longrightarrow \Ext_\A^1(V,\ker(q))\hookrightarrow \Ext_\A^1(V,X)\to \cdots,
\] 
where the map $\A(V,q)\colon \A(V,X)\twoheadrightarrow \A(V,T)$ is surjective by definition of precover, so the connecting map  $\A(V,T)\to \Ext_\A^1(V,\ker(q))$ is necessarily trivial, showing that there is an inclusion $\Ext_\A^1(V,\ker(q))\hookrightarrow \Ext_\A^1(V,X)$. But in fact $\Ext_\A^1(V,X)=0$ as $X\in \Add(V)\subseteq {}^{\perp_1}\T$, so also $\Ext_\A^1(V,\ker(q))=0$, showing that $\ker(q)\in V^{\perp_1}=\Gen(V)$. Hence, the short exact sequence in \eqref{presentation_wannabe_eq}, shows that $T\in \Pres(V)$.

\medskip
{\em We assume in the sequel that $G$ is an epi-generator of $\A$.}

\smallskip\noindent
(1,2)$\Rightarrow$(3).  Conditions (3.1) and (3.3) are clear, while (3.2) follows by Corollary \ref{coro_proj_dim}. As for (3.4), let us start with a universal extension of $V$ by $G$ (which exists by Lemma \ref{tilt_implies_ext_univ_lemma})
\[
0\rightarrow G\stackrel{u}{\longrightarrow} T\stackrel{p}{\longrightarrow} V^{(J)}\rightarrow 0,
\]
with $T\in V^{\perp_1}=\T$. Now, given an epimorphism $q\colon V^{(I)}\twoheadrightarrow T$, put $\widehat {G}:=\Ker (p\circ q)$ and we consider the exact sequence 
\begin{equation}\label{eq.(*)_tilt_main}
0\rightarrow\widehat {G}\stackrel{\widehat {u}}{\longrightarrow}V^{(I)}\longrightarrow V^{(J)}\rightarrow 0.
\end{equation}
Then, $\widehat{u}$ is a special $\T$-preenvelope since $\Ext_\A^1(V^{(J)},-)$ vanishes on $\T$.
By Lemma \ref{injective_preenvelope}(3) we know that $u^{(K)}\colon \widehat{G}^{(K)}\to (V^{(I)})^{(K)}\cong V^{(I\times K)}$ is a monomorphism, for each set $K$ for which those coproducts exist in $\A$. Hence, the sequence in \eqref{eq.(*)_tilt_main} satisfies condition (3.4) for the generator $\widehat G$.
  
\smallskip\noindent
(3)$\Rightarrow$(2). Conditions (3.1) and (3.2) imply that $\T:=\Gen(V)\subseteq V^{\perp_1}$ since $\Ext_\A^1(V,-)$ is then a right exact functor. On the other hand,  let $u\colon G\to V_0$ be the monomorphism in (3.4). Then the monomorphism $u^{(I)}\colon G^{(I)}\to V_0^{(I)}$ is a special $\T$-preenvelope since $\Ext_\A^1(V_1^{(I)},-)$ vanishes on $\T$. Given an arbitrary $A\in\A$ and  an epimorphism $p\colon G^{(I)}\twoheadrightarrow A$, for some set $I$, Lemma \ref{lem.special-preenvelope-generator} gives a special $\T$-preenvelope $\mu_A\colon A\hookrightarrow T_A$, with $\Coker(\mu_A)\cong V_1^{(I)}$. In particular $\T$ is a cogenerating subcategory. When $A\in V^{\perp_1}$ in this argument, we get that $\mu_A$ is a section since then $\Ext_\A^1(V_1^{(I)},A)=0$. This implies that $V^{\perp_1}\subseteq\T$. Therefore, $\Gen(V)=V^{\perp_1}$ is  cogenerating. 
\end{proof}

\begin{remark} \label{rem.tilting with projective epigenerator}
If, in Proposition \ref{prop.Ext-universal tilting}, $\A$ has a projective epi-generator $P$ (e.g., if $\A$ is a category of modules) and  in the proof of ``(1,2)$\Rightarrow$(3)'' we take $G=P$, then the universal extension 
\[
0\rightarrow P\stackrel{u}{\longrightarrow} T\stackrel{p}{\longrightarrow}V^{(J)}\rightarrow 0
\] 
has the property that $T\in\T\cap {}^{\perp_1}\T=\Add(V)$ and $u$ is a (special) $\T$-preenvelope. By putting $V_0=T$ and $V_1=V^{(J)}$, one then gets the sequence of condition (3.4) with $G=P$.
\end{remark}

Our next result shows that two of the classical definitions of tilting objects are equivalent, and agree with ours, for a large class of Abelian categories $\A$ appearing ``in nature''.

\begin{corollary} \label{cor.different-definitions-tilting}
Let $\A$ be an Abelian category which is either  ${\Ext}^1$-small and (Ab.4)  (e.g., it is (Ab.3*) Abelian with an injective cogenerator) or  $\Hom$- and ${\Ext}^1$-finite. Then, the following are equivalent for an object $V\in \A$:
\begin{enumerate}[\rm (1)]
\item $V$ is a tilting object;
\item ${\Gen}(V)=V^{\perp_1}$;
\item the following conditions hold: 
\begin{enumerate}[\rm ({3.}1)]
\item the projective dimension of $V$ is $\leq 1$;
\item ${\Ext}_\A^1(V,V^{(I)})=0$,  for all sets $I$ (enough ${\Ext}_\A^1(V,V)=0$ in the $\Hom$-finite case);
\item $V^{\perp_{0,1}}:=\Ker (\A(V,-))\cap\Ker (\Ext_\A^1(V,-))=0$.
\end{enumerate}
\end{enumerate}
\end{corollary}
\begin{proof}
Under both sets of hypotheses, the subcategory $\Add(V)$ is precovering.

\smallskip\noindent
(1)$\Rightarrow$(2) is clear.

\smallskip\noindent
(2)$\Rightarrow$(1). We have already noticed that $\Add (V)$ is precovering.  Moreover, by Proposition \ref{prop.derivcoprod=0-implies-projdim=1} and Corollary \ref{coro_ext_univ_finite}, we know that $V^{\perp_1}$ is cogenerating. Then, assertion (2) of Proposition \ref{prop.Ext-universal tilting} holds.
%\textcolor{black}{$(2)\Longrightarrow (1)$ By the first line of this proof,  $\T=\Gen (V)$ is a torsion class.  Moreover, by Proposition \ref{prop.derivcoprod=0-implies-projdim=1}, we know that $V^{\perp_1}$ is a cogenerating class. Then assertion 2 of Proposition \ref{prop.Ext-universal tilting} holds.}

\smallskip\noindent
(1,2)$\Leftrightarrow$(3)  follows as for Grothendieck categories (see \cite{colpi}).
\end{proof}

The following is an interesting consequence of the previous results:

 \begin{corollary} \label{prop.Ext-univ-tilting implies Ext-small}
Let $\A$ be an Abelian category and $V\in\A$ a tilting object. 
 \begin{enumerate}[\rm (1)]
 \item If $\A$ is (Ab.3) then $\A$ is $\Ext^1$-small.
 \item Suppose that $R$ is a coherent commutative ring and that $\A$ is an $R$-linear category such that $\A(A,B)$ is a finitely presented $R$-module, for all $A,\, B\in\A$ (e.g., any $\Hom$-finite category $\A$ over a commutative Noetherian ring). Then, $\Ext_\A^1(A,B)$ is a finitely presented $R$-module, for all $A,\, B\in\A$. In particular, $\A$ is  $\Ext^1$-finite.
 \end{enumerate}
 \end{corollary} 
 \begin{proof}
 Let us start recalling that each tilting object is $\Ext^1$-universal (see Lemma \ref{tilt_implies_ext_univ_lemma}).
 
  \smallskip\noindent
 (1). By Proposition \ref{prop_ext_univ_infinite}, we know that $\Ext^1_\A(V,B)$ is a set, for each  $B\in\A$. Consider then $A\in\A$ and take a universal extension $0\rightarrow A\to T_A\to V^{(I)}\rightarrow 0$ with $T_A\in V^{\perp_1}=\T$. We get an exact sequence of (a priori big) Abelian groups 

%(1). By Proposition \ref{prop_ext_univ_infinite}, we know that $\Ext^1_\A(V,B)$ is a set, for each  $B\in\A$. Consider then $A\in\A$ and take a universal extension $0\rightarrow A\to T_A\to V^{(I)}\rightarrow 0$ with $T_A\in V^{\perp_1}=\T$ (that we can do by Lemma \ref{tilt_implies_ext_univ_lemma}). We get an exact sequence of (a priori big) Abelian groups 

\[
\Ext^1_\A(V^{(I)},B)\to\Ext^1_A(T_A,B)\to\Ext^1_A(A,B)\to\Ext^2_A(V^{(I)},B)=0
\] 
(see Corollary \ref{coro_proj_dim} for the last equality). Then our task reduces to prove that $\Ext_\A^1(T,B)$ is a set, for all $T\in\T$ and $B\in\A$. But, by hypothesis, we have that $\T=\Pres(V)$, so we can fix an exact sequence $0\rightarrow T'\to V^{(J)}\to T\rightarrow 0$, with $T'\in\T$,  leading to an exact sequence of (a priori big) Abelian groups $\A(T',B)\to\Ext_\A^1(T,B)\to\Ext_\A^1(V^{(I)},B)$, where the outer terms are sets.% It then follows that $\Ext_\A^1(T,B)$ is a set, as desired. 
 
 \smallskip\noindent
 (2). By Corollary \ref{coro_ext_univ_finite}, we know that $\Ext^1_\A(V,B)$ is a finitely generated $R$-module, for each  $B\in\A$. Proceeding as in the proof of part (1) (and using Proposition \ref{prop.Hom-finite-are-coherent} to show that the set $I$ has to be finite), we are reduced to prove that $\Ext_\A^1(T,B)$ is a finitely presented $R$-module, for all $T\in\T$ and $B\in\A$. Since $V$ is $\Ext^1$-universal, we can fix a universal extension 
 \[
 0\rightarrow B\longrightarrow T_B\stackrel{\pi}{\longrightarrow} V^{(r)}\rightarrow 0.
 \] 
 This gives an exact sequence of (a priori big) $R$-modules 
 \[
 \A(T,T_B)\longrightarrow\A (T,V^{(r)})\longrightarrow\Ext_\A^1(T,B)\longrightarrow\Ext_\A^1(T,T_B)\stackrel{\pi_*}{\longrightarrow}\Ext_\A^1(T,V^{(r)}),
 \] 
 where $\pi_*=\Ext_\A^1(T,\pi)$.   Bearing in mind  that $\fpmod R$ is closed under extensions in $\mod R$ and that $R$ is coherent, our task is further reduced to prove that $\Ext_\A^1(T,T')$ is a finitely presented $R$-module, for all $T,\, T'\in\T$. Now, use that $\T=\Pres(V)$ again, to construct an exact sequence 
 \[
 0\rightarrow\widetilde{T}\longrightarrow V^{(s)}\longrightarrow T\rightarrow 0, \quad\text{with $\widetilde T\in\T$}.
 \] 
Applying the contravariant functor $\A(-,T')\colon\A\to\mod R$, we then get an exact sequence 
\[
\A(V^{(s)},T')\longrightarrow\A(\widetilde{T},T')\longrightarrow\Ext_\A^1(T,T')\rightarrow 0,
\] 
and so $\Ext_\A^1(T,T')$ is a finitely presented as desired. 
 \end{proof}

\subsection{ (Quasi-)tilting torsion classes}\label{ext_universal_tilting_Sec}
The goal of this subsection is to identify the \mbox{(semi-)}special preenveloping torsion classes which are given by (quasi-)tilting objects.

\begin{proposition} \label{prop.lorthogonal-generates-torsion}
 Let $\A$ be an Abelian category and $\T$ a torsion class in $\A$ 
 such that $\mathcal{B}:=\Sub(\T)$ is a reflective subcategory of $\A$ that has an epi-generator (e.g., if $\A$ has an epi-generator, see Remark \ref{rem.to-main-theorem}). Then, the following assertions are equivalent:
 \begin{enumerate}[\rm (1)]
  \item  $\T=\Gen(V)$, for a quasi-tilting object $V$;
\item $\T$ is semi-special preenveloping in $\A$ and $\mathcal{B}$ contains an epi-generator that is in $\Ker(\Ext_\mathcal{B}^1(-,\T))$;
 \item $(\Ker(\Ext_\mathcal{B}^1(-,\T)),\T)$  is a complete cotorsion pair in $\mathcal{B}$;
 \item $\T$ is semi-special preenveloping in $\A$ and $_{}^{\perp_1}\T$ generates $\T$, that is, each object of $\T$ is the epimorphic image of one in  $_{}^{\perp_1}\T$.
 \end{enumerate}
 \end{proposition}
 \begin{proof}
 (1)$\Rightarrow$(4). By the implication ``(1)$\Rightarrow$(2)'' in Theorem \ref{thm.specialpreenveloping-versus-cotorsionpair}, to see that $\T$ is semi-special preenveloping in $\A$ it is enough to check that $\B$ is reflective (which we know by assumption), and that $\T=\B\cap V^{\perp_1}$ with $V\in \T$ an object which is $\Ext^1$-universal in $\B$ (which is also true by the definition of quasi-tilting, Proposition \ref{prop.properties of quasi-tilting}(1) and Lemma \ref{tilt_implies_ext_univ_lemma}). The fact that ${}^{\perp_1}\T$ generates $\T$ easily follows since $V\in {}^{\perp_1}\Tcal$. 
 
 \smallskip\noindent 
 (4)$\Rightarrow$(2). Let us fix an epi-generator $Q$ of $\B$. By Theorem \ref{thm.specialpreenveloping-versus-cotorsionpair}, $\T$ is special preenveloping in $\B$ so, by Lemma \ref{2-4implies1_lemma}, we have an exact sequence 
 \[
 \xymatrix{
 0 \ar[r] & Q \ar[r]^-{\lambda} & T_Q \ar[r] & V_Q \ar[r] & 0
 }
 \]
 where $\lambda$ is a special $\Tcal$-preenvelope of $Q$ in $\B$, and $\Tcal=\Gen(T_Q)=\Ker(\Ext^{1}_{\B}(V_Q,-))=\B \cap V_Q^{\perp_1}$ (see Lemma~\ref{lem.extension-in-reflective-subcats}(1) for the last equality). Take now an epimorphism $\rho\colon X_Q \twoheadrightarrow T_Q$, with $X_Q\in {}^{\perp_1}\Tcal.$ Fix also a left adjoint $\tau\colon \A \to \B$ to the inclusion $\iota\colon \B \hookrightarrow \A$, and recall that the unit  $u\colon \id_{\A} \to  \iota \circ \tau$ of this adjunction is an epimorphism (see Lemma \ref{reflector_is_epi_lemma}). Consider now the epimorphism 
 \[
 \xymatrix{
 \tau(X_Q) \ar@{>>}[r]^{\tau(\rho) \hspace{0,8 cm}}  & \tau(T_Q)\cong T_Q& \text{ in $\B$.} 
 }
 \]
 We claim that $\tau(X_Q)\in \Ker(\Ext^{1}_{\B}(-,\Tcal))$. In fact, given $T\in \Tcal$ and a short exact sequence 
 \begin{equation}\label{sequence_split_eq}
 \xymatrix{
 0 \ar[r] & T \ar[r] & B \ar[r] & \tau(X_Q) \ar[r] & 0&\text{ in $\B$,}
 }
 \end{equation}
 we get a commutative diagram with exact rows in $\A$: 
 \[
 \xymatrix{
 0 \ar[r] & \iota(T) \ar[r] \ar@{=}[d] & A \ar[r] \ar@{>>}[d] \ar@{}[dr]|{\text{P.O. \ P.B.}} & X_Q \ar[r] \ar@{>>}[d]^{u_{X_{Q}}} & 0 \\ 
 0 \ar[r] & \iota(T) \ar[r] & \iota(B) \ar[r] & (\iota \circ \tau)(X_Q) \ar[r]  & 0
 }
 \] 
 where the square on the right-hand side is then bicartesian. By the choice of $X_Q$, the upper row splits, and it is mapped by $\tau$ onto the sequence \eqref{sequence_split_eq}. Therefore, \eqref{sequence_split_eq} also splits and so \mbox{$\tau(X_Q)\in \Ker(\Ext^{1}_{\B}(-,\Tcal))$}. Consider now the following commutative diagram with exact rows in $\B$: 
 \[
 \xymatrix@C=40pt{
 0 \ar[r] & \widetilde{Q}\ar@{}[dr]|{\text{P.O. \ P.B.}} \ar[r] \ar[d]_{q} & \tau(X_Q) \ar[r] \ar@{>>}[d]^{\tau(\rho)} & V_Q \ar[r] \ar@{=}[d] & 0 \\ 
 0 \ar[r] & Q \ar[r] & T_Q \ar[r] & V_Q \ar[r] & 0
 }
 \]
where the map $q$ is an epimorphism since the left square is bicartesian. Then $\widetilde{Q}$ is an epi-generator of $\B$. On the other hand, by Corollary \ref{coro_proj_dim}, we have that the projective dimension of $V_Q$ in $\B$ is less or equal than one. We then get an exact sequence of functors $\Tcal \to \Ab$ as follow 
\[
\xymatrix{
0 =\Ext^{1}_{\B}(\tau(X_Q),-)_{\restriction_{\Tcal}} \ar[r] & \Ext^{1}_{\B}(\widetilde{Q},-)_{\restriction_{\Tcal}} \ar[r] & \Ext^{2}_{\B}(V_Q,-)_{\restriction_{\Tcal}}=0.  
}
\] 
It then follows that $\widetilde{Q}\in \Ker(\Ext^{1}_{\B}(-,\Tcal)).$
 	
\smallskip\noindent 
(2)$\Leftrightarrow$(3) is Remark \ref{remark_complete_cotorsion}.      
 	 
 \smallskip\noindent 
 (2--3)$\Rightarrow$(1). Let us fix an epi-generator $Q$ of $\B$ such that $Q\in \Ker(\Ext^{1}_{\B}(-,\Tcal))$. As in the proof of the implication ``(4)$\Rightarrow$(2)'', we consider a short exact sequence
\[
\xymatrix{
0 \ar[r] & Q \ar[r]^{\lambda} & T_Q \ar[r]^{\pi} & V_Q \ar[r] & 0,
}
\]
where $\lambda$ is a special $\T$-preenvelope in $\B$. Now, $Q$ and $V_Q$ belong in ${}^{\perp_1}\Tcal\cap \B$, so that  $T_Q\in \Tcal \cap {}^{\perp_1}\Tcal$. We then have $\Gen(T_Q)=V_{Q}^{\perp_1} \cap \B =\Tcal\subseteq T_{Q}^{\perp_1}\cap \B$. By taking $V:=T_{Q}\oplus V_Q$, we get that $\Gen(V)=\Gen(T_Q)=\Tcal$ and $\B\cap V^{\perp_1}=\B \cap T_{Q}^{\perp_1} \cap V_Q^{\perp_1}=\B \cap V_{Q}^{\perp_1}=\Tcal$. Furthermore, $\lambda$ is in particular a $\Gen(V)$-preenvelope of $Q$, whose cokernel is $V_Q\in \Add(V)$. Thus, by Lemma \ref{lem.special-preenvelope-generator}, we have that for every object $B$ in $\mathcal{B}$, there is $\mu_B\colon B \to T_B$ a $\Tcal$-preenvelope of $B$ whose cokernel is in $\Add(V)$. Let $T\in \Gen(V)$ and we consider an exact sequence in $\A$ of the form:
\[
0\rightarrow K\stackrel{u}{\longrightarrow} V^{(I)}\stackrel{p}{\longrightarrow} T\rightarrow 0.
\] 
for some $I$ set. Note that $K\in \B$, and hence, we can consider $\mu_K\colon K \to T_K$ a $\Tcal$-preenvelope of $K$ such that $\Coker(\mu_K)\in \Add(V)$. Consider the following commutative diagram with exact rows:  
%\[
%0\rightarrow K\stackrel{v}{\longrightarrow} \tilde{T}\stackrel{q}{\longrightarrow} V^{(J)}\rightarrow 0.
%\]
%In particular, from the monomorphism $p_*\colon\Ext_\A^1(V,\tilde{T})\hookrightarrow\Ext_\A^1(V,V^{(J)})=0$, we deduce that $\tilde{T}\in V^{\perp_1}=\Gen(V)$. Now, we consider the following commutative diagram with exact rows:
\[
\xymatrix@C=50pt{
0 \ar[r] &\ar@{}[dr]|{\text{P.O.}}  K \ \ar@{^(->}[d]_{\mu_K} \ar@{^(->}[r]^{ u} & V^{(I)} \ar@{>>}[r]^{p} \ar@{^(->}[d] & T \ar[r] \ar@{=}[d] & 0 \\ 
0 \ar[r] & T_K \  \ar@{^(->}[r] \ar@{>>}[d]_{q} & Z  \ar@{>>}[r] \ar@{>>}[d] & T \ar[r] & 0
\\ & \Coker(\mu_K) \ar@{=}[r] & \Coker(\mu_K)
}
\]
From Lemma \ref{lem.Ext-V-versus-Ext-VI}, we obtain that $Z\cong V^{(I)}\oplus \Coker(\mu_K)$ and so $Z\in\Add(V)$. Therefore $T\in \Pres(V)$ since  $T_K\in \Gen(V)$. It follows that $\Pres(V)=\Gen(V)=\T$, so that $V$ is a quasi-tilting object. 
\end{proof}

As an immediate consequence, we get:
\begin{corollary} \label{cor.tilting-bijection}
Let $\A$ be an Abelian category with an epi-generator and $\mathbf{t}=(\T,\mathcal F)$  a torsion pair in $\A$. The following assertions are equivalent:
\begin{enumerate}[\rm (1)]
\item $\mathbf{t}$ is a tilting torsion pair;
\item  $({}^{\perp_1}\T,\T)$ is a right complete cotorsion pair and ${}^{\perp_1}\T$ contains an epi-generator of $\A$;
\item $({}^{\perp_1}\T,\T)$ is a complete cotorsion pair in $\A$;
\item $\T$ is special preenveloping in $\A$ and ${}^{\perp_1}\T$ generates $\T$.
\end{enumerate}
In particular, the assignment $[V]\mapsto\Gen(V)$ gives a one-to-one correspondence 
\[
\xymatrix@R=12pt@C=30pt{
{\left\{\begin{matrix}\text{Equivalence classes of }\\  \text{tilting objects in $\A$}\end{matrix}\right\}}\ar@{->}[rr]^-{1:1}&&{\left\{\begin{matrix}\text{Special preenveloping}\\ \text{torsion classes $\T$ in $\A$}\\
\text{such that ${}^{\perp_1}\T$ generates $\T$}\end{matrix}\right\}}
}
\]
\end{corollary}

\subsection{Quasi-tilting objects versus semi-special preenveloping torsion classes}\label{sub.quasi-tilting}
In this subsection we  apply Proposition \ref{prop.lorthogonal-generates-torsion} to some cases where we can guarantee that the torsion class $\T$ has the property that $\B:=\Sub(\T)$ is a reflective subcategory of $\A$.

\begin{corollary} \label{cor.qtilting bijection for Ab3}
Let $\A$ be an (Ab.3) Abelian category with a generator, (whence it is bicomplete and well-powered   by Lemma \ref{lem.AB3-are-AB3*}). The assignment $[V]\mapsto\Gen(V)$ gives a bijection between the equivalence classes of quasi-tilting objects $V$ such that $\Gen(V)$ is closed under taking products and the semi-special preenveloping torsion classes $\T$  such that $_{}^{\perp_1}\T$ generates $\T$. 
\\
When $\A$ is also (Ab.4*) (e.g., when $\A$ has a projective generator), the mentioned quasi-tilting objects are precisely the {\bf strongly finendo} ones, i.e. those for which $(V^{(J)})^I\in \Gen(V)$, for all sets $I,\, J$. 
\end{corollary}
\begin{proof}
By Corollary \ref{coro.reflective-versus-productclosedness}, we know that if $V$ is a quasi-tilting object as indicated and $\T=\Gen(V)$, then $\B:=\Sub(\T)$ is a reflective subcategory. Moreover, we know that $\B$ is also (Ab.3) with a generator.
Then, Proposition  \ref{prop.lorthogonal-generates-torsion} gives that $\T$ is semi-special preenveloping and  $_{}^{\perp_1}\T$ generates $\T$. 
Conversely, if we have a torsion class $\T$ as the latter, then it is closed under products, which implies, again by Corollary \ref{coro.reflective-versus-productclosedness},  that $\B=\Sub(\T)$ is a reflective subcategory. By  Proposition \ref{prop.lorthogonal-generates-torsion}, we conclude that $\T=\Gen(V)$, for a quasi-tilting object $V$ such that $\Gen(V)$ is closed under products. 

For the last statement, when $\A$ is (Ab.4*), we just need to prove that, if $V$ is strongly finendo, then $\Gen(V)$ is closed under products. Indeed, let $(T_\lambda)_{\Lambda}$ be a family in $\Gen(V)$ and fix epimorphisms $p_\lambda \colon V^{(I_\lambda)}\twoheadrightarrow T_\lambda$ (with $\lambda\in\Lambda$). There is no loss of generality in assuming that $I_\lambda =I_\mu=:I$, for all $\lambda,\, \mu\in\Lambda$. The (Ab.4*) condition gives an induced epimorphism $\prod p_\lambda:(V^{(I)})^\Lambda\twoheadrightarrow \prod_{\Lambda}T_\lambda$, which implies that  $\prod_{\Lambda}T_\lambda\in\T$, since $V$ is strongly finendo. 
\end{proof}

\begin{remark} \label{rem.Colpi-Menini}
Let us start with the same setting of Corollary \ref{cor.qtilting bijection for Ab3}, that is, $\A$ is an (Ab.3) Abelian category with a generator,
and call an object $X\in \A$ {\bf finendo} when $X^I\in\Gen(X)$, for all sets $I$. If $\A$ is also (Ab.4*), $V$ is quasi-tilting and the canonical map $V^{(I)}\to V^I$ is a monomorphism for all sets $I$ (e.g., when $\A$ is (Ab.5)), then $V$ is finendo if and only if it is strongly finendo (as defined in Corollary \ref{cor.qtilting bijection for Ab3}). Indeed, in such case the canonical map $(V^{(J)})^I\to (V^J)^I\cong V^{J\times I}$ is a monomorphism which implies, by the (Ab.4*) condition,  that $\overline{\Gen}(V)$ is closed under products. This, in turn,  implies that $\Gen(V)=\overline{\Gen}(V)\cap V^{\perp_1}$ is closed under products, and hence that $V$ is strongly finendo.  Therefore, when $\A$ is an (Ab.4*) Grothendieck category, Corollary \ref{cor.qtilting bijection for Ab3} gives a bijection between equivalence classes of finendo quasi-tilting objects and semi-special preenveloping torsion classes such that ${}^{\perp_1}\T$ generate $\T$. In fact, when $\A=\mod R$ is a module category (in which case $\A$ is Grothendieck and (Ab.4$^*$)), it is well-known that $X$ is finendo exactly when it is finitely generated as a left module over $\End_{R}(X)$ (see \cite[Lemma after Prop.\,1.5]{colpi1993structure}), so Corollary \ref{cor.qtilting bijection for Ab3} extends the known results for categories of modules (see \cite[Coro.\,3.8]{hugel2015silting}, and also \cite[Thm.\,2.1]{hugel2001tilting} for the tilting case). 
\end{remark}

\begin{corollary} \label{cor.qtilting-specialpreenv-in-Homfinite}
Let $\A$ be a $\Hom$-finite  Abelian category with a generator such that all objects have finite length. Then, the assignment $[V]\mapsto\Gen(V)$ gives a bijection between equivalence classes of quasi-tilting objects and enveloping (=(semi-special) preenveloping) torsion classes $\T$ such that  $_{}^{\perp_1}\T$ generates $\T$. Furthermore, this bijection restricts to  a one-to-one correspondence between equivalence classes of tilting objects and special enveloping torsion classes $\T$ such that  ${}^{\perp_1}\T$ generates $\T$. 
\end{corollary}
\begin{proof}
Bearing in mind that $\A$ is a Krull-Schmidt category, we know by Lemma \ref{lem.precovering=covering} all preenveloping classes closed under direct summands are enveloping. Furthermore, by Corollary~\ref{cor.enveloping-implies-semispecialpreenv}, we know that the enveloping torsion classes are exactly the semi-special preenveloping and that these are special precisely when they are cogenerating. Now,
by Example \ref{example_artinian_reflective}, we know that $\Sub(\X)$ is a reflective subcategory, for all torsion classes $\X$ in $\A$. The result is then a direct consequence of  Proposition~\ref{prop.lorthogonal-generates-torsion}.
\end{proof}

\begin{corollary} \label{cor.qtilting-specialpreenv.-projepigenerator}
Let $\A$ be an Abelian category with a projective epi-generator $P$. Then, the assignment $[V]\mapsto\Gen(V)$ gives a bijection between equivalence classes of quasi-tilting objects $V$ such that $P$ has an $\Add(V)$-preenvelope and semi-special preenveloping torsion classes in $\A$. This bijection restricts to one between equivalence classes of tilting objects and special preenveloping torsion classes.
Moreover, when no non-zero infinite coproduct $X^{(I)}$ exists in $\A$ (e.g., if $\A$ is $\Hom$-finite, see Proposition \ref{prop.Hom-finite-are-coherent}), those (quasi-)tilting objects are precisely the $V$ for which $\A(P,V)$ is finitely generated as a left $\End_\A(V)$-module. 
\end{corollary}
\begin{proof}
If $V$ is a quasi-tilting object such that $P$ has an $\Add(V)$-preenvelope and $\T:=\Gen(V)$, then we know by Lemma \ref{example_4_lemma} that $\B=\Sub(\T)$ is a reflective subcategory. On the other hand, if $\T$ is a (semi-)special preenveloping torsion class in $\A$, by Proposition \ref{ex.reflective-versus-productclosedness}, we have that $\B=\Sub(\T)$ is reflective. Note that, in both cases, if  $\tau\colon \A\to\B$ is the left adjoint to the inclusion, then $\tau (P)$ is a projective epi-generator of $\B$, so that, the bijection is then a direct consequence of the Proposition~\ref{prop.lorthogonal-generates-torsion} since $P\in {}^{\perp_1}\mathcal{X}$, for every class of objects $\mathcal{X}$ in $\A$.

Finally, when no non-zero infinite coproduct $X^{(I)}$ exists in $\A$, it is a standard fact that a morphism $f=(f_1,\dots,f_n )^t\colon P\to V^n$ is an $\Add(V)(=\add(V))$-preenvelope if, and only if, $\{f_1,\dots,f_n\}$ is a set of generators of $\A(P,V)$ as a left $\End_\A(V)$-module. 
\end{proof}

\subsection{A look at the duals: (quasi-)cotilting precovers}\label{sec. duals}

We leave to the reader the literal dualization of the bijections of the previous subsections between classes of (quasi-)tilting objects and classes of (semi-)special torsion classes in Abelian categories. Here we only emphasize a few of them. The following one is a particular case of the dual of Theorem \ref{thm.specialpreenveloping-versus-cotorsionpair} (see also Proposition \ref{prop.lorthogonal-generates-torsion}).
 
 \begin{theorem} \label{thm.dual 5.1}
 Let $\A$ be an $\Ext^1$-small (Ab.3*) Abelian category with a cogenerator (so $\A$ is bicomplete and well-powered), let $\mathcal F$ be a  torsion-free class in $\A$ and $\C:=\Quot(\mathcal F)$. Then, the following are equivalent:
  \begin{enumerate}[\rm (1)]
 \item $\mathcal F$ is closed under taking coproducts in $\A$ and $\mathcal F=\C\cap {}^{\perp_1}Q$, for some $Q\in\mathcal F$;
 \item $\mathcal F$ is semi-special precovering in $\A$;
 \item $\mathcal F$ is closed under coproducts and $(\mathcal F,\Ker(\Ext_\C^1(\mathcal F,-)))$ is a left complete cotorsion pair in $\C$.
 \end{enumerate}
 In this case, $\mathcal F$ is the torsion-free class of a quasi-cotilting torsion pair if and only if the cotorsion pair of assertion (3) is complete in $\mathcal C$, if and only if $\mathcal F^{\perp_1}$ cogenerates $\mathcal F$.
 \end{theorem}
 \begin{proof}
By Proposition~\ref{prop.coreflective}, we know that $\C$ is a coreflective subcategory if, and only if, it is closed under taking coproducts in $\A$, a fact that is guaranteed whenever $\Fcal$ is closed under taking coproducts. So that in all the conditions, we have that $\C$ is a coreflective subcategory (for the condition (2), we use the dual of the Proposition~\ref{ex.reflective-versus-productclosedness}). Thus, the result is a direct consequence of the duals of Theorem~\ref{thm.specialpreenveloping-versus-cotorsionpair} and Proposition~\ref{prop.lorthogonal-generates-torsion}  whenever the object $Q$ of the condition (1) is  $\Ext^1$-universal in $\C^{\op}$. But, in such condition, we have that $\Ker(\Ext^1_{\C}(-,Q))=\C \cap {}^{\perp_1}Q=\Fcal$, that is, $\Ker(\Ext^1_{\C}(-,Q))$ is a generating class of $\C=\Quot(\Fcal)$. Therefore, combining the duals of Example~\ref{ex_exact_copr} and Proposition~\ref{prop_ext_univ_infinite}, we obtain that $Q$ is $\Ext^1$-universal in $\C^{\op}$. 
% The result is a direct consequence of the duals of  Theorem \ref{thm.specialpreenveloping-versus-cotorsionpair} and  Proposition \ref{prop.lorthogonal-generates-torsion}, bearing in mind that, by  Proposition \ref{prop_ext_univ_infinite}, $Q$ is $\Ext^1$-universal in $\C^{\op}$ and that, by  Proposition~\ref{prop.coreflective}, $\C$ is a coreflective subcategory if, and only if, it is closed under taking coproducts in $\A$,  a fact that is guaranteed whenever $\mathcal F$ is closed under taking coproducts. 
 \end{proof}
 
  As an immediate consequence, we get (see Lemma \ref{lem.Ext-V-versus-Ext-VI}(2)):

\begin{corollary} \label{cor-special-precovering-Ab3*}
In the same setting of Theorem \ref{thm.dual 5.1}, the following assertions are equivalent:
\begin{enumerate}[\rm (1)]
 \item $\mathcal F$ is a generating class such that $\mathcal F= {}^{\perp_1}Q$, for some $Q\in\mathcal F$ (resp., $\mathcal F={}^{\perp_1}Q$ for $Q$ cotilting);
 \item $\mathcal F$ is a special precovering  torsion-free class (such that $\mathcal F^{\perp_1}$ cogenerates $\mathcal F$);
 \item $( \mathcal F,\mathcal F^{\perp_1})$ is a left complete (resp., complete) cotorsion pair. 
\end{enumerate}
\end{corollary}

  When  $\A$ is (Ab.3*) with an injective cogenerator, whence  even (Ab.4) (see Lemma \ref{lem.AB3-are-AB3*} and use the fact that any (Ab.3) Abelian category with an injective cogenerator is (Ab.4)), we get the following consequence of Corollaries \ref{cor.qtilting bijection for Ab3}, \ref{cor-special-precovering-Ab3*} and Theorem \ref{thm.dual 5.1}.

\begin{corollary} \label{cor.precovering- torsion-free-in-Ab3*}
Let $\A$ be an (Ab.3*) Abelian category with an injective cogenerator. The assignment $[Q]\mapsto\Cogen(Q)$ gives a one-to-one correspondence:
\[
\xymatrix@R=12pt@C=30pt{
{\left\{\begin{matrix}\text{equivalence classes of strongly cofinendo}\\ \text{ quasi-cotilting objects in $\A$} \end{matrix}\right\}}\ar@{->}[rr]^-{1:1}&&{\left\{\begin{matrix}\text{semi-special precovering}\\ \text{torsion-free classes in $\A$}\end{matrix}\right\}}.
}
\]
(where an object is strongly cofinendo if it is strongly finendo in $\A^{\op}$). This bijection restricts to the following one:
\[
\xymatrix@R=12pt@C=30pt{
{\left\{\begin{matrix}\text{equivalence classes of }\\ \text{cotilting objects in $\A$}\end{matrix}\right\}}\ar@{->}[rr]^-{1:1}&&{\left\{\begin{matrix}\text{special precovering}\\ \text{torsion-free classes in $\A$}\end{matrix}\right\}}.
}
\]
\end{corollary}

When $\A=\G$ is a Grothendieck category, one can say even more. We refer the reader to \cite{parra2019locally} for the definition of cosilting object and {cosilting }torsion pair in such a category.

\begin{proposition} \label{prop.}
Let $\G$ be a Grothendieck category. For a torsion pair $\mathbf{t}=(\T,\mathcal F)$, the following assertions are equivalent:
\begin{enumerate}[\rm (1)]
\item $\mathbf{t}$ is of finite type, i.e. $\mathcal F$ is closed under taking direct limits;
\item $\mathbf{t}$ is a quasi-cotilting torsion pair;
\item $\mathbf{t}$ is a cosilting torsion pair;
\item $\mathcal F$ is (semi-special) precovering;
\item $\mathcal F$ is covering.
\end{enumerate}
In particular, the assignment $[Q]\mapsto\Cogen(Q)$ induces a bijection between equivalence classes of quasi-cotilting (resp., cosilting) objects and (pre)covering  torsion-free classes in $\G$.
\end{proposition}
\begin{proof}
The equivalences ``(1)$\Leftrightarrow$(2)$\Leftrightarrow$(3)'' are \cite[Thm.\,4.1]{parra2019locally} and ``(2)$\Leftrightarrow$(4)'' follows from Corollary \ref{cor.precovering- torsion-free-in-Ab3*}, bearing in mind that in a Grothendieck category all objects are strongly cofinendo. On the other hand,  ``(1,4)$\Rightarrow$(5)'' follows from a theorem of Enochs and Xu (see \cite[Thm.\,1.2]{el2006covers}). The implication ``(5)$\Rightarrow$(4)'' is clear.  
\end{proof}

\begin{remark}
In the setting of module categories, the above result was obtained independently by Breaz and {\v{Z}}emli{\v{c}}ka \cite[Thm.\,3.5]{breaz2018torsion} and Zhang and Wei \cite[Thm.\,3.5]{zhang2017cosilting} (for more on this, see also \cite[Thm.\,3.8 and Coro.\,3.9]{hugel2018abundance}).
\end{remark}

\section{Finitely presented modules over right coherent rings} \label{Sec_fin_pres}

In this last section we fix a (usually right coherent) ring $R$ and we take a deeper look at the results of the previous sections when $\A:=\fpmod R$ is the category of finitely presented (right) $R$-modules (by Lemma \ref{lem.only-finite-copr-project-epigener} all Abelian categories with a projective finite epi-generator are of this kind). 

\subsection{Basic facts and notations}
In what follows we will mostly work in two categories: the category of all right $R$-modules $\mod R$ and that of finitely presented right $R$-modules $\fpmod R$. To avoid confusion, let us fix the following notational conventions:
\begin{itemize}
\item we use the symbol $\coprod$ to denote coproducts in the category $\mod R$, while the symbol $\mathrlap{\scalebox{.5}{\text{\hspace{8.5pt}$\mathrm{fp}$}}}\coprod$ denotes coproducts in $\fpmod R$. For any non-empty set $I$, denote by $\mathrlap{\scalebox{.5}{\text{\hspace{8.5pt}$\mathrm{fp}$}}}\coprod_IV$ the coproduct of $|I|$-many copies of $V$ in $\fpmod R$, when such a coproduct exists;
\item all throughout this section, the symbols $\Sub$, $\subGen$, $\Gen$, and $\Pres$, without suffix, will be meant in the category $\mod R$. We will use the suffix ``${\fpmod R}$'' (e.g., $\Sub_{\fpmod R}(-)$) when we want them taken in $\fpmod R$.
\end{itemize}

Let us start noting  that $\fpmod R$ may have infinite coproducts.

\begin{proposition} \label{prop.infinite coproducts in fpmod}
Let $R$ be any ring and $I$ an infinite set. The following assertions are equivalent:
\begin{enumerate}[\rm (1)]
\item $\fpmod R$ has $I$-coproducts;
\item the coproduct of $|I|$-many copies of $R$ exists in $\fpmod R$;
\item $R$ is Morita equivalent to a ring $S$ such that $S$ is isomorphic to $S^I$ in $S\modl$.
\end{enumerate}
\end{proposition}
\begin{proof}
(1)$\Leftrightarrow$(2). Since $R$ is a finite epi-generator of $\fpmod R$, this is a consequence of  Lemma \ref{existence_of_coproducts_in_pres_lemma}(3).

\smallskip\noindent
(3)$\Rightarrow$(1, 2). Since assertion (1) is Morita invariant, replacing $R$ by $S$ if necessary, we can assume that $R\cong R^I$ in $R \modl$. We then get a family of orthogonal idempotents $(e_i)_{I}$ such that each $Re_i$ is isomorphic to $_RR$ (although not canonically) and the canonical  map $\rho \colon R\to \prod_{I}Re_i$, such that $r\rightarrow (re_i)_{I}$, is an isomorphism in $R \modl$. We then have  isomorphisms 
\[
e_iR\cong \hom_R({}_RRe_i,{}_RR)\cong \hom_R({}_RR,{}_RR)\cong R,\qquad\text{in $\fpmod R$.}
\] 
We claim that $R_R$, together with the inclusions $\iota_i\colon e_iR\hookrightarrow R$ (with $i\in I$), gives a coproduct of the $e_iR$ in $\fpmod R$. Take  an arbitrary object $X\in\fpmod R$ and consider the induced map 
\[
\xymatrix{
\varphi_X\colon X\cong \hom_R(R,X)\ar[r]^-{(\iota_i^*)_I}&\prod_{I}\hom_R(e_iR,X)\cong\prod_{I}Xe_i,
}
\] 
which is clearly identified with the map $x\rightarrow (xe_i)_{I}$. Let now $U\colon \fpmod R\to\Ab$ be the forgetful functor and $F\colon \fpmod R\to\Ab$ the obvious functor that acts on objects by $F(X):=\prod_{I}Xe_i$. Then, there is a natural transformation $\varphi\colon U\Rightarrow F$ where, for each $X\in \fpmod R$, $\varphi_X\colon X\to \prod_{I}Xe_i$ is defined as above.  The subcategory $\C\subseteq \fpmod R$ of all the $X$ for which $\varphi_X$ is an isomorphism  contains $R$ since $\varphi_R=\rho$, and is clearly closed under finite coproducts. Moreover, since all the $e_iR$ are projective in $\fpmod R$, and products are exact in $\Ab$,  $\C$ is also closed under taking cokernels. It then follows that $\C=\fpmod R$, so that $\varphi$ is a natural isomorphism and our claim  is settled.  

\smallskip\noindent
(2)$\Rightarrow$(3). Let $P:=\mathrlap{\scalebox{.5}{\text{\hspace{8.5pt}$\mathrm{fp}$}}}\coprod_I R$ be the coproduct of $|I|$-many copies of $R$ in $\fpmod R$. Then $P$ is projective in $\fpmod R$. As $P$ is finitely presented, there is an epimorphism $R^n\to P$ and, as $P$ is projective in $\fpmod R$, this morphism splits (both in $\fpmod R$ and in $\mod R$). Hence, $P$ is projective also in $\mod R$  and $R$ is a direct summand of $P$.  Therefore, $P$ is a progenerator of $\mod R$ and, moreover, we have that $\mathrlap{\scalebox{.5}{\text{\hspace{8.5pt}$\mathrm{fp}$}}}\coprod_I P\cong P$. Hence,
\[
\xymatrix{
S=\hom_R(P,P)\cong \hom_R(\mathrlap{\scalebox{.5}{\text{\hspace{8.5pt}$\mathrm{fp}$}}}\coprod_I P,P)\cong \hom_R(P,P)^I\cong S^I
}
\] 
in $S\modl$, where $S=\End_R(P_R)$, and this is a ring that is Morita equivalent to $R$.
\end{proof}

\begin{example} \label{ex.ring isomorphic to its product}
If $K$ is a field and $V=K^{(I)}$, for $I$ an infinite set, then the ring $R=\End_K(V)$ is a Von Neumann regular ring (whence coherent on both sides) such that 
\[
R^I\cong\Hom_K(V^{(I)},V)\cong\Hom_K(V,V)\cong R
\]
 as left $R$-modules. More generally (see \cite{Bergman}), 
if $A$ is any ring and $M$ is an $A$-module such that $M^{(I)}\cong M$, then $R^I\cong R$ in $R \modl$, when $R:=\End_A(M)$. 
% [G.M. Bergman: Two statements about infinite products that are not quite true. In Groups, Rings and Algebras (in honor of Donald S. Passman). W. Chin, ? Osterburg and D. Quinn (editors). Contemp. Maths. \textbf{420} (2006), 35-58.]) 
\end{example}

\subsection{Tilting objects in $\fpmod R$ versus classical tilting modules}

In this subsection we start the comparison between the classical tilting modules (as defined below) and the tilting objects in the Abelian category $\A:=\fpmod R$, as in Definition \ref{def.quasi-tilting torsion pair}. Let us start recalling the classical definition of tilting in a category of modules:

\begin{definition} \label{def.classical tilting module}
Let $R$ be any ring. A {\bf classical tilting right $R$-module} is a right $R$-module $V$ satisfying the following conditions:
\begin{enumerate}[\rm (CT.1)]
\item there is a short exact sequence 
\[
0\rightarrow P_1\longrightarrow P_0\longrightarrow V\rightarrow 0, 
\]
with $P_0$ and $P_1$ finitely generated projective right $R$-modules;
\item $\Ext_{\mod R}^1(V,V)=0$ or, equivalently,  $\Ext_{\mod R}^1(V,V^{(I)})=0$ for all sets $I$;
\item there is a short exact sequence 
\[
0\rightarrow R\longrightarrow V_0\longrightarrow V_1\rightarrow 0
\] 
such that $V_0,\,V_1\in\add(V)$.
\end{enumerate}
\end{definition}

It is well-known, and easy to see, that for an arbitrary ring $R$, the right $R$-module $V$ is classical tilting if, and only if, it is finitely presented and it is a tilting object of $\mod R$ in the sense of Definition \ref{def.quasi-tilting torsion pair}. On the other hand, when $R$ is right coherent (so $\fpmod R$ is Abelian), it is not clear when a classical tilting right $R$-module $V$ is also a tilting object of the Abelian category $\fpmod R$. In this subsection we study the relation between these two concepts. Before proceeding with the following auxiliary result, let us recall that a right $R$-module $V$ is said to be {\bf finendo}, provided it is finitely generated as a left module over $\End_R(V)$ (see also Remark \ref{rem.Colpi-Menini}).

\begin{lemma} \label{lem.fp quasitilted restricting}
Let $R$ be a right coherent ring, $V$ a finitely presented quasi-tilting right $R$-module (i.e., a quasi-tilting object in $\mod R$) and $\mathbf{t}_V:=(\Gen(V),V^\perp )$ the associated torsion pair in $\mod R$. Suppose also that the torsion pair $\mathbf{t}_V$ restricts to $\fpmod R$. Then, the following assertions hold true:
\begin{enumerate}[\rm (1)]
\item $\Gen(V)\cap\fpmod R=\Gen_{\fpmod R}(V)$ and $\subGen_{\fpmod R}(V)=\overline{\Gen}(V)\cap\fpmod R$;
\item $V$ is a quasi-tilting object of $\fpmod R$;
\item if $V$ is finendo, then $\subGen_{\fpmod R}(V)$ is reflective in $\fpmod R$.
\end{enumerate}
\end{lemma}
\begin{proof}
(1). We obviously have the inclusions $\Gen(V)\cap\fpmod R=\gen(V)\cap\fpmod R\subseteq\Gen_{\fpmod R}(V)$. On the other hand,  since $\Gen(V)\cap\fpmod R$ is a torsion class in $\fpmod R$, also the converse holds:
\begin{equation}
\Gen(V)\cap\fpmod R=\Gen_{\fpmod R}(V).
\end{equation}
By \cite[Lem.\,4.4]{Craw}, we know that then $\Gen(V)=\varinjlim (\gen(V)\cap\fpmod R)$, from which one easily deduces that $\overline{\Gen}(V)\cap\fpmod R\subseteq\subGen_{\fpmod R}(V)=\Sub_{\fpmod R}(\gen(V)\cap\fpmod R)$. But then,
\begin{equation}
\overline{\Gen}(V)\cap\fpmod R=\subGen_{\fpmod R}(V),
\end{equation}
since the converse inclusion is obvious. 

\smallskip\noindent
(2). From the equality $\Gen (V)=\overline{\Gen}(V)\cap V^{\perp_1}$ and part (1) we get that 
\begin{align*}
\Gen_{\fpmod R} (V)=\Gen (V)\cap\fpmod R&=\overline{\Gen}(V)\cap V^{\perp_1}\cap\fpmod R\\
&=\subGen_{\fpmod R}(V)\cap\ker(\Ext_{\fpmod R}^1(V,-)).
\end{align*}
To prove that $V$ is a quasi-tilting object of $\fpmod R$, we need to prove that $\Gen(V)\cap\fpmod R\subseteq\pres(V)$, for then we  have  that $\Gen_{\fpmod R}(V)=\Pres_{\fpmod R}(V)$. Indeed, given $T\in \Gen(V)=\Pres(V)$, there is a presentation 
\[
V^{(\Lambda)}\overset{f}\longrightarrow V^{(I)}\longrightarrow T\to 0.
\] 
We apply an argument due to Lazard (see \cite[Lem.\,1.11]{parra2019locally}), that we just outline here, to show that $T\in \varinjlim\,\pres(V)$. 
%Let $T\in\Gen(V)\cap\fpmod R=\Pres(V)\cap\fpmod R$ and fix an exact sequence $V^{(L)}\stackrel{f}{\to}V^{(K)}\stackrel{q}{\to}T\rightarrow 0$, for some sets $L$ and $K$. We consider the set $\Lambda$ of triples $(F',F,g)$ such that $F'\subseteq L$ and $F\subseteq K$ are finite subsets and $g:V^{(F')}\to V^{(F)}$ is a map such that $\iota_F\circ g=f\circ\iota_{F'}$, where the $\iota$'s are the respective inclusions into the coproducts $V^{(L)}$ and $V^{(K)}$. We put $(F'_1,F_1,g_1)\leq (F'_2,F_2,g_2)$ whenever $F'_1\subseteq F'_2$, $F_1\subseteq F_2$ and $g_2\circ\iota_{F'_1F'_2}=\iota_{F_1F_2}\circ g_1$. This is a partial order in $\Lambda$ making it into a directed poset. We next put $T_{(F',F,g)}=\Coker(g)$, for each $(F',F,g)\in\Lambda$. It is routine to check that we get a direct system $(T_{(F',F,g)})_{(F',F,g)\in\Lambda}$ in  $\pres(V)$ such that $\varinjlim T_{(F',F,g)}\cong T$. 
%
Given finite subsets $\Lambda'\subseteq \Lambda$ and $I'\subseteq I$ denote, respectively, by 
\[
\xymatrix{
\iota_{\Lambda'}\colon V^{(\Lambda')}\to V^{(\Lambda)}\quad\text{and}\quad\epsilon_{I'}\colon V^{(I')}\to V^{(I)}
}
\]
the inclusions into the coproduct, and define the following set 
\[
\Upsilon:=\{(\Lambda',I'):\Lambda'\subseteq \Lambda,\, I'\subseteq I\ \text{finite},\ \text{such that $f\circ\iota_{\Lambda'}$ factors through $\epsilon_{I'}$}\}.
\]
%In other words, $(\Lambda',I')\in \Upsilon$ if and only if we have a commutative diagram like the following one:
%\[
%\xymatrix@R=5pt@C=50pt{
%\coprod_{\Lambda'}K_\lambda\ar[rr]^{f\circ\iota_{\Lambda'}}\ar[dr]_{f_{(\Lambda',I')}}&&\coprod_{I}S_i.\\
%&\coprod_{I'}S_i\ar[ur]_{\varepsilon_{I'}}
%}
%\] 
%Note  that, if it exists, the map $f_{(\Lambda',I')}$ is uniquely determined by $f$. 
Endow $\Upsilon$ with the product order. %Due to the fact that, for each finite subset $I'\subseteq I$, the map $\varepsilon_{I'}$ is a monomorphism, one easily sees that the poset $\Upsilon$ is directed. 
It is  routine to check that $\Upsilon$ is directed and that 
\[
T\cong\Coker(f)\cong{\varinjlim}_\Upsilon\Coker(f_{(\Lambda',I')})\in \varinjlim\, \pres(V).
\] 
Finally, if we suppose that $T$ is also finitely presented, then there is some $(\Lambda',I')\in\Upsilon$ such that the canonical map  $u:=u_{(\Lambda',I')}:\Coker(f_{(\Lambda',I')})=:T'\to T$ is a retraction. We then get a commutative diagram in $\fpmod R$ with exact rows:
\[
\xymatrix@C=40pt@R=20pt{
0 \ar[r] & T'' \ar[r] \ar[d] & V^{(I')} \ar@{=}[d]\ar[r] & T' \ar[r] \ar[d]^{u} & 0\\ 
0 \ar[r] & U \ar[r] & V^{(I')} \ar[r] & T \ar[r] & 0  
}
\]
By the Snake Lemma, we have an exact sequence $0\rightarrow T''\to U\to\Ker (u)\rightarrow 0$, where the outer terms are in $\T\cap\fpmod R$. It follows that $U\in\T\cap\fpmod R$, and so $T\in\pres(V)$.

\smallskip\noindent
(3). Let $\T_0:=\Gen(V)\cap\fpmod R=\Gen_{\fpmod R}(V)$. Since $V$ is finendo, we can choose a finite set $\{f_1,\dots,f_n\}$ of generators of ${}_{\End_R(V)}V$. It is  routine to check that $f:=(f_1,\dots,f_n)\colon R\to V^n$ is an $\add (V)$-preenvelope. Hence, by Lemma \ref{example_4_lemma}, $\subGen_{\fpmod R}(V)=\Sub_{\fpmod R}(\T_0)$ is reflective in $\fpmod R$.  
\end{proof}

\begin{proposition} \label{prop.quasi-tilting in fpmodR}
Let $R$ be a right coherent ring and $V\in \fpmod R$. The following are equivalent:
\begin{enumerate}[\rm (1)]
\item $V$ is a quasi-tilting object of $\fpmod R$ and $\subGen_{\fpmod R}(V)$ is reflective in $\fpmod R$;
\item there is a (possibly infinite) set $J$ such that $V_J:=\mathrlap{\scalebox{.5}{\text{\hspace{8.5pt}$\mathrm{fp}$}}}\coprod_JV$ is a finitely presented finendo quasi-tilting $R$-module and the torsion pair $\t_{V_J}:=(\Gen(V_J),V_J^\perp )$ in $\mod R$ restricts to $\fpmod R$. 
\end{enumerate}
In this case, the restricted torsion pair $(\Gen(V_J)\cap\fpmod R,V_J^\perp\cap\fpmod R)$ is the torsion pair associated with $V$ in $\fpmod R$, that is, $\Gen(V_J)\cap\fpmod R=\Gen_{\fpmod R}(V)$. 
\end{proposition}
\begin{proof}
Note that $\Add_{\fpmod R}(V)=\Add_{\fpmod R}(V_J)$, whenever $V_J$ exists. Therefore $V$ is a quasi-tilting object of $\fpmod R$ if, and only if, so is $V_J$. In that case the associated torsion pairs in $\fpmod R$ coincide. 

\smallskip\noindent
(2)$\Rightarrow$(1). It follows from Lemma \ref{lem.fp quasitilted restricting}(2,3), together with the above comment. 

\smallskip\noindent
(1)$\Rightarrow$(2). Put $\T_0:=\Gen_{\fpmod R}(V)$ and $\B':=\Sub_{\fpmod R}(\T_0)$, and fix a left adjoint $L\colon\fpmod R\to\B'$ to the inclusion $\iota \colon\B'\to\fpmod R$. Then, $L(R)$ is a projective generator of $\B'$ and, by Proposition \ref{prop.properties of quasi-tilting}, we know that $V$ is tilting in $\B'$. By Proposition \ref{prop.Ext-universal tilting} and Remark \ref{rem.tilting with projective epigenerator}, we have an exact sequence 
\[
0\rightarrow L(R)\stackrel{\lambda}{\longrightarrow}V_0\longrightarrow V_1\rightarrow 0\qquad\text{in $\B'$,} 
\]
where $V_0,\,V_1\in\Add_{\fpmod R}(V)$, that is kept exact for any coproduct that might exist in $\fpmod R$. Consider now the unit $\mu\colon \id_{\fpmod R}\Rightarrow\iota\circ L$, that is an epimorphism (see Lemma \ref{reflector_is_epi_lemma}). By considering the composition $u:=\iota (\lambda)\circ\mu_R\colon R\stackrel{}{\to}(\iota\circ L)(R)\stackrel{}{\to}\iota (V_0)=V_0$, we obtain an exact sequence 
\[
R\stackrel{u}{\longrightarrow}V_0\longrightarrow V_1\to 0\qquad\text{in $\fpmod R$,} 
\]
and one easily sees that $u$ is a semi-special $\T_0$-preenvelope. We fix now a (possibly infinite) set $J$ such that $V_J:=\mathrlap{\scalebox{.5}{\text{\hspace{8.5pt}$\mathrm{fp}$}}}\coprod_JV$ exists in $\fpmod R$ and $V_0,\,V_1$ are both direct summands of $V_J$. Then, $u$ is an $\add(V_J)$-preenvelope, which implies that $V_J$ is finendo. We claim that 
\[
\T_0=\gen(V_J)\cap\fpmod R.
\] 
By the comment at the beginning of this proof, the inclusion ``$\supseteq$'' is obvious. Conversely, let $T\in\T_0$ and fix an epimorphism $p\colon R^n\to T$. Then, $p$ factors through $u^n\colon R^n\to V_0^n$ since the latter is a \mbox{$\T_0$-preenvelope}. This implies that $T$ is an epimorphic image of $V_0^n$, and so it belongs in $\gen (V_J)\cap\fpmod R$. 
Put now 
\[
\T:=\varinjlim\T_0=\varinjlim(\gen(V_J)\cap\fpmod R),
\] 
that is a torsion class in $\mod R$ by  \cite[Lem.\,4.4]{Craw}. We then have that $\T\subseteq\Gen(V_J)$ and the reverse inclusion is obvious since $V_J\in\T$ and $\T$ is closed under coproducts and quotients. On the other hand,  $V_J^{\perp_1}$ is closed under direct limits, since $V_J$ is finitely presented and $R$ right coherent. We then get that $\T=\Gen(V_J)\subseteq\overline{\Gen}(V)\cap V_J^{\perp_1}$.
We need just to verify the converse inclusion and the proof will be finished (see Remark \ref{rem.usual definition of quasi-tilting}). Now,  by Lemma \ref{lem.extension-in-reflective-subcats}(1), we have that $\overline{\Gen}(V)\cap V_J^{\perp_1}=\ker(\Ext_\B^1(V_J,-))$, where $\B:=\Sub(\T)=\overline{\Gen}(V_J)$. 
Furthermore, we have that $\B'=\B\cap\fpmod R$. In fact, the inclusion $\B'\subseteq\B\cap\fpmod R$ is clear, while the converse follows from the fact that $\T=\varinjlim\T_0$. On the other hand, 
since $u$ is an $\add(V_J)$-preenvelope, we immediately get that $\mathfrak{a}:=\ker (u)$ is the annihilator of $V_J$, i.e. 
\[
\mathfrak{a}=\{r\in R:V_J\cdot r=0\}. 
\]
In particular, ${\frak a}$ is a  two-sided ideal of $R$ such that $R/\mathfrak{a}\cong\Im(u)\in\B$. This implies that $\mod R/\mathfrak{a}\subseteq\B$, with the obvious identification. But the converse is also true, since $\mathfrak{a}$ annihilates $\T=\Gen(V)$, and hence $B\cdot\mathfrak{a}=0$, for all $B\in\B$. Hence, $\B=\mod R/\mathfrak{a}$. On the other hand, since $\mathfrak{a}$ is the kernel of a morphism in $\fpmod R$, we have that $R/\mathfrak{a}\in\fpmod R$, and so $\fpmod R/\mathfrak{a}\subseteq\fpmod R$ and $\B'=\B\cap\fpmod R=\mod R/\mathfrak{a}\cap\fpmod R=\fpmod R/\mathfrak{a}$. Moreover, we have that $R/\mathfrak{a}$ is a right coherent ring (see \cite[Coro.\,3.6]{Swan}) and, by the the last paragraph, we also have that 
\[
\overline{\Gen}(V)\cap V_J^{\perp_1}=\ker(\Ext_\B^1(V_J,-))=\ker(\Ext_{\mod R/\mathfrak{a}}^1(V_J,-)).
\] 
We will be done if we prove that $V_J$ is  tilting in $\mod R/\mathfrak{a}$. We already know that $V_J$ is a tilting object of $\B'=\fpmod R/\mathfrak{a}$, which implies that conditions (CT.1) and (CT.2) of Definition \ref{def.classical tilting module} hold for $V_J$ in $\mod {R/\mathfrak a}$. But,  by the choice of the set $J$ and the definition of $\mathfrak{a}$, we have a short exact sequence $0\rightarrow R/\mathfrak{a}\to V_0\to V_1\rightarrow 0$, so that also condition (CT.3) holds, thus ending the proof.  
\end{proof}

The results of this subsection naturally  raise the question of whether the torsion pair in $\mod R$ induced by a finitely presented finendo quasi-tilting module, in particular, by a classical tilting R-module,  restricts to $\fpmod R$. The following results will clarify this question. Let us first recall that an object $V^\bullet$ of the derived category $\Der(R):=\Der(\mod R)$ is a \textbf{classical tilting complex} when it satisfies the following three conditions:  
\begin{enumerate}[\rm (CTC.1)]
\item  the functor $\Der(R)(V^\bullet ,-)\colon\Der(R)\to\Ab$ preserves coproducts;
\item  $\Der(R)(V^\bullet ,V^\bullet [n])=0$ for all integers $n\neq 0$;  
\item  if $X^\bullet\in\Der(R)$ is a complex such that $\Der(R)(V^\bullet,X^\bullet [n])=0$, for all $n\in\mathbb{Z}$, then $X^\bullet =0$. 
\end{enumerate}
We refer the reader to \cite{BBD} for the definition of $t$-structure in a triangulated category and what it means for it to restrict to a triangulated subcategory.

\begin{proposition} \label{prop.classical tilting complex}
Let $R$ be a right coherent ring, $V^\bullet\in\Der(R)$ a classical tilting complex and take its endomorphism ring $S:=\End_{\Der(R)}(V^\bullet )$. The following assertions are equivalent:
\begin{enumerate}[\rm (1)]
\item the $t$-structure $((V^\bullet )^{\perp_{ >0}},(V^\bullet )^{\perp_{ <0}})$ generated by $V^\bullet$ in $\Der(R)$ restricts to $\Der^b(\fpmod R)$;
\item $S$ is a right coherent ring.
\end{enumerate}
\end{proposition}
\begin{proof}
Up to replacing $V^\bullet$ by a quasi-isomorphic complex, we can assume that $V^\bullet$ is a complex of $S$-$R$-bimodules, and then we have an induced triangulated equivalence  
\[
F:=\mathbb{R}\Hom_R(V^\bullet ,-)\colon\Der(R)\tilde{\longrightarrow}\Der(S)
\] 
and,  by adapting the proof of \cite[Prop.\,8.1]{RicMorita}, we know that it restricts to an equivalence 
\[
F_{\restriction{\K^{-,b}(\proj R)}}\colon \K^{-,b}(\proj R)\tilde{\longrightarrow}\K^{-,b}(\proj S).
\] 
Furthermore,  $F$ takes the $t$-structure  $((V^\bullet )^{\perp_{ >0}},(V^\bullet )^{\perp_{ <0}})$ to the canonical $t$-structure $(\Der^{\leq 0}(S),\Der^{\geq 0}(S))$ in $\Der(S)$. Hence, it takes the pair $\tau_R:=(\K^{-,b}(\proj R)\cap (V^\bullet )^{\perp_{>0}},\K^{-,b}(\proj R)\cap (V^\bullet )^{\perp_{<0}})$ of  subcategories of $\K^{-,b}(\proj R)$ to the pair $\tau_S:=(\K^{-,b}(\proj S)\cap\Der^{\leq 0}(S),\K^{-,b}(\proj S)\cap\Der^{\leq 0}(S))$ of subcategories  $\K^{-,b}(\proj S)$. With this in mind, we can now proceed with the proof.

\smallskip\noindent
(1)$\Rightarrow$(2).  Since $R$ is right coherent, there is a triangulated equivalence $\K^{-,b}(\proj R)\cong\Der^b(\fpmod R)$. Assertion (1) is then equivalent to say that $\tau_R$ is a $t$-structure in $\K^{-,b}(\proj R)$, which is then  equivalent to say that  $\tau_S$ is a $t$-structure in $\K^{-,b}(\proj S)$. Take now a morphism $f\colon S^n\to S$ in $\mod S$ and consider it as a complex 
\[
Y_f^\bullet:\quad\cdots \longrightarrow 0\longrightarrow S^n\stackrel{f}{\longrightarrow} S\longrightarrow 0 \longrightarrow \cdots,
\] 
concentrated in degrees $0$ and $1$. The associated $\tau_S$-truncation triangle, which coincides with  the truncation triangle with respect to the canonical $t$-structure $(\Der^{\leq 0}(S),\Der^{\geq 0}(S))$, takes the form 
\[
(\ker (f))[0]\longrightarrow Y^\bullet_f\longrightarrow (\coker (f))[-1]\longrightarrow(\ker (f))[1].
\] 
It then follows that $(\ker (f))[0]\in\K^{-,b}(\proj S)$, which implies that $\ker (f)$ admits a projective resolution in $\mod S$ whose terms are finitely generated (projective) $S$-modules. In particular, $\ker (f)$ is finitely presented, and so $S$ is right coherent.

\smallskip\noindent
(2)$\Rightarrow$(1). Since $R$ and $S$ are right coherent, we have equivalences of triangulated categories 
\[
\K^{-,b}(\proj R)\cong\Der^b(\fpmod R)\quad\text{and}\quad\K^{-,b}(\proj S)\cong\Der^b(\fpmod S).
\] 
Therefore, $F$   restricts to an equivalence $\Der^b(\fpmod R)\cong\Der^b(\fpmod S)$. Since the canonical $t$-structure of $\Der(S)$ restricts to $\Der^b(\fpmod S)$, we conclude that  $((V^\bullet )^{\perp_{ >0}},(V^\bullet )^{\perp_{ <0}})$ restricts to $\Der^b(\fpmod R)$.
\end{proof}

We are now in position to determine when the torsion pair associated with a quasi-tilting finitely presented right $R$-module $V$ restricts to $\fpmod R$.

\begin{corollary} \label{cor.restriction of qt torsion pair}
Let $R$ be a right coherent ring and $V$ a finitely presented quasi-tilting right $R$-module. Consider the following assertions:
\begin{enumerate}[\rm (1)]
\item the associated torsion pair $\mathbf{t}_V:=(\Gen(V),V^\perp )$ in $\mod R$ restricts to $\fpmod R$;
\item $\add (V)$ is a precovering subcategory of $\fpmod R$;
\item $\hom_R(V,X)$ is a finitely presented right $\End_R(V_R)$-module, for all $X\in\fpmod R$;
\item $\End_R(V_R)$ is a right coherent ring.
\end{enumerate}
The implications ``(1)$\Leftrightarrow$(2)$\Leftrightarrow$(3)'' hold true and  when, in addition, $V$ is  finendo, the implication ``(3)$\Rightarrow$(4)'' also holds. When $V$ is a (classical) tilting $R$-module, then all assertions are equivalent.
\end{corollary}
\begin{proof}
(1)$\Rightarrow$(2). Let $\T:=\Gen(V)=\Pres (V)$. Then, as $\T_0:=\T\cap\fpmod R=\pres(V)$ (see Lemma~\ref{lem.fp quasitilted restricting}), we can conclude by  the argument in the proof of implication ``(1)$\Rightarrow$(2)'' of Proposition \ref{prop.Ext-universal tilting}, with ${J}$ a finite set in this case.

\smallskip\noindent
(2)$\Rightarrow$(1). The torsion radical in $\mod R$ associated with $\mathbf{t}_V$ is the trace of $V$, i.e., $t(M)=\mathrm{tr}_V(M)$ for all $M\in\mod R$. If now $X\in\fpmod R$, then any morphism $V\to X$ factors through any fixed $\add (V)$-precover $p\colon V^n\to X$. It  follows that $t(X)=\Im (p)$, which is a finitely presented module since it is the image of a morphism in $\fpmod R$. Therefore, $\mathbf{t}_V$ restricts to $\fpmod R$.

\smallskip\noindent
(2)$\Leftrightarrow$(3) is folklore.

\smallskip\noindent
(1)$\Leftrightarrow$(4) {\em if $V$ is a (classical) tilting $R$-module}. This is a consequence of Proposition \ref{prop.classical tilting complex}.  Indeed, each classical tilting $R$-module $V$ is a classical tilting complex, when identified with the stalk complex $V[0]$ in $\Der(R)$. Moreover, we have an obvious ring isomorphism $\End (V_R)\cong\End_{\Der(R)}(V[0])$.  In that case the $t$-structure $(V[0]^{{\perp}_{>0}}, V[0]^{{\perp}_{<0}})$ in $\Der(R)$ is precisely the so-called Happel-Reiten-Smal\o\ $t$-structure associated with the torsion pair $\mathbf{t}_V=(\Gen(V),V^\perp )$ of $\mod R$ generated by $V$ (see the proof of \cite[Prop.\,5.2]{CGM}).  Then, by  \cite[Prop.\,5.1]{Saorin:Locally}, we know that  $(V[0]^{{\perp}_{>0}}, V[0]^{{\perp}_{<0}})$ restricts to $\Der^b(\fpmod R)$ if, and only if, $\mathbf{t}_V$ restricts to $\fpmod R$.

\smallskip\noindent
(1--3)$\Rightarrow$(4) {\em assuming that $V$ is finendo}. In this case, what happens is that $\B:=\Sub(\T)$ is a reflective subcategory of $\mod R$ (see Lemma \ref{example_4_lemma} and Remark \ref{rem.Colpi-Menini}). In particular, if 
\[
\mathfrak{a}=\mathrm{ann}(V_R):=\{r\in R: V\cdot r=0\},
\] 
we have that $R/\mathfrak{a}\in\B$ since we have a monomorphism $R/\mathfrak{a}\hookrightarrow V^V$ and $V^V\in\Gen(V)=\T$. It immediately follows that $\B=\mod R/\mathfrak{a}$.   On the other hand, if $\{v_1,\dots,v_t\}$ is a finite set of generators of $V$ as a left $\End(V_R)$-module, then the morphism $\mu\colon R\to V^t$, defined by $r\mapsto (v_1r,\dots,v_tr)$ is an $\add(V)$-preenvelope of $R_R$. This  implies that $\mathfrak{a}=\ker (\mu)$, so this is a finitely generated right ideal of $R$. It follows that $R/\mathfrak{a}$ is a right coherent ring (see \cite[Coro.\,3.6]{Swan}) and that $\fpmod R/\mathfrak{a}\subset\fpmod R$, with the obvious abuse of notation. Recall that, by Proposition \ref{prop.properties of quasi-tilting}, $V$ is a tilting object in $\B$, and hence $V$ is  a classical tilting $R/\mathfrak{a}$-module whose associated torsion pair $(\T,V^\perp\cap\mod R/\mathfrak{a})$ in $\mod R/\mathfrak{a}$ restricts to $\fpmod R/\mathfrak{a}$. By the equivalence ``(1)$\Leftrightarrow$(4)'' in the tilting case, we conclude that $\End_{R/\mathfrak{a}}(V)=\End_R (V_R)$ is right coherent.
\end{proof}

\begin{remark} \label{rem. Rickard}
Since any ring which is derived equivalent to $R$ is isomorphic to the endomorphism ring of a classical tilting complex in $\Der(R)$, Proposition \ref{prop.classical tilting complex} naturally leads to the question of whether ``right coherence'' is a derived invariant property for rings. That is, if the equivalent conditions of the proposition always hold.  Jeremy Rickard, to whom we are grateful,   has shown to us a counterexample.  However, as far as we know,  the following restricted version of this question is still open.
\end{remark}

\begin{question} \label{ques.tilting invariance of coherence}
Is ``right coherence'' a property of rings which is invariant under classical ($1$)-tilting equivalences?. Equivalently (see Proposition \ref{prop.classical tilting complex}), given a right coherent ring $R$ and a classical ($1$)-tilting right $R$-module $V$, does the torsion pair $(\Gen(V),V^\perp )$ in $\mod R$  restrict always to $\fpmod R$?.
\end{question}

\begin{remark}
\textcolor{black}{The answer to the above question is negative when $\mod R$ is replaced by a locally coherent Grothendieck category (see Example \ref{tilt_not_torsion_ex})}.
%In \cite{parra2019locally} (see the discussion right after Proposition 8.19 in [Op.Cit.]), we have shown that the answer to the above question is negative when $\mod R$ is replaced by a locally coherent Grothendieck category.
\end{remark}

\subsection{{Construction of nontrivial classical tilting modules over coherent rings}}

  We want to show that, given any ring $A$, there is a canonical way of giving plenty of examples of non-trivial (i.e., not progenerators) classical tilting modules over the (upper) triangular matrix ring 
  \[
  T_n(A):=
  \left[\begin{array}{cccc}
  A & A &\cdots & A \\ 
  0 & A & \cdots& A \\
\vdots & \ddots &  \ddots & \vdots \\ 
0 & \cdots&0 & A 
\end{array}\right]
\] 
that, in case  $T_n(A)$ (or, equivalently, $A$) is right coherent (see Corollary \ref{cor.coherence-triangmatrixring}), always give rise to a torsion pair that restricts to $\lfpmod{T_n(A)}$. 

\begin{lemma} \label{lem.coherence-triangring}
Let $A$ and $B$ be rings and ${}_AM_B$ an $A\text{-}B\text{-}$bimodule that is finitely generated as a right $B$-module. Consider the ring $R:=\left[\begin{smallarray}{cc} A& M\\ 0 & B\end{smallarray}\right]\cong\left[\begin{smallarray}{cc} B& 0\\ M & A\end{smallarray}\right]$. The following assertions are equivalent:
\begin{enumerate}[\rm (1)]
\item $R$ is right coherent;
\item $A$ and $B$ are right coherent and $M$ is a coherent right $B$-module, i.e., any finitely generated $B$-submodule of $M$ is finitely presented.
\end{enumerate}
\end{lemma}
\begin{proof}
The two-sided ideal $I:=\left[\begin{smallarray}{cc} 0 & M\\ 0 & 0 \end{smallarray}\right]$ is finitely generated as a right ideal since $M_B$ is finitely generated. Moreover $R/I$ is isomorphic to $A\times B$, so that $R/I$ is right coherent if and only if so are $A$ and $B$. Now use \cite[Thm.\,2]{Har67}.
\end{proof}

\begin{corollary} \label{cor.coherence-triangmatrixring}
Let $A$ be a ring and $n>0$ an integer. Then, the triangular matrix ring $T_n(A)$ is right (resp., left) coherent if, and only if, so is $A$. 
\end{corollary}
\begin{proof}
The statement about right coherence follows by Lemma \ref{lem.coherence-triangring} and an easy induction.
On the other hand, in the situation of that lemma,   the opposite algebra of $\left[\begin{smallarray}{cc}  A& M\\ 0 & B\end{smallarray}\right]$ is isomorphic $\left[\begin{smallarray}{cc}  B^{\op}& M\\ 0 & A^{\op}\end{smallarray}\right]$, where $M$ is viewed as a $B^{\op}\text{-}A^{\op}\text{-}$bimodule in the obvious way. It  follows that, if $M$ is finitely generated as a left $A$-module, then  $\left[\begin{smallarray}{cc} A& M\\ 0 & B\end{smallarray}\right]$ is left coherent if, and only if, B and $A$ are left coherent and $M$ is a coherent left $A$-module. Applying this dual version of the lemma, again by an easy induction argument, we conclude that $T_n(A)$ is left coherent if, and only if, so is $A$.
\end{proof}

%\begin{remark} \label{rem.matrix algebra epimorphism}
%It is well-known that, when $K$ is a commutative ring,  the inclusion $\iota:T_n(K)\hookrightarrow M_n(K)$ is a ring epimorphism. If now $A$ is any $K$-algebra,  then we have obvious algebra isomorphisms $A\otimes_K T_n(K)\cong T_n(A)$ and  $A\otimes_K M_n(K)\cong M_n(A)$ which identify $1_A\otimes\iota :A\otimes_K T_n(K)\rightarrow A\otimes_K M_n(K)$ with the inclusion $T_n(A)\hookrightarrow M_n(A)$. But it is very easy to prove that the tensor product of algebra epimorphisms is an algebra epimorphism. Therefore the inclusion $T_n(A)\hookrightarrow M_n(A)$ is also an algebra epimorphism. 
%\end{remark}

Our generic example is the following:

\begin{example} \label{ejem.classical-tilting-over-coherentring}
Let $A$ by a ring, $n>0$ an integer and consider the triangular matrix ring $T_n(A)$ and its ideal $I$ of strictly upper triangular matrices. If $e_1$ is the idempotent matrix having 1 in its \mbox{$(1,1)$-entry} and zero elsewhere, then $V=\bigoplus_{k=1}^{n}(e_1T_n(A)/e_1I^k)$ is a classical tilting right \mbox{$T_n(A)$-module} that is not a progenerator. 
 In case $A$ is right coherent, the  associated tilting torsion pair restricts to  $\fpmod T_n(A)$. 
\end{example}
\begin{proof}
The inclusion $\iota\colon T_n(A)\hookrightarrow M_n(A)$ is a ring epimorphism, where $M_n(A)$ denotes the full matrix ring of $A$ (see \cite[Prop.\,2.10]{silver}. Denote by $e_i$ ($i=1,\dots,n$) the idempotent matrix having $1$ in its $(i,i)$-entry and zero elsewhere. We have that $M_n(A)\cong (e_1T_n(A))^{(n)}$ as  right $T_n(A)$-modules and  $M_n(A)\cong (T_n(A)e_n)^{(n)}$ as left $T_n(A)$-modules, so that $M_n(A)$ is projective on both sides as a module over $T_n(A)$. It follows from \cite[Thm.\,3.5]{hugel2011sanchez} that $V':=M_n(A)\oplus (M_n(A)/T_n(A))$ is classical tilting, both as a right and as left $T_n(A)$-module.  Furthermore, we have an isomorphism $M_n(A)/T_n(A)\cong\bigoplus_{k=1}^{n-1}(e_1T_n(A)/e_1I^k)$ in $\mod {T_n(A)}$. Therefore, since $I^n=0$, we have that $V:=\bigoplus_{k=1}^{n}(e_1T_n(A)/e_1I^k)$ is a classical (1-)tilting right $T_n(A)$-module which is equivalent to $V'$. Moreover, when $k< n$ (i.e. $I^k\neq 0$) the quotient $e_1T_n(A)/e_1I^k$ cannot be projective in $\mod T_n(A)$ since this would imply that $e_1I^k$ is a direct summand of $e_1T_n(A)$ that, due to the nilpotentcy of $I$, is contained in the radical $\text{rad}(T_n(A))$. This is absurd. Then $V$ is not a progenerator of $\mod T_n(A)$. 

When $A$ is right coherent, we know by Corollary \ref{cor.coherence-triangmatrixring} that $T_n(A)$ is right coherent. Note now that then any morphism of right $T_n(A)$-modules $(e_1T_n(A)/e_1I^k)\rightarrow (e_1T_n(A)/e_1I^l)$ is induced by left multiplication by a matrix   $X=(x_{i,j})_{1\leq i,j\leq n}\in e_1T_n(A)e_1$ (that is, $x_{i,j}=0$ whenever $i+j>2$).  Given that $X\cdot I^k\not\subseteq I^{l}$, whenever $X\neq 0$ (i.e., $x_{1,1}\neq 0$) and $k<l$, and that $e_1 T_n(A)e_1\cong A$, we readily see that $\End(V_{T_n(A)})$ is isomorphic to the lower triangular matrix ring 
\[
T'_n(A):=\left[\begin{array}{cccc}
A & 0  & \cdots& 0\\ 
\vdots & \ddots   & \ddots& \vdots \\  
A & \cdots &A & 0\\
A & \cdots &A & A \end{array}\right].
\] 
It  is easy to see that  this matrix ring is isomorphic to $T_n(A)$, and hence it is a right coherent ring. %Indeed, considering the permutation $\sigma =(1,n)(2,n-1)...$, where $(i,j)$ denotes here the trasposition of indices $i$ and $j$, we get a map $f:T'_n(A)\longrightarrow T_n(A)$ that takes the matrix $a'=(a'_{ij})$ to the matrix $a=(a_{ij})$, where $a_{ij}=a'_{\sigma (i)\sigma (j)}$, for all $i,j\in\{1,2,...,n\}$. It is routine to check that $f$ is a ring isomorphism.  
Therefore,  $\End(V_{T_n(A)})$ is a right coherent ring and Corollary \ref{cor.restriction of qt torsion pair}  applies. 
\end{proof} 
}

\subsection{Quasi-tilting preenvelopes in $\fpmod R$}

The goal of this last subsection is to identify the quasi-tilting objects $V$ of $\fpmod R$ that give rise to semi-special preenveloping torsion classes, equivalently (see Proposition \ref{prop.lorthogonal-generates-torsion}) such that $\subGen_{\fpmod R}(V)$ is a reflective subcategory of $\fpmod R$. We start with the following result, which is a trivial translation of Proposition \ref{prop.quasi-tilting in fpmodR}, bearing in mind that ``classical'' and ``finitely presented'' are synonymous terms when referred to a tilting right $R$-module.

\begin{corollary} \label{cor.tilting-in-fpmod-versus-classicaltilting}
Let $R$ be a right coherent ring and $V\in \fpmod R$. The following are equivalent:
\begin{enumerate}[\rm (1)]
\item $V$ is a tilting object of $\fpmod R$;
\item there is a (possibly infinite) set $I$ such that $V_I:=\mathrlap{\scalebox{.5}{\text{\hspace{8.5pt}$\mathrm{fp}$}}}\coprod_I V$ exists,  it is a classical tilting module and the torsion pair $\t_{V_I}:=(\Gen(V_I),V_I^\perp )$ in $\mod R$ restricts to $\fpmod R$. 
\end{enumerate}
In this case, the restricted torsion pair $(\Gen(V_I)\cap\fpmod R,V_I^\perp\cap\fpmod R)$ is the torsion pair in $\fpmod R$ associated with $V$. 
\end{corollary}

We are now ready to parametrize the (semi)special preenveloping torsion classes in $\fpmod R$.

\begin{proposition} \label{prop.bijection for right coherent rings}
Let $R$ be a right coherent ring. Then, the assignment $[V]\mapsto\pres(V)$ gives a one-to-one correspondence between:
\begin{enumerate}[\rm (1)]
 \item equivalence classes of finitely presented finendo quasi-tilting (resp., classical tilting) right $R$-modules $V$ whose associated torsion pair in $\mod R$ restricts to $\fpmod R$;
\item semi-special (resp., special) preenveloping torsion classes in $\fpmod R$.
\end{enumerate}
When $R$ is also a Krull-Schmidt ring, the semi-special (resp., special) preenveloping torsion classes in $\fpmod R$ are exactly the (pre)enveloping (resp., cogenerating (pre)enveloping) ones.
\end{proposition}
\begin{proof}
By Proposition \ref{prop.lorthogonal-generates-torsion} and Theorem \ref{thm.specialpreenveloping-versus-cotorsionpair}, the assignment $[V]\mapsto\T_V=\Gen_{\fpmod R}(V)$ establishes a bijection between the equivalence classes of quasi-tilting objects $V$ of $\fpmod R$ such that $\subGen_{\fpmod R}(V)$ is reflective and the semi-special preenveloping torsion classes in $\fpmod R$. By Proposition \ref{prop.quasi-tilting in fpmodR}, any such $[V]$ is represented by a finitely presented finendo quasi-tilting $R$-module $V$ whose associated torsion pair in $\mod R$ restricts to $\fpmod R$. Moreover, by Lemma \ref{lem.fp quasitilted restricting}, in such case $\T_V:=\pres(V)$. It is obvious that, for quasi-tilting $R$-modules $V$ and $V'$ as indicated, the equality $\pres(V)=\pres (V')$ is tantamount to their equivalence. Therefore, the general bijection is clear.

The restricted version now follows from the fact that, by \cite[Thm.\,1.5]{miyashita-tilt}, any classical (=finitely presented) tilting $R$-module is finendo.  On the other hand, when $R$ is also Krull-Schmidt, we know that ``enveloping'' (resp.,  ``cogenerating enveloping'')  and  ``(semi-special) preenveloping''  (resp.,  ``special preenveloping'') are synonymous terms in $\fpmod R$.
\end{proof}

Recall that a {\bf Noether} (resp., {\bf Artin}) {\bf algebra} is an algebra $A$ over some commutative Noetherian (resp., Artinian) ring $K$ such that $A$ is finitely generated (=finitely presented) as a $K$-module. 

\begin{corollary} \label{cor.bijection-for-coherent-fg-over-center}
Let $R$ be a ring that is finitely presented over a commutative coherent ring $K$ (e.g., a Noether or an Artin algebra). Then, the assignment  $[V]\mapsto\pres(V)$ gives a one-to-one correspondencence between the equivalence classes of finitely presented quasi-tilting (resp., tilting) right $R$-modules and the semi-special (resp., special) preenveloping torsion classes in $\text{mod-}R$. When $R$ is in addition Krull-Schmidt (see Example~\ref{Examples_KS_rings}(2,3)), the latter torsion classes are  precisely the enveloping (resp., cogenerating and enveloping) ones. 
\end{corollary}
\begin{proof}
In this situation the category $\fpmod R$ is $\Hom$-finite over $K$, which, by Proposition \ref{prop.Hom-finite-are-coherent} and Corollary \ref{cor.restriction of qt torsion pair},  implies that $R$ is left and right coherent, and that any torsion pair in $\mod R$ given by a classical tilting right $R$-module restricts to $\fpmod R$.  On the other hand, 
 each $M\in\fpmod R$ is a finitely presented $K$-module and any finite set of generators of $M$  as a $K$-modules also generates it as a left $\End(M_A)$-modules. In particular, all finitely presented right $R$-modules are finendo. The result then follows from  Proposition \ref{prop.bijection for right coherent rings} since the Krull-Schmidt part is obvious.
\end{proof}

\begin{remark}
Note that, if $R$ is right Noetherian in Proposition \ref{prop.bijection for right coherent rings}, then all torsion pairs in $\mod R$ restrict to $\fpmod R$.
%} %gives a one-to-one correspondence $[V]\mapsto\pres(V)=\gen(V)$ between equivalence classes of finitely generated finendo quasi-tilting $R$-modules and the semi-special preenveloping torsion classes in $\fpmod R$. In the situation of  Example~\ref{Examples_KS_rings}(2, 3), we know that the latter are all (pre)enveloping torsion classes. When, instead, $R$ is finitely generated over a commutative noetherian ring $K$, all finitely generated $R$-module are finendo. 
%In Example \ref{Examples_KS_rings} the rings in (2) and (3) are coherent, as so are Artin algebras, and hence the last statement of  Proposition \ref{prop.fpmod-coherent} applies.  For Artin algebras, for the rings in (3) and for those in (2) whose center  is Noetherian,  all finitely generated (=finitely presented) modules are finendo. Therefore, Proposition \ref{prop.fpmod-coherent} gives a one-to-one correspondence between add-equivalence classes of quasi-tilting  finitely presented $R$-modules and (pre)enveloping torsion classes in $\fpmod R$. 
\end{remark}

As a consequence of Proposition \ref{prop.bijection for right coherent rings}, we can give a new proof of a nice result, which is already known for finite dimensional algebras over a field (see \cite[Proposition 3.15]{hugel2015silting}) and, furthermore, we extend it to Artin algebras. The reader is referred to \cite[Def.\,0.1]{AIR14} for the definition of {\bf support $\tau$-tilting $\Lambda$-module} and to \cite{hugel2015silting} or \cite{hugel2018abundance} for the definition of {\bf silting module}.

\begin{corollary} \label{cor.quasi-tilting = support tau-tilting}
Let $\Lambda$ be an Artin algebra and let $V$ be a finitely generated (=finitely presented) $\Lambda$-module. The following conditions are equivalent:
\begin{enumerate}[\rm (1)]
\item $V$ is quasi-tilting;
\item $V$ is silting;
\item $V$ is support $\tau$-tilting.
\end{enumerate}
\end{corollary}
\begin{proof}
(2)$\Leftrightarrow$(3) is  \cite[Ex.\,2.4(2)]{hugel2018abundance}.

\smallskip\noindent
(1)$\Leftrightarrow$(3). Note that any finitely generated $\Lambda$-module is finendo and $\fpmod \Lambda$ is a Krull-Schmidt category. On the other hand, a torsion class 
 in $\fpmod \Lambda$ is (pre)enveloping if, and only if, it is functorially finite in the terminology of \cite{AIR14}. Therefore, by Proposition \ref{prop.bijection for right coherent rings}, the assignment $[W]\mapsto\pres(W)=\gen(W)$ gives  a bijection between $\add$-equivalence classes of finitely generated quasi-tilting $\Lambda$-modules and functorially finite torsion classes in $\fpmod \Lambda$. Its inverse takes $\T$ to $[W]$, where, by Proposition \ref{prop.properties of quasi-tilting},  $W$ may be chosen to be the (finite) direct sum of one isomorphic copy of each indecomposable $\Lambda$-module in $_{}^{\perp_1}\T\cap\T$. This particular $W$ is denoted by $\mathcal{P}(\T)$ in \cite{AIR14} and, by \cite[Thm.\,2.7]{AIR14},  the assignment $\T\mapsto\mathcal{P}(\T)$ gives a bijection between functorially finite torsion classes in $\fpmod \Lambda$ and basic support $\tau$-tilting  $\Lambda$-modules. Therefore, $V$ is quasi-tilting if, and only if, it is $\add$-equivalent to a basic support $\tau$-tilting $\Lambda$-module. 
\end{proof}

\begin{remark}
The bijection in the tilting case of  Corollary \ref{cor.bijection-for-coherent-fg-over-center} for Artin algebras is part of the main theorem of \cite{smalo1984torsion}. Note that  Smal\o's   bijection also with covering  torsion-free classes in $\fpmod R$ in that case is valid because of the duality between $\fpmod R$ and $R\text{-}\mathrm{mod}$,  that is no longer available for a general $R$ as in  Corollary \ref{cor.bijection-for-coherent-fg-over-center}.
\end{remark}

Let us conclude with the following example:

 \begin{example} \label{rem.BHPST}
 There exist  right coherent rings $R$  and torsion classes $\T_0$ in  $\fpmod R$ such that $\T:=\varinjlim\T_0$ is a special preenveloping torsion class in $\mod R$ while $\T_0$ is not  so in $\fpmod R$  (see, e.g., \cite[Lem.\,4.4]{Craw} to see that $\T$ is a torsion class in $\mod R$). Therefore $\T_0$ is a cogenerating preenveloping torsion class that is not special preenveloping in $\fpmod R$.
Indeed, in \cite{bazzoni2017pure} the authors give an example of a pure-projective tilting module $T$ over a Noetherian ring that is not $\Add$-equivalent to a finitely presented tilting module. In such case $\T=\Gen(T)$ is special preenveloping in $\mod R$ and  $\T_0=\T\cap\fpmod R$ has the property that $\T=\varinjlim\T_0$.  However, according to Proposition~\ref{prop.bijection for right coherent rings},  $\T_0$ is not special preenveloping in $\fpmod R$.
 \end{example}

\newcommand{\etalchar}[1]{$^{#1}$}


\begin{thebibliography}{AHMV15}

\bibitem[AIR14]{AIR14}
Takahide Adachi, Osamu Iyama, and Idun Reiten. 
\newblock $\tau$-tilting theory.
\newblock {\em Compositio Mathematica}, 150(3): 415--452, 2014.

\bibitem[AF92]{AF92}
Frank W. Anderson, Kent R. Fuller,
\newblock Rings and categories of modules, 2nd edition.
\newblock Springer-Verlag, 1992.

\bibitem[AH18]{hugel2018abundance} 
Lidia Angeleri~H{\"u}gel. 
\newblock On the abundance of silting modules.
\newblock {\em ``Surveys in Representation Theory of Algebra'', Contemp. Math}, 716: 1--23, 2018.


\bibitem[AC01]{hugel2001infinitely}
Lidia Angeleri~H{\"u}gel and Fl{\'a}vio~Ulhoa Coelho.
\newblock Infinitely generated tilting modules of finite projective dimension.
\newblock {\em Forum Mathematicum}, 13:239--250, 2001.

\bibitem[AHMV15]{hugel2015silting}
Lidia Angeleri~H{\"u}gel, Frederik Marks, and Jorge Vit{\'o}ria.
\newblock Silting modules.
\newblock {\em International Mathematics Research Notices}, 2016(4):1251--1284,
  2015.
  
\bibitem[AS11]{hugel2011sanchez}
Lidia Angeleri~H{\"u}gel and Javier S\'anchez.
\newblock Tilting modules arising from ring epimorphisms.
\newblock {\em Algeb. Represent. Theory}, 14:217--246, 2011.
  2015.
 

\bibitem[AHTT01]{hugel2001tilting}
Lidia Angeleri~H{\"u}gel, Alberto Tonolo, and Jan Trlifaj.
\newblock Tilting preenvelopes and cotilting precovers.
\newblock {\em Algebras and Representation Theory}, 4(2):155--170, 2001.

\bibitem[AR91]{auslander1991applications}
Maurice Auslander and Idun Reiten.
\newblock Applications of contravariantly finite subcategories.
\newblock {\em Advances in Mathematics}, 86(1):111--152, 1991.

\bibitem[AS80]{auslander1980preprojective}
Maurice Auslander and Sverre~O Smal{\o}.
\newblock Preprojective modules over artin algebras.
\newblock {\em Journal of algebra}, 66(1):61--122, 1980.

\bibitem[BHP{\etalchar{+}}17]{bazzoni2017pure}
Silvana Bazzoni, Ivo Herzog, Pavel P{\v{r}}{\'\i}hoda, Jan {\v{S}}aroch, and
  Jan Trlifaj.
\newblock Pure projective tilting modules.
\newblock {\em Doc. Math}, 25:401-424, 2020.    
%\newblock 2017.


\bibitem[BBD82]{BBD}
A.~A. Beilinson, J.~Bernstein, and P.~Deligne.
\newblock Faisceaux pervers.
\newblock {\em Ast{\'e}risque}, 100:5-171, 1982.

\bibitem[Ber06]{Bergman}
George~M. Bergman.
\newblock Two statements about infinite products that are not quite true. In
  `Groups, rings and algebras' (in honor of Donald S. Passman). William Chin, James Osterburg and Declan Quin (edts).
\newblock {\em Contemp. Maths}, 420:35-58, 2006.


\bibitem[BB80]{BB80}
Sheila Brenner and M.~C.~R. Butler.
\newblock Generalizations of the bernstein-gelfand-ponomarev reflection
  functors.
\newblock In Vlastimil Dlab and Peter Gabriel, editors, {\em Representation
  Theory II}, pages 103--169, Berlin, Heidelberg, 1980. Springer Berlin
  Heidelberg.




\bibitem[Bon81]{Bo81}
Klaus Bongartz.
\newblock Tilted algebras.
\newblock In {\em Representations of algebras}, pages 26--38. Springer, 1981.

%\bibitem[Bon81b]{bongartz1981tilted}
%Klaus Bongartz.
%\newblock Tilted algebras.
%\newblock In {\em Representations of algebras}, pages 26--38. Springer, 1981.

\bibitem[Bor94]{borceux1994handbook1}
Francis Borceux.
\newblock {\em Handbook of Categorical Algebra 1: Basic Category Theory
  (Encyclopedia of Mathematics and Its Applications)}.
\newblock Cambridge University Press, 1994.

\bibitem[B{\v{Z}}{\etalchar{+}}18]{breaz2018torsion}
Simion Breaz, Jan {\v{Z}}emli{\v{c}}ka.
\newblock Torsion classes generated by silting modules.
\newblock {\em Arkiv. Math}, 56(1):15--32, 2018. 
%\newblock 2018.

\bibitem[CYZ08]{chen2008algebras}
Xiao-Wu Chen, Yu~Ye, and Pu~Zhang.
\newblock Algebras of derived dimension zero.
\newblock {\em Communications in Algebra}, 36(1):1--10, 2008.


\bibitem[CPS86]{CPS86}
Edward Cline, Brian Parshall, and Leonard Scott.
\newblock Derived categories and morita theory.
\newblock {\em Journal of Algebra}, 104(2):397--409, 1986.


\bibitem[CB94]{Craw}
William Crawley-Boevey.
\newblock Locally finitely presented additive categories.
\newblock {\em Comm. Algebra}, 22(5):1641--1674, 1994.

\bibitem[CF90]{CF90}
Robert R. Colby and Kent R. Fuller.
\newblock Tilting, cotilting, and serially tilted rings.
\newblock {\em Communications in Algebra}, 18(5):1585--1615, 1990.

\bibitem[Co99]{colpi}
Ricardo Colpi.
\newblock Tilting in Grothendieck categories.
\newblock {\em Forum Mathematicum}, 11(6):735--759, 1999.


\bibitem[CGM07]{CGM}
Riccardo Colpi, Enrico Gregorio, and Francesca Mantese.
\newblock On the heart of a faithful torsion theory.
\newblock {\em Journal of Algebra}, 307(2):841--863, 2007.

\bibitem[CM93]{colpi1993structure}
Riccardo Colpi and Claudia Menini.
\newblock On the structure of $*$-modules.
\newblock {\em Journal of Algebra}, 158(2):400--419, 1993.


\bibitem[CT95]{colpi1995tilting}
Riccardo Colpi and Jan Trlifaj.
\newblock Tilting modules and tilting torsion theories.
\newblock {\em Journal of Algebra}, 178(2):614--634, 1995.

%[Pavel Coupek- Jan Stovicek: Cotilting sheaves on noetherian schemes. Math. Zeithschr. \textbf{296}, 275-312 (2020)

\bibitem[{CS}20]{coupek-stovicek}
Pavel {\v{C}}oupek and Jan {\v{S}}{\v{t}}ov{\'\i}{\v{c}}ek.
\newblock Cotilting sheaves on Noetherian schemes.
\newblock {\em Math. Z.}, 296:275--312, 2020.

\bibitem[EB06]{el2006covers}
Robert El~Bashir.
\newblock Covers and directed colimits.
\newblock {\em Algeb. Represent. Theory}, 9:423-430, 2006.
%\newblock 2006.

\bibitem[Eno81]{enochs1981injective}
Edgar~E Enochs.
\newblock Injective and flat covers, envelopes and resolvents.
\newblock {\em Israel journal of mathematics}, 39(3):189--209, 1981.

\bibitem[GT06]{gobel2006approximations}
Rudiger Gobel and Jan Trlifaj.
\newblock {\em Approximations and endomorphism algebras of modules}.
\newblock De Gruyter, 2006.


\bibitem[GT73]{golan1973torsion}
Jonathan~S Golan and Mark~L Teply.
\newblock Torsion-free covers.
\newblock {\em Israel Journal of Mathematics}, 15(3):237--256, 1973.


\bibitem[Hap87]{H87}
Dieter Happel.
\newblock On the derived category of a finite-dimensional algebra.
\newblock {\em Commentarii Mathematici Helvetici}, 62(1):339--389, 1987.

\bibitem[Hap88]{happel1988triangulated}
Dieter Happel.
\newblock {\em Triangulated categories in the representation of finite
  dimensional algebras}, volume 119.
\newblock Cambridge University Press, 1988.



\bibitem[HR82]{happel1982tilted}
Dieter Happel and Claus~Michael Ringel.
\newblock Tilted algebras.
\newblock {\em Transactions of the American Mathematical Society},
  274(2):399--443, 1982.
  
\bibitem[Har67]{Har67}
Morton E. Harris
\newblock Some results on coherent rings II.
\newblock{\em  Glasgow Math. J.}, 8(2): 123--126, 1967.




\bibitem[{Her}97]{zbMATH01032187}
Ivo {Herzog}.
\newblock {The Ziegler spectrum of a locally coherent Grothendieck category.}
\newblock {\em {Proc. Lond. Math. Soc. (3)}}, 74(3):503--558, 1997.

\bibitem[{Ho}02]{hovey}
Mark Hovey.
\newblock Cotorsion pairs, model category structures, and representation theory.
\newblock {\em Math. Z.}, 241:553--592, 2002.

%[M. Hovey: Cotorsion pairs, model category structures, and representation theory. Math. Zeitschr. 241, 553-592 (2002)]

\bibitem[Kel94]{kel94-dg}
Bernhard Keller.
\newblock Deriving dg categories.
\newblock {\em Ann. Sci. {\'E}cole Norm. Sup.(4)}, 27(1):63--102, 1994.




\bibitem[Kra01]{Henning_spectrum_module} Henning Krause. The spectrum of a module category. {\em Vol. 707. American Mathematical Soc.}, 2001.



\bibitem[KS97]{krause1997minimal}
Henning Krause and Manuel Saor{\'\i}n.
\newblock  On minimal approximations of modules. In `Trends in Representation Theory of Finite Dimensional Algebras', Proc. Seattle Conference, 1997, Edward L. Green and Birge Huisgen-Zimmermann (edts). 
\newblock {\em Contemp. Maths}, 229:228--236, 1998.

\bibitem[ML98]{McL}
Saunders Mac~Lane.
\newblock {\em Categories for the working mathematician}, volume~5 of {\em
  Graduate Texts in Mathematics}.
\newblock Springer-Verlag, New York, second edition, 1998.


\bibitem[Miy86]{miyashita-tilt}
Yoichi Miyashita.
\newblock Tilting modules of finite projective dimension.
\newblock {\em Math. Z.}, 193(1):113--146, 1986.



\bibitem[PS16]{parra2016hearts}
Carlos~E Parra and Manuel Saor{\'\i}n.
\newblock On hearts which are module categories.
\newblock {\em Journal of the Mathematical Society of Japan}, 68(4):1421--1460,
  2016.

\bibitem[PS20]{parra2020hrs}
Carlos~E Parra and Manuel Saor{\'\i}n.
\newblock The {HRS} tilting process and Grothendieck hearts of $t$-structures.
\newblock {\em Contemp. Maths.} 769: 209-241, 2021.

\bibitem[PSV21]{parra2019locally}
Carlos Parra, Manuel Saor{\'\i}n, and Simone Virili.
\newblock Locally finitely presented and coherent hearts.
\newblock Preprint arXiv:1908.00649, 2021.

\bibitem[Pop]{popescu}
Nicolae Popescu.
\newblock {\em Abelian categories with applications to rings and modules},
  volume~3.
  \newblock London Math. Soc. Monogr. 3, Academic Press, 1973.


\bibitem[Ric89]{RicMorita}
Jeremy Rickard.
\newblock Morita theory for derived categories.
\newblock {\em Journal of the London Mathematical Society}, 2(39):436--456,
  1989.

\bibitem[Ric91]{rickard1991derived}
Jeremy Rickard.
\newblock Derived equivalences as derived functors.
\newblock {\em Journal of the London Mathematical Society}, 2(1):37--48, 1991.

\bibitem[Row86]{rowen1986finitely}
Louis~H Rowen.
\newblock Finitely presented modules over semiperfect rings.
\newblock {\em Proceedings of the American Mathematical Society}, 97(1):1--7,
  1986.

\bibitem[Row87]{rowen1987finitely}
Louis~H Rowen.
\newblock Finitely presented modules over semiperfect rings satisfying
  {ACC}-$\infty$.
\newblock {\em Journal of Algebra}, 107(1):284--291, 1987.

\bibitem[Sal79]{Salce_cotorsion}
Luigi Salce.
\newblock Cotorsion theories for Abelian groups.
\newblock {\em Symposia Mathematica}, Vol. XXIII (Conf. Abelian Groups and
  their Relationship to the Theory of Modules, INDAM, Rome, 1977), Academic
  Press, London-New York:11--32, 1979.

\bibitem[Sao17]{Saorin:Locally}
Manuel Saor{\'\i}n.
\newblock On locally coherent hearts.
\newblock {\em Pacific Journal of Mathematics}, 287(1):199--221, 2017.


\bibitem[S{\v{S}}11]{SS}
Manuel Saor\'{\i}n and Jan {\v{S}}{\v{t}}ov{\'\i}{\v{c}}ek.
\newblock On exact categories and applications to triangulated adjoints and
  model structures.
\newblock {\em Advances in Mathematics}, 228(2):968--1007, 2011.

\bibitem[S{\v{S}}20]{saorin2020t}
Manuel Saor{\'\i}n and Jan {\v{S}}t'ov{\'\i}{\v{c}}ek.
\newblock $ t $-structures with Grothendieck hearts via functor categories.
\newblock {\em arXiv preprint arXiv:2003.01401v2}, 2020.

%\textcolor{black}{[L. Silver:  J. Alg. \textbf{7}, 44-76 (1967)][Proposition 2.10])}

\bibitem[Sil67]{silver}
L. Silver.
\newblock Noncommutative localizations and applications.
\newblock {\em Journal of Algebra}, 7:44--76, 1967.


\bibitem[Sma84]{smalo1984torsion}
Sverre~O Smal{\o}.
\newblock Torsion theories and tilting modules.
\newblock {\em Bulletin of the London Mathematical Society}, 16(5):518--522,
  1984.

\bibitem[Ste75]{S75}
Bo~Stenstr{\"o}m.
\newblock Rings and modules of quotients.
\newblock In {\em Rings of Quotients}, 195--212. Springer, 1975.


\bibitem[{\v{S}}t'14]{vstovivcek2014derived}
Jan {\v{S}}t'ov{\'\i}{\v{c}}ek.
\newblock Derived equivalences induced by big cotilting modules.
\newblock {\em Advances in Mathematics}, 263:45--87, 2014.



\bibitem[Swa]{Swan}
Richard~G Swan.
\newblock K-theory of coherent rings.
\newblock {\em Unpublished paper}, 1--14.

\bibitem[Swa60]{swan1960induced}
Richard~G Swan.
\newblock Induced representations and projective modules.
\newblock {\em Annals of Mathematics}, 552--578, 1960.

\bibitem[Tep76]{teply1976torsion}
Mark~L Teply.
\newblock Torsion-free covers ii.
\newblock {\em Israel Journal of Mathematics}, 23(2):132--136, 1976.

\bibitem[Ver96]{verdier1996categories}
J-L Verdier.
\newblock Des cat{\'e}gories d{\'e}riv{\'e}es des cat{\'e}gories
  ab{\'e}liennes.
\newblock {\em Ast{\'e}risque}, 1996.

\bibitem[Wei13]{wei2013semi}
Jiaqun Wei.
\newblock Semi-tilting complexes.
\newblock {\em Israel Journal of Mathematics}, 194(2):871--893, 2013.

\bibitem[Wis91]{wisbauer2018foundations}
Robert Wisbauer.
\newblock {\em Foundations of module and ring theory}.
\newblock Gordon and Breach Science Publishers, 1991.

\bibitem[ZW17]{zhang2017cosilting}
Peiyu Zhang and Jiaqun Wei.
\newblock Cosilting complexes and AIR-cotilting modules.
\newblock 2017.

\end{thebibliography}
\end{document}